\documentclass[
BCOR2pt, 
captions=nooneline, 
bibliography=totoc, 
numbers=noenddot, 
parskip=half, 
headings=normal, 
abstracton 
]{scrartcl} 

\usepackage[USenglish]{babel} 
\usepackage{graphicx} 
\usepackage{tikz}
\usepackage{framed}
\usepackage{floatrow}
\usepackage{pdflscape}
\usepackage{remreset} 
\usepackage[nouppercase]{scrpage2} 
\usepackage{amsmath} 
\usepackage{amssymb} 
\usepackage[thmmarks, 
amsmath, 
hyperref 
]{ntheorem} 
\usepackage{bm} 
\usepackage{bbm} 
\usepackage{enumitem} 
\usepackage[hang]{subfigure} 
\usepackage{wrapfig} 
\usepackage{tabularx} 
\usepackage{color}
\usepackage{centernot}
\usepackage{dcolumn} 
\usepackage{booktabs} 
\usepackage{listings} 
\usepackage{psfrag} 
\usepackage[pdftex, 
setpagesize=false, 
pdfborder={0 0 0}, 
pdfpagemode=UseOutlines, 
]{hyperref} 

\makeatletter
\newcommand\mytoday{\number\year-\ifcase\month\or 01\or 02\or 03\or 04\or 05\or 06\or 07\or 08\or 09\or 10\or 11\or 12\fi-\ifcase\day\or 01\or 02\or 03\or 04\or 05\or 06\or 07\or 08\or 09\or 10\or 11\or 12\or 13\or 14\or 15\or 16\or 17\or 18\or 19\or 20\or 21\or 22\or 23\or 24\or 25\or 26\or 27\or 28\or 29\or 30\or 31\fi} 
\makeatother
\setkomafont{sectioning}{\normalcolor\bfseries} 
\clearscrheadfoot 
\automark[section]{subsection} 
\setkomafont{captionlabel}{\bfseries} 
\newcolumntype{d}[2]{D{.}{.}{#1.#2}} 
\newcommand*{\abstractnoindent}{} 
\let\abstractnoindent\abstract
\renewcommand*{\abstract}{\let\quotation\quote\let\endquotation\endquote
\abstractnoindent}
\makeatletter
\renewcommand{\p@enumii}[1]{\theenumi(#1)}
\makeatother


\theoremstyle{break} 
\theoremheaderfont{\bfseries}
\theorembodyfont{}
\theoremseparator{}
\newtheorem{definition}{Definition}[section] 
\newtheorem{lemma}[definition]{Lemma}
\newtheorem{theorem}[definition]{Theorem}
\newtheorem{problem}[definition]{Problem}

\newtheorem{remark}[definition]{Remark}
\newtheorem{example}[definition]{Example}

\theoremstyle{nonumberbreak} 
\theoremsymbol{$\Box$}
\newtheorem{proof}{Proof}

\newcommand*{\IP}{\mathbb{P}}
\newcommand*{\IE}{\mathbb{E}}
\newcommand*{\IR}{\mathbb{R}}
\newcommand*{\F}{\mathcal{F}}

\newcommand*{\IN}{\mathbb{N}}

\newcommand{\norm}[1]{\left\lVert#1\right\rVert}

\usetikzlibrary{snakes}
\usetikzlibrary{arrows,shapes,positioning}
\usetikzlibrary{decorations.markings}
\tikzstyle arrowstyle=[scale=2]
\tikzstyle directed=[postaction={decorate,decoration={markings,
    mark=at position 1 with {\arrow[arrowstyle]{stealth}}}}]

\begin{document}

\renewcommand{\figurename}{Fig.}

\thispagestyle{plain}
	\begin{center}
		{\bfseries\Large The infinite extendibility problem for exchangeable real-valued random vectors}
		\par\bigskip
		
		{\Large Jan-Frederik Mai}\\
		XAIA Investment\\
Sonnenstr.\ 19, 80331 M\"unchen\\
		email: mai@tum.de\\
	\end{center}

We survey known solutions to the infinite extendibility problem for (necessarily exchangeable) probability laws on $\IR^d$, which is:
\par
\begingroup
\leftskip=1cm
\rightskip=1cm
\emph{Can a given random vector $\bm{X}=(X_1,\ldots,X_d)$ be represented in distribution as the first $d$ members of an infinite exchangeable sequence of random variables?}
\par
\endgroup
This is the case if and only if $\bm{X}$ has a stochastic representation that is ``conditionally iid'' according to the seminal de Finetti's Theorem. Of particular interest are cases in which the original motivation behind the model $\bm{X}$ is not one of conditional independence. After an introduction and some general theory, the survey covers the traditional cases when $\bm{X}$ takes values in $\{0,1\}^d$, has a spherical law, a law with $\ell_1$-norm symmetric survival function, or a law with $\ell_{\infty}$-norm symmetric density. The solutions in all these cases constitute analytical characterizations of mixtures of iid sequences drawn from popular, one-parametric probability laws on $\IR$, like the Bernoulli, the normal, the exponential, or the uniform distribution. The survey further covers the less traditional cases when $\bm{X}$ has a Marshall-Olkin distribution, a multivariate wide-sense geometric distribution, a multivariate extreme-value distribution, or is defined as a certain exogenous shock model including the special case when its components are samples from a Dirichlet prior. The solutions in these cases correspond to iid sequences drawn from random distribution functions defined in terms of popular families of non-decreasing stochastic processes, like a L\'evy subordinator, a random walk, a process that is strongly infinitely divisible with respect to time, or an additive process. The survey finishes with a list of potentially interesting open problems. In comparison to former literature on the topic, this survey purposely dispenses with generalizations to the related and larger concept of finite exchangeability or to more general state spaces than $\IR$. Instead, it aims to constitute an up-to-date comprehensive collection of known and compelling solutions of the real-valued extendibility problem, accessible for both applied and theoretical probabilists, presented in a lecture-like fashion. 

\newpage
\tableofcontents

\pagestyle{empty}
\begin{landscape}
\begin{table}[!htbp]
\caption{Summary of main results surveyed. Whereas most appearing notations are introduced in the main body of the article, here $\mathcal{L}[X]$ denotes the Laplace transform of a random variable $X \geq 0$, and $\langle \bm{x},\bm{y} \rangle:=\sum_{k=1}^{d}x_k\,y_k$.}
\begin{center}
\begingroup
\renewcommand{\arraystretch}{1.75}
\begin{tabular}{>{\raggedright}p{41mm}|l|l|l}
\hline
law of $\bm{X}$ & conditionally iid with $H_t=$ & analytically & see\\
\hline
arbitrary in $M_+^1(\IR^d)$ & arbitrary in $M_+^1(\mathfrak{H})$ & $\sup_{g}\Big\{ \frac{|\IE[g(\bm{X})]|}{\sup\limits_{\bm{Y}}|\IE[g(\bm{Y})]|}\Big\} \leq 1$, $g$ bounded, $Y_k$ iid  & Theorem \ref{thm_probgen}\\
arbitrary in $M_+^1(\{0,1\}^d)$ & $(1-M)\,1_{\{t \geq 0\}}+M\,1_{\{t \geq 1\}}$ & $\IP(\bm{X}=\bm{x})=\nabla^{d-||\bm{x}||_1}b_{||\bm{x}||_1}, \,b_k=\IE[M^k]$ & Theorem \ref{thm_definetti01}\\
spherical law & $\Phi\Big( \frac{t}{M}\Big)$ & $\IE[ \exp\{\mathrm{i}\,\langle \bm{x},\bm{X}\rangle\}]=\varphi(||\bm{x}||_2), \,\varphi = \mathcal{L}[M]$ & Theorem \ref{schoenberg_thm}\\
$\ell_{\infty}$-norm symmetric density & $\max\Big\{0,\min\Big\{\frac{t}{M}, 1\Big\}\Big\} $ & $f_{\bm{X}}(\bm{x})=g_d(x_{[d]}),\,g_d(x)=\IE[M^{-d}\,1_{\{M>x\}}]$ & Theorem \ref{thm_gnedin}\\
ranks of Dirichlet prior samples & $DP(c,\{t\}_{t \in [0,1]})$ & $\IP(\bm{X}\leq \bm{x})=x_{[1]}\,\prod_{k=2}^{d}\frac{c\,x_{[k]}+k-1}{c+k-1}$ & Lemma \ref{lemma_DPrs}\\
\hline
law of $\bm{X}$ & conditionally iid with $Z_t=$ & $\IP(\bm{X}>\bm{x})=$ & see\\
\hline
$\ell_{1}$-norm symmetric survival function& $M\,t $ & $\varphi(||\bm{x}||_2), \,\varphi = \mathcal{L}[M]$ & Theorem \ref{thm_AC}\\
Marshall--Olkin exponential law & L\'evy subordinator & $\exp\big\{\sum_{k=1}^{d}\Psi(k)\,\nabla x_{[d-k]}\big\},\,e^{-\Psi}=\mathcal{L}[Z_1]$ & Theorem \ref{thm_ciidlom}($\mathcal{E}$)\\
multivariate wide-sense geometric distribution & $\sum_{n=1}^{\lfloor t \rfloor}Y_n$, $Y_k$ iid & $\prod_{k=1}^{d}\IE[\exp(-k\,Y_1)]^{x_{[d-k+1]}-x_{[d-k]}}$ & Theorem \ref{thm_ciidlom}($\mathcal{G}$)\\
min-stable multivariate exponential law & $b\,t-c\,\log\big\{\prod\limits_{n \geq 1}G^{(n)}_{\sum\limits_{i=1}^{n}\eta_i/t-}\big\},\,G^{(n)}$ iid $\sim \gamma$ & $\exp\Big\{- b\,\sum\limits_{k=1}^{d}x_k-c\,\int_{\mathfrak{H}_{+,1}} \ell_G(\bm{x})\,\gamma(\mathrm{d}G)\Big\}$ & Theorem \ref{thm_minstable}\\
exogenous shock model & additive subordinator & $\frac{\exp\big\{-\sum_{k=1}^{d}\Psi_{x_{[d-k+1]}}(k)\big\}}{\exp\big\{-\Psi_{x_{[d-k+1]}}(k-1)\big\}},\,e^{-\Psi_t}=\mathcal{L}[Z_t]$ & Theorem \ref{thm_exShock}\\
Sato-frailty model & self-similar, additive subordinator & $\frac{\exp\big\{-\sum_{k=1}^{d}\Psi(k\,x_{[d-k+1]})\big\}}{\exp\big\{-\Psi((k-1)\,x_{[d-k+1]})\big\}},\,e^{-\Psi}=\mathcal{L}[Z_1]$ & Lemma \ref{lemma_BFchar} \\
\hline
\end{tabular}
\endgroup
\end{center}
\label{tab}
\end{table}
\end{landscape}
\pagestyle{plain}

\newpage 

\section{Introduction and general background}
\subsection{General notation}
Before we start, let us clarify some general notation used throughout, while some section-specific notations are introduced where they appear.
\par
\emph{Some very general mathematical definitions:}\\
We denote by $\IN=\{1,2,3,\ldots\}$ the set of natural numbers and $\IN_0:=\IN \cup \{0\}$, by $\IR$ the set of real numbers, by $\IR^d$ the set of $d$-dimensional row vectors with entries in $\IR$, for $d \in \IN$. For $n \in \IN_0$ we denote by $f^{(n)}$ the $n$-th derivative of a function $f:\IR \rightarrow \IR$, provided existence. For $d$ numbers $x_1,\ldots,x_d \in \IR$ we denote by $x_{[1]} \leq x_{[2]} \leq \ldots \leq x_{[d]}$ an ordered list. For $x \in \IR$ we denote by $\lceil x \rceil$ the smallest integer greater or equal to $x$ (ceiling function), and by $\lfloor x \rfloor$ the largest integer less than or equal to $x$ (floor function). We denote by $\mbox{det}[A]$ the determinant of a square matrix $A \in \IR^{d \times d}$. We denote elements $\bm{x}=(x_1,\ldots,x_d) \in \IR^d$ by bold letters in comparison to (one-dimensional) elements $x \in \IR$. Expressions like $\bm{x}>\bm{y}$ for $\bm{x},\bm{y} \in \IR^d$ are meant component-wise, i.e.\ $x_k>y_k$ for each $k=1,\ldots,d$. We furthermore use the notation $\norm{\bm{x}}_p:=(|x_1|^p+\ldots+|x_d|^p)^{1/p}$ for the $\ell_p$-norm of $\bm{x} \in \IR^d$, $p \geq 1$. 
\par
\emph{Some general definitions regarding probability spaces:}\\
All random objects to be introduced are formally defined on some probability space $(\Omega,\F,\IP)$ with $\sigma$-algebra $\F$ and probability measure $\IP$, and the expected value of a random variable $X$ is denoted by $\IE[X]$. As usual, the argument $\omega \in \Omega$ of some random variable $X:\Omega \rightarrow \IR$ will always be omitted. The symbol $\stackrel{d}{=}$ denotes equality in distribution and the symbol $\sim$ means ``is distributed according to''. We recall that for random vectors $(X_1,\ldots,X_d) \stackrel{d}{=} (Y_1,\ldots,Y_d)$ means $\IE[g(X_1,\ldots,X_d)]=\IE[g(Y_1,\ldots,Y_d)]$ for all bounded, continuous functions $g:\IR^d \rightarrow \IR$, where the expectation values $\IE$ are taken on the respective probability spaces of $(X_1,\ldots,X_d)$ and $(Y_1,\ldots,Y_d)$, which might be different. Equality in law for two stochastic processes $X=\{X_t\}$ and $Y=\{Y_t\}$ means that $(X_{t_1},\ldots,X_{t_d})\stackrel{d}{=}(Y_{t_1},\ldots,Y_{t_d})$ for arbitrary $d \in \IN$ and $t_1,t_2,\ldots,t_d$. Throughout, the abbreviation \emph{iid} stands for \underline{i}ndependent and \underline{i}dentically \underline{d}istributed. The index $t$ of a real-valued random variable $f_t$ that belongs to some stochastic process $f=\{f_t\}_{t \in T}$ is purposely written as a sub-index, in order to distinguish it from the value $f(t)$ of some (non-random) function $f:T \rightarrow \IR$. If $F$ is the distribution function of some random variable taking values in $\IR$, we denote by 
\begin{gather*}
F^{-1}(y):=\inf\{x \in \IR\,:\,F(x) \geq y\},\quad y \in [0,1],
\end{gather*}
its \emph{generalized inverse}, see \cite{embrechts13} for background. Any distribution function $C$ of a random vector $\bm{U}=(U_1,\ldots,U_d)$ whose components $U_k$ are uniformly distributed on $[0,1]$ is called a \emph{copula}, see \cite{mai12} for a textbook treatment. We further recall that an arbitrary survival function $\bar{F}$ of some $d$-variate random vector $\bm{X}$ can always be written\footnote{See \cite[p.\ 195--196]{mcneil05}.} as $\bar{F}(\bm{x})= \hat{C}\big(  \IP(X_1>x_1),\ldots,\IP(X_d>x_d)\big)$, where $\hat{C}$ is a copula, called a \emph{survival copula} for $\bar{F}$, and it is uniquely determined in case the random variables $X_1,\ldots,X_d$ have continuous distribution functions. This is the survival analogue of the so-called Theorem of Sklar, due to \cite{sklar59}. The Theorem of Sklar itself states that the distribution function $F$ of $\bm{X}$ can be written as $F(\bm{x})= C\big(  \IP(X_1 \leq x_1),\ldots,\IP(X_d \leq x_d)\big)$ for a copula $C$, called a \emph{copula} for $F$. The relationship between a copula $C$ and its survival copula $\hat{C}$ is that if $(U_1,\ldots,U_d) \sim C$ then $(1-U_1,\ldots,1-U_d) \sim \hat{C}$. 
\par
\emph{A notation of specific interest in the present survey:}\\
We denote by $\mathfrak{H}$ the set of all distribution functions of real-valued random variables, and by $\mathfrak{H}_+$ the subset containing all elements $F$ such that $x<0$ implies $F(x)=0$, i.e.\ distribution functions of non-negative random variables. Elements $F \in \mathfrak{H}$ are right-continuous, and we denote by $F(x-):=\lim_{t \uparrow x}F(x)$ their left-continuous versions. If $\mathfrak{X}$ is some Hausdorff space, we denote by $M_+^1(\mathfrak{X})$ the set of all probability measures on the measurable space $(\mathfrak{X},\mathcal{B}(\mathfrak{X}))$, where $\mathcal{B}(\mathfrak{X})$ denotes the Borel-$\sigma$-algebra of $\mathfrak{X}$. This notation is borrowed from \cite{berg84}. Now recall that $\mathfrak{H}$ is metrizable (hence in particular Hausdorff) when topologized with the so-called \emph{L\'evy metric} that induces weak convergence of the associated probability distributions on $\IR$, see \cite{sibley71}. Consequently, we denote by $M_+^{1}(\mathfrak{H})$ the set of all probability measures on $\mathfrak{H}$. A random element $H=\{H_t\}_{t \in \IR} \sim \gamma \in M_+^1(\mathfrak{H})$ is almost surely a c\`adl\`ag stochastic process, and a common way of treating probability laws of such objects works via the so-called \emph{Skorohod metric} on the space of c\`adl\`ag paths. However, even though the Skorohod topology and the topology induced by the L\'evy metric are not identical, see \cite[p.\ 327-328]{pestman09}, their induced Borel-$\sigma$-algebras on the set $\mathfrak{H}$ can indeed be shown to coincide, so that our viewpoint is equivalent. 
 
\subsection{Motivation and mathematical preliminaries}
Throughout, we consider by $\bm{X}=(X_1,\ldots,X_d)$ a random vector taking values in $\IR^d$. Since we are only interested in the probability distribution of $\bm{X}$, we identify $\bm{X}$ with its probability law in the sense that we often say $\bm{X}$ has some property if and only if its probability distribution has this property. The central theme of the present survey deals with the following formal definition, using the nomenclature in \cite{dykstra73}.
\begin{definition}[Conditionally iid]\label{def_condiid}
We say that a probability measure $\mu \in M_+^1(\IR^d)$ is \emph{conditionally iid} if there exists a probability measure $\gamma \in M_+^1(\mathfrak{H})$ such that the equality
\begin{gather*}
\mu\big((-\infty,x_1]\times \ldots \times (-\infty,x_d]\big) = \int_{\mathfrak{H}} h(x_1)\,\cdots\,h(x_d)\,\gamma(\mathrm{d}h)
\end{gather*}
holds for all $x_1,\ldots,x_d \in \IR$.\\
We say that a random vector $\bm{X}=(X_1,\ldots,X_d)$ taking values in $\IR^d$ is \emph{conditionally iid} if its probability distribution is conditionally iid.\\
Given a family of probability distributions $\mathfrak{M} \subset M_+^1(\IR^d)$, we introduce the notation
\begin{gather*}
\mathfrak{M}_{\ast} = \{\mu \in \mathfrak{M}\,:\,\mu \mbox{ is conditionally iid}\}.
\end{gather*}
\end{definition}

Consider a random vector $\bm{X}=(X_1,\ldots,X_d)$ on a probability space $(\Omega,\F,\IP)$ that is defined via 
\begin{gather*}
X_k:=f(U_k,H), \quad k = 1,\ldots,d, 
\end{gather*}
with some measurable ``functional'' $f$, an iid sequence of random objects $U_1,\ldots,U_d$, and some independent random object $H$ that is measurable with respect to the sub-$\sigma$-algebra $\mathcal{H}=\sigma(H) \subset \F$ that it generates itself. Such $\bm{X}$ is always conditionally iid and the probability distribution of the stochastic process $H_t:=\IP(X_1 \leq t\,|\,\mathcal{H})$, $t \in \IR$, plays the role of the probability measure $\gamma$ in Definition \ref{def_condiid}. The object $H$, sometimes called a \emph{latent (dependence-inducing) factor}, then induces dependence between the components, which are iid conditioned on $\mathcal{H}$. This is a Bayesian viewpoint, based on a two-step construction: first simulate an instance of $H$, then simulate $X_1,\ldots,X_d$ iid given $H$. However, it is important to be aware that a random vector that is conditionally iid according to our definition does not necessarily have to be defined by a stochastic model that relies on such a two-step construction. Our definition only requires that such a construction exists, possibly on another probability space. In fact, typical cases of interest are such that $\mu$ is defined in terms of a stochastic model or probabilistic property which is a priori unrelated to the concept of conditional independence, as we will see. 
\par
Throughout, we are interested in a solution to the following problem:

\begin{problem}[Motivating problem] \label{motivatingproblem}
Given a collection $\mathfrak{M} \subset M_+^1(\IR^d)$ and $\mu \in \mathfrak{M}$, provide necessary and sufficient conditions ensuring that $\mu \in \mathfrak{M}_{\ast}$.
\end{problem}

\begin{remark}[Nomenclature]\label{rmk_nomenclature}
In the literature, elements of $\mathfrak{M}_{\ast}$ are not always called conditionally iid, but other names have been given. For instance, \cite{shaked77} calls them \emph{positive dependent by mixture (PDM)}, \cite{spizzichino82,gnedin95,liggett07} call them \emph{infinitely extendible}, and \cite[Definition 1.10, p.\ 43]{mai12} call them simply \emph{extendible}. The nomenclature ``(infinite) extendibility'' refers to the fact that conditionally iid random vectors can always be thought of as finite margins of infinite conditionally iid sequences, as will be explained below. The nomenclature ``PDM'' becomes intuitive from Lemmata \ref{lemma_nonegass}, \ref{lemma_kendall} and \ref{lemma_Shakedpos} below, but is rather unusual.
\end{remark}

The investigation of conditionally iid random vectors is closely related to the concept of exchangeability. Recall that the probability distribution of a random vector $\bm{X}$ is called \emph{exchangeable} if it is invariant under an arbitrary permutation of the components of $\bm{X}$. The following observation is immediate but important.
\begin{lemma}[Exchangeability] \label{lemma_exchangeable}
If $\bm{X}$ is conditionally iid, it is also exchangeable. 
\end{lemma} 
\begin{proof}
Let $\bm{X} \sim \mu$ with $\gamma$ as in Definition \ref{def_condiid}. If $\pi$ is an arbitrary permutation of $\{1,\ldots,d\}$ we observe that
\begin{align*}
& \IP(X_{1} \leq x_1,\ldots,X_{1} \leq x_d) = \int_{\mathfrak{H}} h(x_1)\,\cdots\,h(x_d)\,\gamma(\mathrm{d}h)\\
& \qquad= \int_{\mathfrak{H}} h(x_{\pi(1)})\,\cdots\,h(x_{\pi(d)})\,\gamma(\mathrm{d}h) = \IP(X_{\pi^{-1}(1)} \leq x_{1},\ldots,X_{\pi^{-1}(1)} \leq x_{d}).
\end{align*}
Since $\pi$ was arbitrary, this shows that the distribution function (hence law) of $\bm{X}$ is invariant with respect to permutations of its components.
\end{proof}

Exchangeability is a property which is convenient to investigate by means of Analysis, whereas the notion ``conditionally iid'', in which we are interested, is a priori purely probabilistic and more difficult to investigate. Unfortunately, exchangeability is only a necessary but no sufficient condition for the solution of our problem. For instance, a bivariate normal distribution is obviously exchangeable if and only if the two means and variances are identical, also for negative correlation coefficients. However, Example \ref{ex_Normal} and Lemma \ref{lemma_nonegass} below show that conditionally iid random vectors necessarily have non-negative correlation coefficients. One can show in general that the correlation coefficient -- if existent -- between two components of an exchangeable random vector on $\IR^d$ is bounded from below by $-1/(d-1)$, see, e.g., \cite[p.\ 7]{aldous85}. As the dimension $d$ tends to infinity, this lower bound becomes zero. Even better, the difference between exchangeabilty and a conditionally iid structure vanishes completely as the dimension $d$ tends to infinity, which is the content of de Finetti's Theorem.
\begin{theorem}[de Finetti's Theorem]\label{thm_definetti}
Let $\{X_k\}_{k \in \IN}$ be an infinite sequence of random variables on some probability space $(\Omega,\F,\IP)$. The sequence $\{X_k\}_{k \in \IN}$ is \emph{exchangeable}, meaning that each finite subvector is exchangeable, if and only if it is iid conditioned on some $\sigma$-field $\mathcal{H} \subset \F$. In this case, $\mathcal{H}$ equals almost surely the tail-$\sigma$-field of $\{X_k\}_{k \in \IN}$, which is given by $\cap_{n \geq 1}\sigma(X_n,X_{n+1},\ldots)$. 
\end{theorem}
\begin{proof}
Originally due to \cite{definetti37}. We refer to \cite{aldous85} for a proof based on the reversed martingale convergence theorem, which is briefly sketched. Of course, we only need to verify that exchangeability implies conditionally iid, as the converse follows from Lemma \ref{lemma_exchangeable}. For the sake of a more convenient notation we assume the infinite sequence $\{X_k\}_{k \in \IN_0}$ is indexed by $\IN_0 = \IN \cup \{0\}$, and we define $\sigma$-algebras $\F_{-n}:=\sigma(X_n,X_{n+1},\ldots)$ for $n \in \IN_0$. The tail-$\sigma$-filed of the sequence is $\mathcal{H}:=\cap_{n \leq 0}\F_n$. In order to establish the claim, three auxiliary observations are helpful with an arbitrary bounded, measurable function $g$ fixed:
\begin{itemize}
\item[(i)] Exchangeability implies $(X_0,X_1,\ldots) \stackrel{d}{=}(X_0,X_{n+1},\ldots)$ for arbitrary $n \in \IN$.  This implies $\IE[g(X_0) \,|\,\F_{-1}]\stackrel{d}{=}\IE[g(X_0)\,|\,\F_{-(n+1)}]$, $n \in \IN$.
\item[(ii)] The sequence $Y_n:=\IE[g(X_0)\,|\,\F_n]$, $n \leq 0$, is easily checked to be a reversed martingale. The reversed martingale convergence theorem implies that $Y_n$ converges almost surely and in $L^1$ to $\IE[g(X_0)\,|\,\mathcal{H}]$. See \cite[p.\ 264 ff]{durrett10} for background on reversed martingales (convergence).
\item[(iii)] Letting $n \rightarrow \infty$ in (i), we observe from (ii) that $Y_{-1} \stackrel{d}{=} \IE[g(X_0)\,|\,\mathcal{H}]$. We can further replace this equality in law by an almost sure equality, since $\mathcal{H} \subset \F_{-1}$ and the second moments of $Y_{-1}$ and $\IE[g(X_0)\,|\,\mathcal{H}]$ coincide. Thus, the sequence $\{Y_{-n}\}_{n \in \IN}$ is almost surely a constant sequence.
\end{itemize}
With these auxiliary observations we may now finish the argument. On the one hand, exchangeability implies $(X_0,X_{n+1},\ldots) \stackrel{d}{=}(X_n,X_{n+1},\ldots)$, which gives the almost sure equality $\IE[g(X_0)\,|\,\F_{-(n+1)}]=\IE[g(X_n)\,|\,\F_{-(n+1)}]$. Taking $\IE[.\,|\,\mathcal{H}]$ on both sides of this equation implies with the tower property of conditional expectation that $\IE[g(X_0)\,|\,\mathcal{H}]=\IE[g(X_n)\,|\,\mathcal{H}]$. Since $g$ was arbitrary, $X_0$ and $X_n$ are identically distributed conditioned on $\mathcal{H}$, and since $n$ was arbitrary all members of the sequence are identically distributed conditioned on $\mathcal{H}$. To verify conditional independence, let $g_1,g_2$ be two bounded, measurable functions. For $n \geq 1$ arbitrary, using (iii) in the third equality below, we compute
\begin{align*}
\IE[g_1(X_0)\,g_2(X_n)\,|\,\mathcal{H}] & = \IE[ \IE[g_1(X_0)\,g_2(X_n)\,|\,\F_{-n}]\,|\,\mathcal{H}]\\
&  = \IE[ g_2(X_n)\,\IE[g_1(X_0)\,|\,\F_{-n}]\,|\,\mathcal{H}] \\
& = \IE[ g_2(X_n)\,\IE[g_1(X_0)\,|\,\mathcal{H}]\,|\,\mathcal{H}]\\
&= \IE[ g_2(X_n)\,|\,\mathcal{H}]\,\IE[g_1(X_0)\,|\,\mathcal{H}].
\end{align*} 
The precisely same tower property argument inductively also implies
\begin{align*}
\IE\Big[\prod_{j=1}^{k}g_{i}(X_{i_j})\,\Big|\,\mathcal{H}\Big]= \prod_{j=1}^{k}\IE[ g_i(X_{i_j})\,|\,\mathcal{H}]
\end{align*} 
for arbitrary $0 \leq i_1<\ldots<i_k$ and bounded measurable functions $g_1,\ldots,g_k$. Thus, the random variables $X_0,X_1,\ldots$ are independent conditioned on $\mathcal{H}$. 
\end{proof}

\emph{Which topics are covered in the present survey?}\\
The present article surveys known answers to Problem \ref{motivatingproblem} for families of multivariate probability distributions $\mathfrak{M}$ that are well known in the statistical literature, and/ or have proven useful as a mathematical model for specific applications. While several traditional results of the theory have been studied in the last century, some significant achievements have been accomplished only within the last decade, so the present author feels that this is a good time point to recap what has been achieved, hence to write this overview article. One goal of the present survey is to collect the numerous results under one common umbrella in a reader-friendly summary to make them accessible for a broader audience of applied and theoretical probabilists, and in order to inspire others to join this interesting strand of research in the future. Proofs, or at least proof sketches, are presented  for most results in order to (a) demonstrate how solutions to Problem \ref{motivatingproblem} often unravel surprising links between seemingly different fields of mathematics/probability theory, and (b) render this document a useful basis for the use as lecture notes in an advanced course on multivariate statistics or probability theory.
\par
\emph{Which topics are not covered in the present survey?}\\
The scope of former literature on the topic is often wider, in particular the references \cite{aldous85,kingman78,aldous10} are very popular surveys on the topic with wider scope. On the one hand, many references on the topic discuss conditionally iid models under the umbrella of exchangeability, which has been mentioned to be a weaker notion for finite random vectors. The characterization of the (finitely) exchangeable subfamily of $\mathfrak{M}$ is often easier than the characterization of the (in general) smaller set $\mathfrak{M}_{\ast}$, and is typically an important first step towards a solution to Problem \ref{motivatingproblem}. However, the second (typically harder) step from (finite) exchangeability to conditionally iid is usually the more important and more interesting step from both a theoretical and practical perspective. The algebraic structure of a general theory on (finite) exchangeability is naturally of a different, often more combinatorial character, whereas ``conditionally iid'' by virtue of de Finetti's Theorem naturally is the concept of an infinite limit (of exchangeability) so that techniques from Analysis enter the scene. Thus, we feel it is useful to provide an account with a more narrow scope on conditionally iid, even though for some of the presented examples we are well aware that an interesting (finite) exchangeable theory is also viable. On the other hand, many references consider the case when the components of $\bm{X}$ take values in more general spaces than $\IR$, for instance in $\IR^n$ (i.e.\ lattices instead of vectors) or even function spaces. In particular, de Finetti's Theorem \ref{thm_definetti} can be generalized in this regard, seminal references are \cite{hewitt55,ressel85}. Research in this direction is by nature more abstract and thus maybe less accessible for a broader audience, or for more practically oriented readers. One goal of the present survey is to provide an account that is not exclusively geared towards theorists but also to applicants of the theory, and in particular to point out relationships to classical statistical probability laws on $\IR^d$. We believe that a limitation of this survey's scope to the real-valued case is still rich enough to provide a solid basis for an interesting and accessible theory. In fact, we seek to demonstrate that Problem \ref{motivatingproblem} has been solved satisfactorily in quite a number of highly interesting cases, and the solutions contain interesting links to different probabilistic topics. Of course, it might be worthwhile to ponder about generalizations of some of the presented results to more abstract settings in the future (unless already done) - but purposely these lie outside the present survey.
\par
\emph{Why is Problem \ref{motivatingproblem} interesting at all?}\\
Broadly speaking, because of two reasons: (a) conditionally iid models are convenient for applications, and (b) solutions to the extendibility problem sometimes rely on compelling relationships between a priori different theories.
\begin{itemize}
\item[(a)] Roughly speaking, conditionally iid models allow to model (strong and weak) dependence between random variables in a way that features many desirable properties which are taylor-made for applications, in particular when the dimension $d$ is large. Firstly, a conditionally iid random vector is ``dimension-free'' in the sense that components can be added or removed from $\bm{X}$ without altering the basic structure of the model, which simply follows from the fact that an iid sequence remains an iid sequence after addition or removal of certain members. This may be very important in applications that require a regular change of dimension, e.g.\ the readjustment of a large credit portfolio in a bank, when old credits leave and new credits enter the portfolio frequently. Secondly, if $\bm{X}$ has a distribution from a parametric family, the parameters of this family are typically determined by the parameters of the underlying latent probability measure $\gamma$, irrespective of the dimension $d$. Consequently, the number of parameters does usually not grow significantly with the dimension $d$ and may be controlled at one's personal taste. This is an enormous advantage for model design in practice, in particular since the huge degree of freedom/ huge number of parameters in a high-dimensional dependence model is often more boon than bane. Thirdly, fundamental statistical theorems relying on the iid assumption, like the law of large numbers, may still be applied in a conditionally iid setting, making such models very tractable. Last but not least, in dependence modeling a ``factor-model way of thinking'' is very intuitive, e.g.\ it is well-established in the multivariate normal case (thinking of principle component analyses etc.). On a high level, if one wishes to design a multi-factor dependence model within a certain family of distributions $\mathfrak{M}$, an important first step is to determine the one-factor subfamily $\mathfrak{M}_{\ast}$. Having found a conditionally iid stochastic representation of $\mathfrak{M}_{\ast}$, the design of multi-factor models is sometimes obvious from there, see also paragraph \ref{open_multifactor}.
\item[(b)] The solution to Problem \ref{motivatingproblem} is often mathematically challenging and compelling. It naturally provides an interesting connection between the ``static'' world of random vectors and the ``dynamic'' world of (one-dimensional) stochastic processes. The latter enter the scene because the latent factor being responsible for the dependence in a conditionally iid model for $\bm{X}$ may canonically be viewed as a non-decreasing stochastic process (a random distribution function), which is further explained in Section \ref{sec:canonic} below. In particular, for some classical families $\mathfrak{M}$ of multivariate laws from the statistical literature the family $\mathfrak{M}_{\ast}$ in Problem \ref{motivatingproblem} is conveniently described in terms of a well-studied family of stochastic processes like L\'evy subordinators, Sato subordinators, or processes which are infinitely divisible with respect to time. Moreover, in order to formally establish the aforementioned link between these two seemingly different fields of research the required mathematical techniques involve classical theories from Analysis like Laplace transforms, Bernstein functions, and moment problems. 
\end{itemize}

\emph{Before we start, can we please study a first simple example?}\\
It is educational to end this motivating paragraph by demonstrating the motivating problem with a simple example that all readers are familiar with. Denoting by $\mathcal{N}(\bm{\mu},\Sigma)$ the multivariate normal law with mean vector $\bm{\mu}=(\mu_1,\ldots,\mu_d) \in  \IR^d$ and covariance matrix $\Sigma \in \IR^{d \times d}$, Example \ref{ex_Normal} provides the solution for Problem \ref{motivatingproblem} in the case when $\mathfrak{M}$ consists of all multivariate normal laws.
\begin{example}[The multivariate normal law]\label{ex_Normal}
We want to solve Problem \ref{motivatingproblem} for the family 
\begin{gather*}
\mathfrak{M}=\{\mathcal{N}(\bm{\mu},\Sigma)\,:\,\bm{\mu} \in \IR^d,\,\Sigma \in \IR^{d \times d} \mbox{ symmetric, positive definite}\}.
\end{gather*}
We claim that $\mathfrak{M}_{\ast}$ equals the set of all multivariate normal distributions satisfying
\begin{gather*}
\bm{\mu}=(\mu,\ldots,\mu), \quad\Sigma= \begin{bmatrix} 
\sigma^2 & \rho\,\sigma^2 &  \ldots & \rho\,\sigma^2\\
\rho\,\sigma^2 & \sigma^2  &   & \rho\,\sigma^2\\
\vdots &  & \ddots  & \\
\rho\,\sigma^2 & \rho\,\sigma^2   &  & \sigma^2\\
\end{bmatrix},\quad \mu \in \IR,\,\sigma>0,\,\rho \in [0,1].
\end{gather*}
\begin{proof}
Consider $\bm{X}=(X_1,\ldots,X_d) \sim \mathcal{N}(\bm{\mu},\Sigma)$ on some probability space $(\Omega,\F,\IP)$ for $\bm{\mu}=(\mu_1,\ldots,\mu_d) \in \IR^d$, and $\Sigma = (\Sigma_{i,j}) \in \IR^{d \times d}$ a positive definite matrix. If we assume that the law of $\bm{X}$ is in $\mathfrak{M}_{\ast}$, it follows that there is a sub-$\sigma$-algebra $\mathcal{H} \subset \F$ such that the components $X_1,\ldots,X_d$ are iid conditioned on $\mathcal{H}$. Consequently,
\begin{gather}
\mu_k=\IE[X_k]=\IE[\IE[X_k\,|\,\mathcal{H}]]=\IE[\IE[X_1\,|\,\mathcal{H}]]=\IE[X_1]=\mu_1,
\label{normcalc1}
\end{gather}
irrespectively of $k=1,\ldots,d$. The analogous reasoning also holds for the second moment of $X_k$, which implies $\Sigma_{k,k}=\Sigma_{1,1}$ for all $k$. Moreover,
\begin{gather}
\IE[X_i\,X_j]= \IE[\IE[X_i\,|\,\mathcal{H}]\,\IE[X_j\,|\,\mathcal{H}]]=\IE[\IE[X_1\,|\,\mathcal{H}]^2]\geq \mu_1^2, 
\label{normcalc2}
\end{gather}
for arbitrary components $i \neq j$, where we used the conditional iid structure and Jensen's inequality. This finally implies that all off-diagonal elements of $\Sigma$ are identical and non-negative.
\par
Conversely, let $\mu \in \IR$, $\sigma>0$, and $\rho \in [0,1]$. Consider a probability space on which $d+1$ iid standard normally distributed random variables $M,M_1,\ldots,M_d$ are defined. We define
\begin{gather}
X_k:=\mu+\sigma\,\big( \sqrt{\rho}\,M+\sqrt{1-\rho}\,M_k\,\big),\quad k=1,\ldots,d.
\label{constr_exnormal}
\end{gather}
It is readily observed that $\bm{X}=(X_1,\ldots,X_d)$ has a multivariate normal law with pairwise correlation coefficients all being equal to $\rho$, and all components having mean $\mu$ and variance $\sigma^2$. Notice in particular that the non-negativity of $\rho$ is important in the construction (\ref{constr_exnormal}) because the square root is not well-defined otherwise. The components of $\bm{X}$ are obviously conditionally iid given the $\sigma$-algebra $\mathcal{H}$ generated by $M$. Hence the law of $\bm{X}$ is in $\mathfrak{M}_{\ast}$. 
\end{proof}
\end{example}

There are already some interesting remarks to be made about this simple example. First of all, it is observed that the family $\mathfrak{M}_{\ast}$ is always three-parametric, irrespective of the dimension $d$. This stands in glaring contrast to $\mathfrak{M}$, which has $d+d\,(d+1)/2$ parameters in dimension $d$. Second, in general a Monte Carlo simulation of a $d$-dimensional normal random vector $\bm{X}$ requires a Cholesky decomposition of the matrix $\Sigma$, which typically has computational complexity of order $d^3$, see \cite[Algorithm 4.3, p.\ 182]{mai12}. In contrast, the simulation of $\bm{X}$ with law in $\mathfrak{M}_{\ast}$ according to (\ref{constr_exnormal}) has only linear complexity in the dimension $d$. Especially in large dimensions this can be a critical improvement of computational speed. Third, the proof above shows that each random vector with law in $\mathfrak{M}_{\ast}$ may actually be viewed as the first $d$ components of an infinite sequence of conditionally iid random variables such that arbitrary finite $n$-margins have a multivariate normal law. Thus, we have actually solved the re-fined Problem \ref{motivatingproblemrefined} to be introduced in the upcoming paragraph.

\subsection{Canonical probability spaces} \label{sec:canonic}
We have mentioned earlier that a conditionally iid random vector $\bm{X}$ is usually constructed as $X_k:=f(U_k,H)$, $k = 1,\ldots,d$, from an iid sequence $U_1,\ldots,U_d$, some independent stochastic object $H$, and some functional $f$. Clearly, this general model is inconvenient because neither the law of $U_1$, nor the nature of the stochastic object $H$ or the functional $f$ are given explicitly. However, there is a canonical choice for all three entities, which we are going to consider in the sequel. By definition, conditionally iid means that conditioned on the object $H$ the random variables $X_1,\ldots,X_d$ are iid, distributed according to a univariate distribution function $F$, which may depend on $H$. A univariate distribution function $F$ is nothing but a non-decreasing, right-continuous function $F:\IR \rightarrow [0,1]$ with $\lim_{t \rightarrow -\infty}F(t)=0$ and $\lim_{t \rightarrow \infty}F(t)=1$, see \cite[Theorem 12.4, p.\ 176]{billingsley95}. Without loss of generality we may assume that the random object $H=\{H_t\}_{t \in \IR}$ already is the conditional distribution function itself, i.e.\ is a random variable in the space of distribution functions -- or, in other words, a non-decreasing, right-continuous stochastic process with $\lim_{t \rightarrow -\infty}H_t=0$ and $\lim_{t \rightarrow \infty}H_t=1$. In other words, $H \sim \gamma \in M_+^{1}(\mathfrak{H})$. In this case, a canonical choice for the law of $U_1$ is the uniform distribution on $[0,1]$ and the functional $f$ may be chosen as
\begin{gather}
\label{canonical_construction}
X_k = f(U_k,H) := \inf\{ t \in \IR \,:\,H_t > U_k\}=H^{-1}_{U_k},\quad k=1,\ldots,d.
\end{gather}
Recall here that $H$ is interpreted as a random distribution function and $H^{-1}$ denotes its generalized inverse. In particular, $X_k \leq x$ if and only if $U_k \leq H_x$. Indeed, one verifies that $X_1,\ldots,X_d$ are iid conditioned on $\mathcal{H}:=\sigma\big( \{H_t\}_{t \in \IR}\big)$, with common univariate distribution function $H$, since 
\begin{align*}
\IP(X_1 \leq t_1,\ldots,X_d \leq t_d\,|\,\mathcal{H}) &= \IP(U_1 \leq H_{t_1},\ldots,U_d \leq H_{t_d}\,|\,\mathcal{H}) \\
&= H_{t_1}\,H_{t_2}\,\cdots\,H_{t_d}, 
\end{align*}
for all $t_1,\ldots,t_d \in \IR$. Every random vector which is conditionally iid can be constructed like this, i.e.\ there is a one-to-one relation between such models and random variables in the space of (one-dimensional) distribution functions, as already adumbrated in Definition \ref{def_condiid}. For each given $H=\{H_t\}_{t \in \IR} \sim \gamma \in M_+^{1}(\mathfrak{H})$, and a given dimension $d \in \IN$, the canonical construction (\ref{canonical_construction}) induces a multivariate probability distribution on $\IR^d$, and we denote this mapping from $M_+^{1}(\mathfrak{H})$ to a subset $\mathfrak{M}_{\ast}$ of $M_+^1(\IR^d)$ by $\Theta_d$ throughout. It is implicit that $\Theta_d(\gamma)$ depends on the law of $H \sim \gamma$ via
\begin{gather*}
\Theta_d(\gamma)\big( (-\infty,t_1]\times \ldots \times (-\infty,t_d]\big) = \IE[H_{t_1}\,\dots\,H_{t_d}].
\end{gather*}
Given $\mathfrak{M} \subset M_+^1(\IR^d)$ we denote the pre-image of $\mathfrak{M}_{\ast}$ under $\Theta_d$ in $M_+^{1}(\mathfrak{H})$ by $\Theta_d^{-1}(\mathfrak{M}_{*})$. In words, it equals the subset of $M_+^{1}(\mathfrak{H})$ which consists of all probability laws $\gamma$ of stochastic processes $\{H_t\}_{t \in \IR}$ such that $\bm{X}$ of the canonical construction (\ref{canonical_construction}) has a law in $\mathfrak{M}$, hence in $\mathfrak{M}_{\ast}$. From this equivalent viewpoint our motivating Problem \ref{motivatingproblem} becomes

\begin{problem}[Motivating problem reformulated] \label{motivatingproblem2}
For a given family of $d$-dimensional probability distributions $\mathfrak{M}$, determine the family of stochastic processes $\Theta_d^{-1}(\mathfrak{M}_{\ast}) \subset M_+^{1}(\mathfrak{H})$. 
\end{problem}

Admittedly, this reformulation in terms of the stochastic process $H$ might appear quite artificial at this point, but we will see later that in some cases we obtain interesting correspondences between classical probability distributions on $\IR^d$ and families of stochastic processes. On a high level, the problem of determining the intersection of a given family $\mathfrak{M}$ of distributions with the family of conditionally iid distributions may also be re-phrased as the problem of finding an increasing stochastic process whose stochastic nature induces the given multivariate distribution when inserted into a canonical stochastic model.
 
\begin{center}
		\begin{tikzpicture}[fill=gray]
\draw (-3,-2) rectangle (3,1)  (0,1) node [text=black,below] {given family} (0,0.5) node[text=black,below]{(some probability space)};
\draw (-1,-3.5) rectangle (5,-0.5)  (2,-3) node [text=black,above] {conditionally iid} (2,-3.5) node[text=black,above]{(canonical probability space)}  (1,-1) node [text=black,below] {?};
		\end{tikzpicture}
\end{center}

Obviously, in the stochastic model (\ref{canonical_construction}) it is possible to let $d$ tend to infinity, since the $U_k$ are iid. Thus, we may without loss of generality think of a conditionally iid random vector $\bm{X}$ as the first $d$ members of an infinite sequence $\{X_k\}_{k \in \IN}$ on $(\Omega,\F,\IP)$ such that conditioned on some $\sigma$-algebra $\mathcal{H} \subset \F$ the sequence $\{X_k\}_{k \in \IN}$ is iid. De Finetti's Theorem thus allows us to view conditionally iid random vectors $\bm{X}=(X_1,\ldots,X_d)$ as the first $d$ members of an infinite exchangeable sequence $\{X_k\}_{k \in \IN}$. More clearly, a probability law $\mu \in M_+^1(\IR^d)$ is conditionally iid if and only if there exists an infinite exchangeable sequence $\{X_k\}_{k \in \IN}$ on some probability space such that $\bm{X}=(X_1,\ldots,X_d) \sim \mu$. At this point it is important to highlight that we deal with a fixed dimension $d$. In general, it is possible that two truly different probability laws $\gamma_1 \neq \gamma_2$ are mapped onto the same element $\mu \in M_+^{1}(\IR^d)$ under the mapping $\Theta_d$. By virtue of de Finetti's Theorem, this ambiguity vanishes if we let $d$ tend to infinity in (\ref{canonical_construction}), i.e.\ the probability laws of $H$ and $\{X_k\}_{k \in \IN}$ stand in a one-to-one correspondence, see also Lemma \ref{lemma_condGC} below. In many cases of interest, we are actually not only interested in finding \emph{some} $\gamma$ that is mapped onto a given $\mu \in \mathfrak{M}$ under $\Theta_d$, but actually wish to find such $\gamma$ which is mapped onto a probability law with a desired property for arbitrary $n \geq 1$ by $\Theta_n$. In order to formalize this idea, we introduce the following definition.

\begin{definition}[Conditionally iid respecting (P)]
Let (P) be some property which makes sense in arbitrary dimension $n \geq 1$, and define the sets
\begin{gather*}
\mathfrak{M}_{n,(P)}:=\{\mu \in M_+^1(\IR^n)\,:\,\mu \mbox{ has property (P)}\} \subset M_+^1(\IR^n).
\end{gather*}
We say that $\mu \in \mathfrak{M}=\mathfrak{M}_{d,(P)}$ is \emph{conditionally iid respecting (P)} if there exists a stochastic process $H \in M_+^{1}(\mathfrak{H})$ whose probability law is mapped onto $\mu$ under $\Theta_d$ and, in addition, is mapped into an element of $\mathfrak{M}_{n,(P)}$ for arbitrary $n \geq 1$ under $\Theta_n$. We furthermore introduce the notation
\begin{gather*}
\mathfrak{M}_{\ast\ast}:=\{\mu \in \mathfrak{M}\,:\,\mu \mbox{ is conditionally iid respecting (P)}\}.
\end{gather*}
\end{definition}

Now we refine Problem \ref{motivatingproblem}.

\begin{problem}[Motivating problem refined] \label{motivatingproblemrefined}
Let (P) be a property that makes sense in any dimension, and consider $\mathfrak{M}=\mathfrak{M}_{d,(P)}$. For $\mu \in \mathfrak{M}$ provide necessary and sufficient conditions ensuring that $\mu \in \mathfrak{M}_{\ast\ast}$.
\end{problem}

In the situation of Problem \ref{motivatingproblemrefined} we have $\mathfrak{M}_{\ast\ast} \subset \mathfrak{M}_{\ast}$, and the inclusion can be proper in general, although this is unusual in cases of interest. A non-trivial example for the situation $\mathfrak{M}_{\ast\ast} \neq \mathfrak{M}_{\ast}$ is presented in Example \ref{ex_MstarnotMstarstar} in Section \ref{subsec_AC} below. A typical example for (P) is the property of ``being a multivariate normal distribution (in some dimension)''. For a given $d$-variate multivariate normal law it is a priori unclear whether there exists an infinite exchangeable sequence with $d$-margins being equal to the given multivariate normal law and such that all $n$-margins are multivariate normal as well for $n>d$. This is indeed the case and we have $\mathfrak{M}_{\ast\ast} = \mathfrak{M}_{\ast}$ in this particular situation, as can be inferred from Example \ref{ex_Normal}. The typical questions in the theory deal with subsets of $M_+^1(\IR^d)$ of the form $\mathfrak{M}=\mathfrak{M}_{d,(P)}$ for a property (P) that makes sense in arbitrary dimension, so most results presented are actually solutions to Problem \ref{motivatingproblemrefined} rather than to Problem \ref{motivatingproblem}, see also paragraph \ref{subsec_effectext} below for a further discussion related to this subtlety.

\subsubsection{Laws with positive components}
If the given family $\mathfrak{M}$ consists only of probability laws on $[0,\infty)^d$, it is convenient to slightly reformulate the stochastic model (\ref{canonical_construction}). Clearly, if we have non-negative components, necessarily $H_t=0$ for all $t < 0$ almost surely. Therefore, without loss of generality we may assume that $H=\{H_t\}_{t \geq 0}$ is indexed by $t \in [0,\infty)$. Moreover, applying the substitution $z=-\log(1-F)$ it trivially holds true that
\begin{align*}
\mathfrak{H}_+ & = \big\{t \mapsto 1-\exp(-z(t))\,\big|\,z:[0,\infty) \rightarrow [0,\infty] \mbox{ non-decreasing,}\\
& \qquad \qquad\mbox{right-continuous, with }z(0) \geq 0 \mbox{ and }\lim_{t \rightarrow \infty}z(t)=\infty \big\}.
\end{align*}
One may therefore rewrite the canonical construction (\ref{canonical_construction}) as
\begin{gather}
\label{canonical_construction_2}
X_k  := \inf\{ t \geq  0 \,:\,Z_t > \epsilon_k\},\quad k=1,\ldots,d,
\end{gather}
where the $\epsilon_k:=-\log(1-U_k)$, $k=1,\ldots,d$, are now iid exponential random variables with unit mean, and $Z=\{Z_t\}_{t \geq 0}$ is now no longer a distribution function, but instead a non-decreasing, right-continuous process with $Z_0 \geq 0$ and $\lim_{t \rightarrow \infty}Z_t=\infty$, related to $H$ via the substitution $Z_t=-\log(1-H_t)$. Figure \ref{fig:CC} illustrates one realization of the simulation mechanism (\ref{canonical_construction_2}). 
\vspace{1cm}

\begin{figure}[!ht]
\floatbox[{\capbeside\thisfloatsetup{capbesideposition={right,top},capbesidewidth=4cm}}]{figure}[7.5cm]
{\caption{\scriptsize One simulation of the canonical construction (\ref{canonical_construction_2}) in dimension $d=4$. One observes that $Z=\{Z_t\}_{t \geq 0}$ in this particular illustration has jumps, thus there is a positive probability that two components of $\bm{X}$ take the identical value. This does not happen if $Z$ is continuous, see Lemma \ref{lemma_pathcont} below.}}
{\includegraphics[width=8cm]{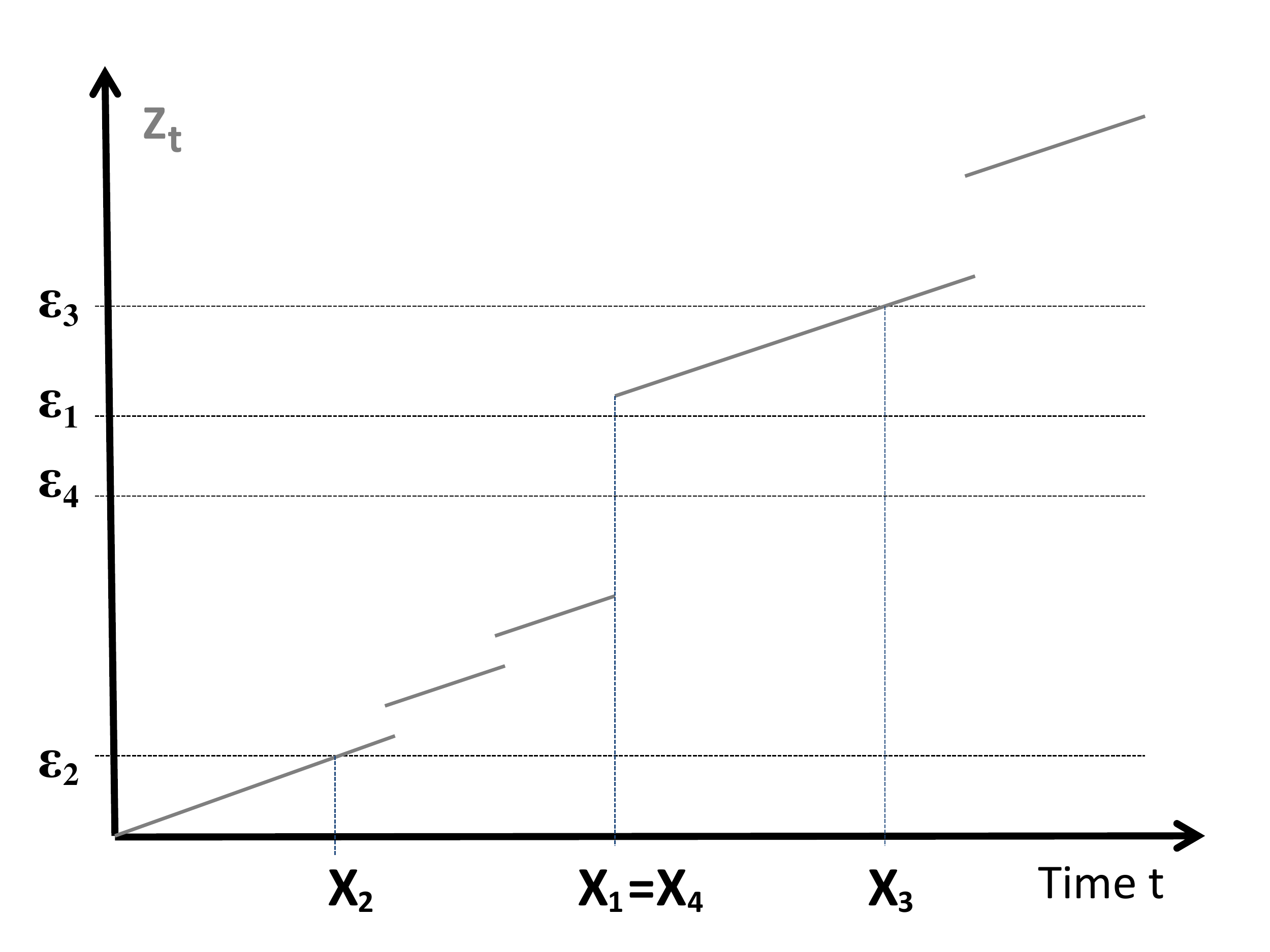} \label{fig:CC}}
\end{figure}

\subsection{General properties of conditionally iid models} \label{sec:generalM}
In this section we briefly collect some general properties of conditionally iid models. To this end, throughout this section we assume that $\mathfrak{M}=M_+^1(\IR^d)$ denotes the family of all $d$-dimensional probability laws on $\IR^d$ and we collect general properties of $\mathfrak{M}_{\ast}$. 
\subsubsection{Positive dependence}
If the law of $\bm{X}$ is in $\mathfrak{M}_{\ast}$, the covariance matrix of $\bm{X}$ - provided existence - cannot have negative entries. 
\begin{lemma}[Non-negative correlations] \label{lemma_nonegass}
If the law of $\bm{X}$ is in $\mathfrak{M}_{\ast}$ and the covariance matrix of $\bm{X}$ exists, then all its entries are non-negative. 
\end{lemma}
\begin{proof}
This follows from precisely the same computations that have been carried out already in (\ref{normcalc1}) and (\ref{normcalc2}) for the particular example of the multivariate normal distribution. 
\end{proof}
Correlation coefficients are sometimes inappropriate dependence measurements outside the Gaussian paradigm, see \cite{mcneil05,taleb20}. For instance, their existence depends on the existence of second moments, or we might have a correlation coefficient that is strictly less than one despite the fact that one component of the random vector is a monotone function of the other, since correlation coefficients depend on the marginal distributions as well. For these reasons, several alternative dependence measurements have been developed. One popular among them is the concordance measurement \emph{Kendall's Tau}. Recall that $\bm{x},\bm{y} \in \IR^2$ are called \emph{concordant} if $(x_1-y_1)\,(x_2-y_2)>0$ and \emph{discordant} if $(x_1-y_1)\,(x_2-y_2)<0$. In words, concordance means that one of the two points lies north-east of the other, while discordance means that one of the two lies north-west of the other. Kendall's Tau for a bivariate random vector $\bm{X}$ is defined as the difference between the probability of concordance and the probability of discordance for two independent copies of $\bm{X}$. If $\bm{X}$ is conditionally iid, Kendall's Tau is necessarily non-negative.
\begin{lemma}[Non-negative Kendall's Tau] \label{lemma_kendall}
If the law of $\bm{X}=(X_1,X_2)$ is in $\mathfrak{M}_{\ast}$, then Kendall's Tau is necessarily non-negative. 
\end{lemma}
\begin{proof}
Let $\bm{X}^{(1)}$ and $\bm{X}^{(2)}$ be two independent copies of $\bm{X}$, both defined on some common probability space. By assumption we find a $\sigma$-algebra $\mathcal{H}$ such that conditioned on $\mathcal{H}$ all four random variables $X^{(1)}_1,X^{(1)}_2,X^{(2)}_1,X^{(2)}_2$ are independent with respective distribution functions $H^{(1)}$ (for $X^{(1)}_1,X^{(1)}_2$) and $H^{(2)}$ (for $X^{(2)}_1,X^{(2)}_2$). Notice that $H^{(1)}$ and $H^{(2)}$ are iid. We compute
\begin{align*}
&\IP\big((X^{(1)}_1-X^{(2)}_1)\,(X^{(1)}_2-X^{(2)}_2\big)>0\big) \\
& \quad  = \IE\big[\IP\big((X^{(1)}_1-X^{(2)}_1)\,(X^{(1)}_2-X^{(2)}_2)>0\,\big|\,\mathcal{H}\big)\big]\\
 & \quad= \IE\big[\IP\big(X^{(1)}_1>X^{(2)}_1,X^{(1)}_2>X^{(2)}_2\,|\,\mathcal{H}\big)+\IP\big(X^{(1)}_1<X^{(2)}_1,X^{(1)}_2<X^{(2)}_2\,\big|\,\mathcal{H}\big) \big]\\
& \quad = \IE\Big[ \Big(\int H^{(2)}_{x-}\,\mathrm{d}H^{(1)}_x \Big)^2\Big]+\IE\Big[ \Big(\int H^{(1)}_{x-}\,\mathrm{d}H^{(2)}_x \Big)^2\Big]
\end{align*}
and analogously
\begin{align*}
&\IP\big((X^{(1)}_1-X^{(2)}_1)\,(X^{(1)}_2-X^{(2)}_2\big)<0\big)  \\
& \quad = \IE\big[\IP\big((X^{(1)}_1-X^{(2)}_1)\,(X^{(1)}_2-X^{(2)}_2)<0\,\big|\,\mathcal{H}\big)\big]\\
& \quad = 2\,\IE\Big[ \int H^{(2)}_{x-}\,\mathrm{d}H^{(1)}_x \,\int H^{(1)}_{x-}\,\mathrm{d}H^{(2)}_x \Big],
\end{align*}
so that Kendall's Tau equals
\begin{align*}
\IE\Big[\Big(\int H^{(2)}_{x-}\,\mathrm{d}H^{(1)}_x -\int H^{(1)}_{x-}\,\mathrm{d}H^{(2)}_x\Big)^2 \Big] \geq 0.
\end{align*}
\end{proof}

The next lemma is less intuitive on first glimpse, but like Lemmata \ref{lemma_nonegass} and \ref{lemma_kendall} it qualitatively states that laws in $\mathfrak{M}_{\ast}$ exhibit some sort of ``positive'' dependence. In order to understand it, it is useful to recall the notion of \emph{majorization}, see \cite{marshall79} for a textbook account on the topic. A vector $\bm{a}=(a_1,\ldots,a_d)$ is said to \emph{majorize} a vector $\bm{b}=(b_1,\ldots,b_d)$ if 
\begin{gather*}
\sum_{k=n}^{d}a_{[k]} \geq \sum_{k=n}^{d}b_{[k]},\quad n=2,\ldots,d,\quad \sum_{k=1}^{d}a_k=\sum_{k=1}^{d}b_k.
\end{gather*}
Intuitively, the entries of $\bm{b}$ are ``closer to each other'' than the entries of $\bm{a}$, even though the sum of all entries is identical for both vectors. For instance, the vector $(1,0,\ldots,0)$ majorizes the vector $(1/2,1/2,0,\ldots,0)$, which majorizes $(1/3,1/3,1/3,0,\ldots,0)$, and so on. 
\begin{lemma}[A link to majorization] \label{lemma_Shakedpos}
Consider $\bm{X}$ with law in $\mathfrak{M}_{\ast}$. Further, let $\bm{Y}=(Y_1,\ldots,Y_d)$ be a random vector with components that are iid and satisfy 
\begin{gather*}
Y_1 \stackrel{d}{=} X_1. 
\end{gather*}
We denote $F_Z(x):=\IP(Z \leq x)$ for real-valued $Z$ and $x \in \IR$. 
\begin{itemize}
\item[(a)] For arbitrary $x \in \IR$ the vector $\big(F_{Y_{[1]}}(x),\ldots,F_{Y_{[d]}}(x) \big)$ majorizes the vector $\big(F_{X_{[1]}}(x),\ldots,F_{X_{[d]}}(x) \big)$.
\item[(b)] For any measurable, real-valued function $g$ which is monotone on the support of $X_1$ the vector $\big(\IE[g(Y_{[1]})],\ldots,\IE[g(Y_{[d]})]\big)$ majorizes the vector $\big(\IE[g(X_{[1]})],\ldots,\IE[g(X_{[d]})] \big)$.
\end{itemize}
\end{lemma}
\begin{proof}
This is \cite[Theorem 2.2 and Corollary 2.3]{shaked77}. By definition,
\begin{gather*}
\sum_{k=1}^{d}F_{X_{[k]}}(x) = \sum_{k=1}^{d}F_{X_{k}}(x)=d\,F_{X_1}(x) = d\,F_{Y_1}(x) = \sum_{k=1}^{d}F_{Y_{k}}(x)=\sum_{k=1}^{d}F_{Y_{[k]}}(x).
\end{gather*}
Since $F_{X_{[1]}}(x) \geq \ldots \geq F_{X_{[d]}}(x)$ and $F_{Y_{[1]}}(x) \geq \ldots \geq F_{X_{[d]}}(x)$, for part (a) we have to show that
\begin{gather*}
\sum_{k=1}^{n}F_{X_{[k]}}(x) \leq \sum_{k=1}^{n}F_{Y_{[k]}}(x),\quad n=1,\ldots,d-1.
\end{gather*}
First, it is not difficult to verify that 
\begin{gather*}
h_{n,d}(p):=\sum_{k=1}^{n}\sum_{i=k}^{d}\binom{d}{i}\,p^{i}\,(1-p)^{d-i},\quad p \in [0,1],
\end{gather*}
is concave for arbitrary $1 \leq n \leq d$. Second, concavity implies that
\begin{align*}
\sum_{k=1}^{n}F_{X_{[k]}}(x) & = \IE\Big[\sum_{k=1}^{n} \IP(X_{[k]} \leq x\,\Big|\,\mathcal{H})\Big] = \IE[h_{n,d}(H_x)]\\
& \leq h_{n,d}(\IE[H_x]) =  h_{n,d}(\IP(Y_1 \leq x)) = \sum_{k=1}^{n}F_{Y_{[k]}}(x),
\end{align*}
where Jensen's inequality has been used. Making use of the relation $\IE[Z]=\int_0^{\infty}1-F_Z(z)\,\mathrm{d}z-\int_{-\infty}^{0}F_Z(z)\,\mathrm{d}z$ for real-valued random variables $Z$, part (b) is obtained from (a) for the case $g(x)=x$. For the general case, one simply has to observe that the law of $\big(g(X_1),\ldots,g(X_d) \big)$ is also in $\mathfrak{M}_{\ast}$ and due to monotonicity of $g$ we have either $g(X_{[1]}) \leq \ldots \leq g(X_{[d]})$ in the non-decreasing case or $g(X_{[d]}) \leq \ldots \leq g(X_{[1]})$ in the non-increasing case.
\end{proof}
Intuitively, statement (b) in case $g(x)=x$ states that the expected values of the order statistics $\IE[X_{[k]}]$, $k=1,\ldots,d$, are closer to each other than the respective values if the components of $\bm{X}$ were iid (and not only conditionally iid). Intuitively, the components of a random vector $\bm{X}$ with components that are conditionally iid are thus less spread out than the components of a random vector with iid components. Thus, Lemmata \ref{lemma_nonegass}, \ref{lemma_kendall} and \ref{lemma_Shakedpos} show that dependence models built from a conditionally iid setup can only capture the situation of components being ``more clustered'' than independence, which is loosely interpreted as ``positive dependence''. Generally speaking, negative dependence concepts are more complicated than positive dependence concepts in dimensions $d \geq 3$, the interested reader is referred to \cite{puccetti14} for a nice overview and references dealing with such concepts.
\par
Whereas Lemmata \ref{lemma_nonegass}, \ref{lemma_kendall} and \ref{lemma_Shakedpos} provide three particular quantifications for positive dependence of a conditionally iid probability law, many other possible concepts of positive dependence can be found in the literature, a textbook account on the topic is \cite{mueller02}. \cite[Theorem 4]{scarsini85} claims that if $\bm{X}=(X_1,\ldots,X_d)$ is conditionally iid and $x \mapsto \IP(X_1 \leq x)$ is continuous, then 
\begin{gather*}
\IP(\bm{X} \leq \bm{x}) \geq \prod_{k=1}^{d}\IP(X_k \leq x_k),\quad \bm{x} \in \IR^{d},
\end{gather*}
a positive dependence property called \emph{positive lower orthant dependency}. However, here is a counterexample showing that \cite[Theorem 4]{scarsini85} is not correct and conditionally iid random vectors need not exhibit positive lower orthant dependency in general. 
\begin{example}[Conditionally iid $\centernot \implies$ positive lower orthant dependency]\label{ex_scarsini}
Let $M$ be uniformly distributed on $[0,1/2]$. Conditioned on $M$ let $\bm{X}=(X_1,X_2)$ be a vector of two iid random variables which have distribution function
\begin{gather*}
H_t = \frac{1}{2}\,1_{\{-M+\frac{1}{2}\leq t < M+\frac{1}{2}\}}+1_{\{t \geq M+\frac{1}{2}\}},\quad t \in \IR.
\end{gather*}
It is not difficult to compute that 
\begin{gather*}
\IP(\bm{X} \leq \bm{x}) = \IE[H_{x_1}\,H_{x_2}]= \frac{1}{2}\,x_{[1]}+\frac{1}{2}\,\max\{0,x_1+x_2-1\},\quad  x_1,x_2 \in [0,1],
\end{gather*}
and the distribution function of $\bm{X}$ is a copula, i.e.\ has standard uniform one-dimensional marginals. In particular, 
\begin{gather*}
\IP\Big(X_1 \leq \frac{1}{4},X_2 \leq \frac{3}{4} \Big)=\frac{1}{8}<\frac{3}{16}=\IP\Big(X_1 \leq \frac{1}{4}\Big)\,\IP\Big(X_2 \leq \frac{3}{4}\Big),
\end{gather*}
contradicting positive lower orthant dependency. Notice that Kendall's Tau for $\bm{X}$ is exactly equal to zero, and also the correlation coefficient between the components of $\bm{X}$ equals zero. Figure \ref{fig:scarsini} depicts a scatter plot of $1000$ samples from $\bm{X}$.
\end{example}

In contrast to Example \ref{ex_scarsini}, \cite{dykstra73} prove that the weaker property $\IP(X_1 \in A,\ldots,X_d \in A) \geq \IP(X_1 \in A)^d$ holds indeed true for conditionally iid $\bm{X}$ and an arbitrary measurable set $A \subset \IR$. This makes clear that a decisive point in Example \ref{ex_scarsini} is that the considered $x_i$ are different.

\begin{figure}[!ht]
\caption{$1000$ samples of $(X_1,X_2)$ from Example \ref{ex_scarsini}.}
\includegraphics[width=10cm]{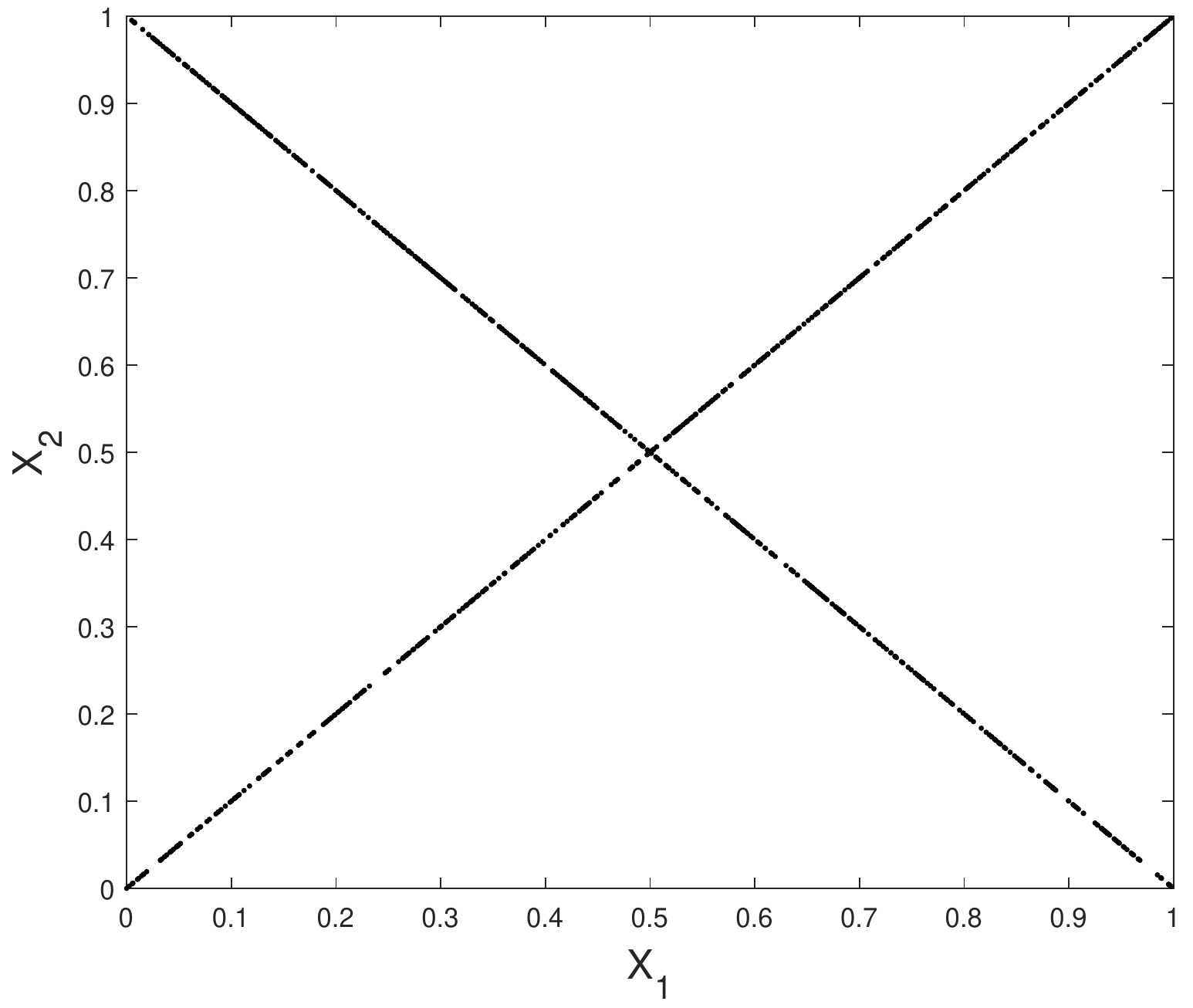} 
\label{fig:scarsini}
\end{figure}

\subsubsection{Further properties}
Even though it is obvious, we find it educational to point out explicitly that path continuity of $H$ corresponds to the absence of a singular component in the law of $\bm{X}$.
\begin{lemma}[Path continuity of $H$]\label{lemma_pathcont}
Let $H \sim \gamma \in M_+^1(\mathfrak{H})$ and consider the random vector $\bm{X}=(X_1,\ldots,X_d)$ constructed in Equation (\ref{canonical_construction}) for arbitrary $d \geq 2$.
Then $\IP(X_1=X_2)=0$ if and only if the paths of $H$ are almost surely continuous. 
\end{lemma}
\begin{proof}
Conditioned on the $\sigma$-algebra $\mathcal{H}$ generated by $H$, the random variables $X_1,X_2$ are iid with distribution function $H$. Since two iid random variables take exactly the same value with positive probability if and only if their common distribution function has at least one jump, the claim follows.
\end{proof}

The following result is shown in \cite[Proposition 4.2]{shaked77}, but we present a slightly different proof.
\begin{lemma}[Closure under convergence in distribution]\label{lemma_closureciid}
If $\bm{X}^{(n)}$ are conditionally iid and converge in distribution to $\bm{X}$, then the law of $\bm{X}$ is also conditionally iid.
\end{lemma}
\begin{proof}
Since we only deal with a statement in distribution, we are free to assume that each $\bm{X}^{(n)}$ is represented as in (\ref{canonical_construction}) from some stochastic process $H^{(n)}=\{H^{(n)}_t\}_{t \in \IR}$, and all objects are defined on the same probability space $(\Omega,\F,\IP)$. The random objects $H^{(n)}$ take values in the set of distribution functions of random variables taking values in $[-\infty,\infty]$. This set is compact by Helly's Selection Theorem and Hausdorff when equipped with the topology of pointwise convergence at all continuity points of the limit, see \cite{sibley71}. Thus, the probability measures on this set are a compact set by \cite[Corollary II.4.2, p.\ 104]{alfsen71}. This implies that we find a convergent subsequence $\{n_k\}_{k \in \IN} \subset \IN$ such that $H^{(n_k)}$ converges in distribution to some limiting stochastic process $H$, which takes itself values in the set of distribution functions of random variables taking values in $[-\infty,\infty]$. It is now not difficult to see that
\begin{align*}
\IP(X_1 \leq x_1,\ldots,X_d \leq  x_d)&=\lim_{k \rightarrow \infty}\IP(X^{(n_k)}_1 \leq x_1,\ldots,X^{(n_k)}_d \leq x_d)\\
 &=\lim_{k \rightarrow \infty}\IE[H^{(n_k)}_{x_1}\,\cdots \,H^{(n_k)}_{x_d}]=\IE[H_{x_1}\,\cdots \,H_{x_d}],
\end{align*}
where bounded convergence is used in the last equality. This implies that the law of $\bm{X}$ can be constructed canonically like in (\ref{canonical_construction}), hence $\bm{X}$ is conditionally iid. Finally, since $\bm{X}$ is assumed to take values in $\IR^d$, necessarily $H$ is almost surely the distribution function of a random variable taking values in $\IR$ (instead of $[-\infty,\infty]$).
\end{proof}

Recall that a random vector $(X_1,\ldots,X_d)$ is called \emph{radially symmetric} if there exist $\mu_1,\ldots,\mu_d \in \IR$ such that
\begin{gather*}
(X_1-\mu_1,\ldots,X_d-\mu_d) \stackrel{d}{=} (\mu_1-X_1,\ldots,\mu_d-X_d).
\end{gather*}
If $(X_1,\ldots,X_d)$ is constructed as in Equation (\ref{canonical_construction}), then radial symmetry can be translated into a symmetry property of the random distribution function $H$, which is the content of the following lemma. 

\begin{lemma}[Radial symmetry] \label{lemma_rs}
Let $H \sim \gamma \in M_+^1(\mathfrak{H})$. The random vector $(X_1,\ldots,X_d)$ constructed in Equation (\ref{canonical_construction}) is radially symmetric if and only if there is some $\mu \in \IR$ such that
\begin{gather*}
\{H_{\mu-t}\}_{t \in \IR} \stackrel{d}{=} \big\{ 1-H_{(t+\mu)-}\big\}_{t \in \IR}.
\end{gather*}
\end{lemma} 
\begin{proof}
On the one hand, we observe
\begin{align*}
& \IP(\mu-X_1 \leq x_1,\ldots,\mu-X_d \leq x_d) = \IP(H_{(\mu+x_1)-} \leq U_1,\ldots,H_{(\mu+x_d)-} \leq U_d)\\
& \qquad = \IE\Big[ (1-H_{(\mu+x_1)-})\,\cdots\,(1-H_{(\mu+x_d)-})\Big].
\end{align*}
On the other hand, we have 
\begin{gather*}
\IP(X_1-\mu \leq x_1,\ldots,X_d-\mu \leq x_d) = \IE[H_{\mu-x_1}\,\cdots\,H_{\mu-x_d}],
\end{gather*}
from where the claimed equivalence can now be deduced easily. Notice that the conditionally iid structure implies that $d$ can be chosen arbitrary and the law of $H$ is determined uniquely by the law of an infinite exchangeable sequence $\{X_k\}_{k \in \IN}$ constructed as in (\ref{canonical_construction}) with $d \rightarrow \infty$.
\end{proof}
\begin{example}[The multivariate normal law, again]\label{ex_normalcont}
The most prominent radially symmetric distribution is the multivariate normal law. Recalling Example \ref{ex_Normal}, it follows from (\ref{constr_exnormal}) that $\mathcal{N}(\bm{\mu},\Sigma)_{\ast}$, the conditionally iid normal laws, are induced by the stochastic process $\{H_t\}_{t \geq 0}$ given by 
\begin{gather}
H_t = \Phi\Bigg( \frac{\frac{t-\mu}{\sigma}-\sqrt{\rho}\,M}{\sqrt{1-\rho}}\Bigg),\quad t \in \IR,
\label{H_normal}
\end{gather}  
for some $\mu \in \IR$, $\sigma>0$, and $\rho \in [0,1]$, and a random variable $M \sim \Phi = $ distribution function of a standard normal law. The reader may check herself that this random distribution function $H$ satisfies the property of Lemma \ref{lemma_rs}.
\end{example}

An immediate but quite useful property of a conditionally iid model is the following corollary to the classical Glivenko-Cantelli Theorem.
\begin{lemma}[Conditional Glivenko-Cantelli] \label{lemma_condGC}
Let $\{X_k\}_{k \in \IN}$ be an infinite exchangeable sequence defined by the canonical construction (\ref{canonical_construction}) from an infinite iid sequence $\{U_k\}_{k \in \IN}$ and an independent random distribution function $H \sim \gamma  \in M_+^1(\mathfrak{H})$. It holds almost surely and uniformly in $t \in \IR$ that
\begin{gather*}
\frac{1}{d}\,\sum_{k=1}^{d}1_{\{X_k \leq t\}} \longrightarrow H_t,\quad \mbox{as }d \rightarrow \infty.
\end{gather*} 
\end{lemma}
\begin{proof}
Follows immediately from the classical Glivenko-Cantelli Theorem, which is applied in the second equality below:
\begin{align*}
& \IP \Big( \lim_{d \rightarrow \infty} \sup_{t \in \IR} \Big|\frac{1}{d}\,\sum_{k=1}^{d}1_{\{X_k \leq t\}} - H_t \Big|=0\Big) \\
& \qquad = \IE\Big[\IP \Big( \lim_{d \rightarrow \infty} \sup_{t \in \IR} \Big|\frac{1}{d}\,\sum_{k=1}^{d}1_{\{X_k \leq t\}} - H_t \Big| = 0 \,\Big|\,\mathcal{H}\Big) \Big]= \IE[1 ] =1.
\end{align*}
\end{proof}
The stochastic nature of the process $\{H_t\}_{t \in \IR}$ clearly determines the law of $\bm{X}$. Conversely, Lemma \ref{lemma_condGC} tells us that the law of the $d$-dimensional vector $\bm{X}$ does not determine the law of the underlying latent factor $\{H_t\}_{t \in \IR}$ in general, but accomplishes this in the limit as $d \rightarrow \infty$. Given some infinite exchangeable sequence of random variables $\{X_k\}_{k \in \IN}$, it shows how we can recover its latent random distribution function $H$. 
\par
A rather obvious property of the set $\mathfrak{M}_{\ast}$ is convexity.
\begin{lemma}[$\mathfrak{M}_{\ast}$ is convex with extreaml boundary the product measures]\label{lemma_convex}
If $\mu_1,\mu_2 \in \mathfrak{M}_{\ast}$ and $\epsilon \in (0,1)$, then $\epsilon\,\mu_1+(1-\epsilon)\,\mu_2 \in \mathfrak{M}_{\ast}$. Furthermore, if $\mu \in \mathfrak{M}_{\ast}$ is \emph{extremal}, meaning that $\mu = \epsilon\,\mu_1+(1-\epsilon)\,\mu_2$ for some $\epsilon \in (0,1)$ and $\mu_1,\mu_2 \in \mathfrak{M}_{\ast}$ necessarily implies $\mu=\mu_1=\mu_2$, then $\mu$ is a product measure\footnote{Meaning that the components of $\bm{X} \sim \mu$ are iid.}.
\end{lemma}
\begin{proof}
The convexity of $\mathfrak{M}_{\ast}$ is an immediate transfer from the (obvious) convexity of $M_+^1(\mathfrak{H})$ under the mapping $\Theta_d$, as the reader can readily check herself. That product measures are extremal is also obvious. Finally, consider an extremal element $\mu \in \mathfrak{M}_{\ast}$. Since $\mu$ is conditionally iid, there is a probability measure $\gamma \in M_+^{1}(\mathfrak{H})$ such that $\mu\big((-\bm{\infty},\bm{x}]\big) = \int_{\mathfrak{H}} h(x_1)\,\cdots\,h(x_d)\,\gamma(\mathrm{d}h)$. We choose a Borel set $A \in \mathfrak{H}$ with $\gamma(A)>0$. If $\gamma(A)=1$ is the only possible choice, $\gamma$ is actually a Dirac measure at some element $h \in \mathfrak{H}$ and $\mu$ is a product measure, as claimed. Let us derive a contradiction otherwise, in which case $\gamma = \gamma(A)\,\gamma(.\,|\,A)+\gamma(A^c)\,\gamma(.\,|\,A^c)$ and both $\gamma(.\,|\,A)$ and $\gamma(.\,|\,A^c)$ are elements of $M_+^1(\mathfrak{H})$. We obtain a convex combination of $\mu$, to wit
\begin{align*}
 \mu((-\bm{\infty},\bm{x}]) &= \gamma(A)\,\int_{\mathfrak{H}} h(x_1)\,\cdots\,h(x_d)\,\gamma(\mathrm{d}h\,|\,A)\\
& \qquad +(1-\gamma(A))\,\int_{\mathfrak{H}} h(x_1)\,\cdots\,h(x_d)\,\gamma(\mathrm{d}h\,|\,A^c).
\end{align*}
Since $\mu$ is extremal and $\gamma(.\,|\,A)$ and $\gamma(.\,|\,A^c)$ are different by definition, we obtain the desired contradiction.
\end{proof}
For the sake of completeness, the following remark gives two equivalent conditions for exchangeability of an infinite sequence of random variables.
\begin{remark}[Conditions equivalent to infinite exchangeability]
A result due to \cite{ryll57} states that an infinite sequence $\{X_k\}_{k \in \IN}$ of random variables is exchangeable (or, equivalently, conditionally iid by de Finetti's Theorem) if and only if the law of the infinite sequence $\{X_{n_k}\}_{k \in \IN}$ is invariant with respect to the choice of (increasing) subsequence $\{n_k\}_{k \in \IN} \subset \IN$. Another equivalent condition to exchangeability is that $\{X_k\}_{k \in \IN} \stackrel{d}{=}\{X_{\tau+k}\}_{k \in \IN}$ for an arbitrary finite stopping time $\tau$ with respect to the filtration $\F_n:=\sigma(X_1,\ldots,X_n)$, $n \in \IN$, see \cite{kallenberg82}. 
\end{remark}

\subsection{A general (abstract) solution to Problem \ref{motivatingproblem}}
\cite{konstantopoulos19} solve Problem \ref{motivatingproblem} on an abstract level for the whole family $\mathfrak{M}=M_+^1(\IR^d)$ of all probability laws on $\IR^d$. Their result is formulated in the next theorem in our notation\footnote{In addition to Theorem \ref{thm_probgen}, \cite{konstantopoulos19} even consider more abstract spaces than $\IR$, and also provide a necessary and sufficient criterion for finite extendibility of the law of $\bm{X}=(X_1,\ldots,X_d)$ to an exchangeable law on $\IR^n$ for $n>d$ arbitrary.}.
\begin{theorem}[General solution to Problem \ref{motivatingproblem}]\label{thm_probgen}
The law of $\bm{X}=(X_1,\ldots,X_d)$ is conditionally iid if and only if
\begin{gather*}
\sup_{g \neq 0}\Big\{ \frac{|\IE[g(\bm{X})]|}{\sup\limits_{\bm{Y}}|\IE[g(\bm{Y})]|}\Big\} \leq 1,
\end{gather*}
where the outer supremum is taken over all (non-zero) bounded, measurable functions $g:\IR^d \rightarrow \IR$, and the inner supremum in the denominator is taken over all random vectors $\bm{Y}=(Y_1,\ldots,Y_d)$ with iid components.
\end{theorem}
\begin{proof}
The proof of sufficiency is the difficult part, relying on functional analytic methods, and we refer the interested reader to \cite[Theorem 5.1]{konstantopoulos19}, but provide some intuition below. Necessity of the condition in Theorem \ref{thm_probgen} is the easy part, as will briefly be explained. Without loss of generality we may assume that $\bm{X}$ is represented by (\ref{canonical_construction}) with some stochastic process $H \in M_+^1(\mathfrak{H})$ and an independent sequence of iid variates $U_1,\ldots,U_d$ uniformly distributed on $[0,1]$. For arbitrary bounded and measurable $g$ we observe
\begin{align*}
|\IE[g(\bm{X})]| &= |\IE[\IE[g(\bm{X})\,|\,H]]| = |\IE[g(H^{-1}_{U_1},\ldots,H^{-1}_{U_d})]| \\
& \leq \sup_{G(.) \in \mathfrak{H}}\big|\IE\big[g \big(G^{-1}({U_1}),\ldots,G^{-1}({U_d})\big)\big]\big|=\sup_{\bm{Y}}|\IE[g(\bm{Y})]|.
\end{align*}
\end{proof}

Regarding the intuition of the sufficiency of the condition in Theorem \ref{thm_probgen}, we provide one demonstrating example. With $X$ standard normal, we have already seen in Example \ref{ex_Normal} that the random vector $\bm{X}=(X,-X)$ is not conditionally iid, since it is bivariate normal with negative correlation coefficient. So how does this random vector violate the condition? Considering the bounded measurable function $g(x_1,x_2)=1_{\{x_1<0<x_2\}}$, we readily observe that $\IE[g(\bm{X})]=\IP(X<0)=1/2$. If $\bm{Y}=(Y_1,Y_2)$ is an arbitrary vector with iid components, we observe that 
\begin{align*}
\IE[g(\bm{Y})]&=\underbrace{\IP(Y_1<0)}_{\leq \IP(Y_1 \leq 0)}\,\underbrace{\IP(Y_2>0)}_{=\IP(Y_1>0)} \leq \IP(Y_1 \leq 0)\,\big( 1-\IP(Y_1\leq 0)\big)\\
& \leq \sup_{p \in [0,1]}\{p\,(1-p)\}=1/4.
\end{align*}
Consequently, the supremum over all such $\bm{Y}$ is bounded from above by $1/4$, hence the supremum over all $g$ in the condition of Theorem \ref{thm_probgen} is at least two, hence larger than one. The intuition behind this counterexample is that we have found one particular bounded measurable $g$ that addresses a distributional property of $\bm{X}$ that sets it apart from any iid sequence. Indeed, the proof of \cite{konstantopoulos19} relies on the Hahn-Banach Theorem and thus on a separation argument, since the set of conditionally iid laws can be viewed as a closed convex subset of $M_+^1(\IR^d)$ with extremal boundary comprising the laws with iid components, see Lemma \ref{lemma_convex}.
\par
On the one hand, Theorem \ref{thm_probgen} is clearly a milestone with regards to the present survey as it solves Problem \ref{motivatingproblem} in the general case. On the other hand, it is difficult to apply the derived condition in particular cases of Problem \ref{motivatingproblem}, when the family $\mathfrak{M}$ is some (semi-)parametric family of interest - simply because the involved suprema are hard to evaluate, see also Example \ref{ex_NP} below. On a high level, Theorem \ref{thm_probgen} solves Problem \ref{motivatingproblem} but not the refined Problem \ref{motivatingproblemrefined}, which depends on an additional dimension-independent property (P). However, the most compelling results of the theory deal precisely with certain dimension-independent properties (P) of interest, see the upcoming sections as well as paragraph \ref{subsec_effectext} for a further discussion. This is because the additional structure provided by some property (P) and the search for structure-preserving extensions is in many cases a more natural and more interesting problem than to simply find \emph{some} extension. We will see that the algebraic structure of this problem is highly case-specific in general, i.e.\ heavily dependent on (P). 
\par
The following example shows that the supremum condition of Theorem \ref{thm_probgen} can lead to an NP-hard problem in general.
\begin{example}[In general, the extendibility problem is difficult]\label{ex_NP}
If $\bm{X}=(X_1,X_2)$ is a random vector taking values in $\{x_1,\ldots,x_n\}^2 \subset \IR^2$, its joint probability distribution is fully described in terms of the matrix $A \in [0,1]^{n \times n}$ defined via $A_{ij}:=\IP(X_1=x_i,X_2=x_j)$, $1 \leq i,j \leq n$. The probability law of $\bm{X}$ is exchangeable if and only if $A=A^T$, and the law of $\bm{X}$ is conditionally iid if and only if\footnote{We denote by $S_m:=\{\bm{y} \in [0,1]^m\,:\,\norm{\bm{y}}_1=1\}$ the $m$-dimensional unit simplex.} there are (row vectors) $\bm{\lambda} \in S_m$ and $\bm{x}_1,\ldots,\bm{x}_m \in S_n$ such that $A=\sum_{i=1}^m \lambda_i\,\bm{x}_i^T\,\bm{x}_i$. Up to normalization, which is only due to the fact that we deal with a probabilistic interpretation, this property is called \emph{complete positivity}. A completely positive matrix $A$ is necessarily also \emph{doubly non-negative}, meaning that it is symmetric, element-wise non-negative and positive semi-definite, and its elements sum up to one. The set of completely positive matrices is a proper subset of doubly non-negative matrices in dimensions $d \geq 5$, and to decide for a given matrix $A$ whether or not it is completely positive is known to be NP-hard, see \cite{dickinson14}. Theorem \ref{thm_probgen} implies that $\bm{X}$, given in terms of $A$, is conditionally iid if and only if
\begin{gather*}
\sup_{G \in \IR^{n \times n}\setminus \{\bm{0}\}}\Big\{ \frac{\big|\sum_{i,j=1}^{n}G_{ij}\,A_{ij}\big|}{\sup_{\bm{y} \in S_n}|\bm{y}^T\,G\,\bm{y}|}\Big\} \leq 1.
\end{gather*}
Notice that the denominator is equal to the absolute value of the maximal eigenvalue of $G$, the so-called spectral radius of $G$. As outlined before, this optimization problem must be NP-hard, unless P=NP. 
\end{example}

\newpage

\section{Binary sequences} \label{sec:binary}

We study probability laws on $\{0,1\}^d$, i.e.\ on the set of finite binary sequences. We start with a short digression on the little moment problem, because it occupies a commanding role, not only in this section but also in Section \ref{sec:lom} below. For a further discussion between the little moment problem and de Finetti's Theorem, the interested reader is also referred to \cite{daboni82}.

\subsection{Hausdorff's moment problem}\label{subsec_hausdorff}
If $(b_0,\ldots,b_d)$ is a finite sequence of real numbers, we write $\nabla b_k = b_k-b_{k+1}$ for $k=0,\ldots,d-1$. The (reversed) difference operator $\nabla$ may be iterated, yielding $\nabla^2 b_k = \nabla(\nabla b_k)=\nabla b_k -\nabla b_{k+1}$ for $k=0,\ldots,d-2$, and so on. In general we obtain the formula
\begin{gather*}
\nabla^j b_k := \sum_{i=0}^{j}(-1)^{i}\binom{j}{i}\,b_{k+i},\quad 0 \leq  j+k \leq d,
\end{gather*}
with $\nabla=\nabla^1$ and $\nabla^0$ the identity.

\begin{definition}[$d$-monotone sequences]
For $d \in \IN$, we say that a finite sequence $(b_0,b_1,\ldots,b_{d}) \in [0,\infty)^{d+1}$ is \emph{$d$-monotone} if $\nabla^{d-k}b_k \geq 0$ for $k=0,1,\ldots,d$. An infinite sequence $\{b_k\}_{k \in \IN_0}$ with positive members is said to be \emph{completely monotone} if $(b_0,\ldots,b_{d})$ is $d$-monotone for each $d \geq 2$.
\end{definition}

If $(b_0,\ldots,b_{d})$ is $d$-monotone, then $\nabla^j b_k \geq 0$ for all $0 \leq j+k \leq d$. In particular, if for $d \geq 2$ the sequence $(b_0,\ldots,b_{d})$ is $d$-monotone, then the shorter sequences $(b_0,\ldots,b_{d-1})$ and $(b_1,\ldots,b_{d})$ are both $(d-1)$-monotone. Intuitively, when viewing $(b_0,\ldots,b_{d})$ as a function $\{0,\ldots,d\} \rightarrow [0,\infty)$, then $(-1)^j\,\nabla^{j}b_k$ is something like the $j$-th derivative at $k$. With this interpretation in mind, $d$-monotonicity means that the higher-order derivatives alternate in sign, i.e.\ first derivative is non-positive, second derivative is non-negative, third derivative is non-positive, and so on. For instance, a $2$-monotone sequence is non-increasing $(b_k \geq b_{k+1})$ and ``convex'' ($b_{k+1}$ is smaller or equal than the arithmetic mean of its neighbors $b_k$ and $b_{k+2}$). The set of all $d$-monotone sequences starting with $b_0=1$ will be denoted by $\mathcal{M}_d$ in the sequel. Similarly, $\mathcal{M}_{\infty}$ denotes the set of completely monotone sequences starting with $b_0=1$. 
\par
Finite sequences in $\mathcal{M}_{d}$ arise quite naturally in the context of certain discrete probability laws, as will briefly be explained. Consider a probability distribution on the power set (including the empty set) of $\{1,\ldots,d\}$ with the property that subsets with the same cardinality are equally likely outcomes. Concretely, the probability of some subset $I \subset \{1,\ldots,d\}$ only depends on the cardinality $|I|$ of $I$, and there are only $d+1$ possible cardinalities. Denote the probability of a subset with cardinality $k$ by $p_k$, $k=0,\ldots,d$. Then $p_0,\ldots,p_d$ are non-negative numbers satisfying
\begin{gather}
\sum_{k=0}^{d}\binom{d}{k}\,p_k = \sum_{I \subset \{1,\ldots,d\}}p_{|I|} = 1.
\label{posnumconddmon}
\end{gather}
Defining the sequence 
\begin{gather}
b_k:=\sum_{i=0}^{d-k}\binom{d-k}{i}p_{d-i},\quad k=0,\ldots,d,
\label{constr_dmon_frompos}
\end{gather}
it follows that $\nabla^{d-k}b_k = p_k \geq 0$ for $k=0,\ldots,d$. In particular, $b_0=1$, so $(b_0,\ldots,b_d) \in \mathcal{M}_d$. Furthermore, the construction (\ref{constr_dmon_frompos}) can be inverted, i.e.\ is general enough to construct all elements of $\mathcal{M}_d$. To wit, if $(b_0,\ldots,b_d)$ is an arbitrary element in $\mathcal{M}_d$, then the vector of non-negative numbers $(p_0,\ldots,p_d)=(\nabla^d b_0,\nabla^{d-1}b_1,\ldots,\nabla^{0}b_d)$ satisfies (\ref{posnumconddmon}), i.e.\ defines a probability law on the power set of $\{1,\ldots,d\}$ with the aforementioned property. Thus, these probability laws on the power set of $\{1,\ldots,d\}$ and $\mathcal{M}_d$ stand in a one-to-one correspondence. Of course, the power set of $\{1,\ldots,d\}$ can naturally be identified with $\{0,1\}^d$, when identifying $\bm{x} \in \{0,1\}^d$ with the subset $I=\{k\,:\,x_k=1\}$. This explains the occurrence of $d$-monotonicity in the present section. 
\par
The so-called \emph{Hausdorff moment problem} (also known as \emph{little moment problem}) states that the sequences $\mathcal{M}_{\infty}$ stand in one-to-one correspondence with the moment sequences of random variables taking values on the unit interval $[0,1]$. Concretely, the sequence $\{b_k\}_{k \in \IN_0}$ with $b_0=1$ is completely monotone if and only if there is a random variable $M$ taking values in $[0,1]$ such that $b_k=\IE[M^k]$, $k \in \IN_0$. Furthermore, the sequence $\{b_k\}_{k \in \IN_0}$ uniquely determines the probability law of $M$. This result is originally due to \cite{hausdorff21,hausdorff23}. See also \cite[p.\ 225]{feller66} for a proof. Uniqueness of the probability law of $M$ relies heavily on the boundedness of the interval $[0,1]$ and is due to the fact that polynomials are dense in the space of continuous functions on a bounded interval (Stone-Weierstrass).
\par
It is important to observe that not every $d$-monotone sequence can be extended to a completely monotone sequence. Being given a $d$-monotone sequence $(b_0,\ldots,b_{d})$, to check whether there exists an extension $b_{d+1},b_{d+2},\ldots$ to an infinite completely monotone sequence $\{b_k\}_{k \in \IN_0}$ is a purely analytical, highly non-trivial problem, and luckily already solved. This problem is known as the \emph{truncated Hausdorff moment problem}. Its solution, due to \cite{karlin53}, states that $(b_0,\ldots,b_{d})$ with $b_0=1$ can be extended to an element in $\mathcal{M}_{\infty}$ if and only if the \emph{Hankel determinants} $\hat{H}_1,\,\check{H}_1,\ldots,\hat{H}_{d-1},\,\check{H}_{d-1}$ are all non-negative, which are defined as
\begin{align}
\hat{H}_{2\,\ell}&:=\mbox{det}\left[ \begin{array}{ccc}
b_0 & \dots & b_{\ell}\\
\vdots & & \vdots\\
b_{\ell} & \dots & b_{2\,\ell}\\
\end{array}\right],\quad \check{H}_{2\,\ell}:=\mbox{det}\left[ \begin{array}{ccc}
\nabla b_1 & \dots & \nabla b_{\ell}\\
\vdots & & \vdots\\
\nabla b_{\ell} & \dots & \nabla b_{2\,\ell-1}\\
\end{array}\right],\nonumber\\
\hat{H}_{2\,\ell+1}&:=\mbox{det}\left[ \begin{array}{ccc}
b_1 & \dots & b_{\ell+1}\\
\vdots & & \vdots\\
b_{\ell+1} & \dots & b_{2\,\ell+1}\\
\end{array}\right],\quad \check{H}_{2\,\ell+1}:=\mbox{det}\left[ \begin{array}{ccc}
\nabla b_0 & \dots & \nabla b_{\ell}\\
\vdots & & \vdots\\
\nabla b_{\ell} & \dots & \nabla b_{2\,\ell}\\
\end{array}\right],
\label{hankeldet}
\end{align}
for all $\ell \in \IN_0$ with $2\,\ell \leq d$, respectively $2\,\ell+1 \leq d$. To provide an example, the sequence $(1,1/2,\epsilon)$ is $2$-monotone for all $\epsilon \in [0,1/2]$, but can only be extended to a completely monotone sequence if $\epsilon \in [1/4,1/2]$.

\subsection{Extendibility of exchangeable binary sequences}
Actually, before Bruno de Finetti published his seminal Theorem \ref{thm_definetti} in 1937, he first published in \cite{definetti31} the same result for the simpler case of binary sequences. In fact, he showed that there is a one-to-one correspondence between exchangeable probability laws on infinite binary sequences and the set $M_+^1([0,1])$ of probability laws on $[0,1]$. 
\par
We start with a random vector $\bm{X}=(X_1,\ldots,X_d)$ taking values in $\{0,1\}^d$. We know from Lemma \ref{lemma_exchangeable} that $\bm{X}$ needs to be exchangeable in order to possibly be conditionally iid, so we concentrate on the exchangeable case. Let $\bm{1}_m$, $\bm{0}_m$ denote $m$-dimensional row vectors with all entries equal to one and zero, respectively, and define 
\begin{gather*}
p_k:=\IP\big(\bm{X}=(\bm{1}_k,\bm{0}_{d-k}) \big),\quad k =0,\ldots,d.
\end{gather*}
Exchangeability implies that $\IP(\bm{X}=\bm{x})=p_{\norm{\bm{x}}_1}$ for arbitrary $\bm{x} \in \{0,1\}^d$. Consequently, the probability law of $\bm{X}$ is fully determined by $p_0,\ldots,p_d$. 

\begin{theorem}[Extendibility of exchangeable binary sequences]\label{thm_definetti01}
Let $\bm{X}$ be an exchangeable random vector taking values in $\{0,1\}^d$. We denote
\begin{gather*}
p_k:=\IP\big(\bm{X}=(\bm{1}_k,\bm{0}_{d-k}) \big),\quad k =0,\ldots,d.
\end{gather*}
The following statements are equivalent:
\begin{itemize}
\item[(a)] $\bm{X}$ is conditionally iid.
\item[(b)] There is a random variable $M$ taking values in $[0,1]$ such that
\begin{gather*}
p_k = \nabla^{d-k} b_{k},\quad k=0,\ldots,d,
\end{gather*}
where $b_k:=\IE[M^k]$ for $k =0,\ldots,d$. 
\item[(c)] The Hankel determinants in (\ref{hankeldet}) are all non-negative, for all $\ell \in \IN_0$ with $2\,\ell \leq d$, respectively $2\,\ell+1 \leq d$, where
\begin{gather*}
b_k:=\sum_{i=0}^{d-k}\binom{d-k}{i}\,p_{d-i},\quad k=0,\ldots,d.
\end{gather*}
\end{itemize}
If one (hence all) of these conditions are satisfied, and $\bm{U}=(U_1,\ldots,U_d)$ is an iid sequence of random variables that are uniformly distributed on $[0,1]$, independent of $M$ in part (b), then
\begin{gather*}
\bm{X} \stackrel{d}{=} ( 1_{\{U_1 \leq M\}},\ldots,1_{\{U_d \leq M\}}).
\end{gather*} 
\end{theorem}
\begin{proof}
The equivalence of (c) and (b) relies on the truncated Hausdorff moment problem and the identities
\begin{gather*}
\IE[M^k\,(1-M)^{d-k}] = \nabla^{d-k}b_k = p_k, \quad k=0,\ldots,d,
\end{gather*}
which are all readily verified. To show that (b) implies (a) works precisely along the stochastic model with $\bm{U}$ as claimed, which is easily checked. To verify the essential part (a) $\implies$ (b) we may simply apply de Finetti's Theorem \ref{thm_definetti} in the special case of a binary sequence\footnote{Alternatively, one may construct a completely monotone sequence $\{b_k\}_{k \in \IN}$ from an infinite extension of $\bm{X}$, as demonstrated in \cite[Equation (1)]{daboni82}, and then make use of Hausdorff's moment problem to obtain $M$.}: (a) implies that we may without loss of generality assume that the given random vector equals the first $d$ members of an infinite exchangeable binary sequence $\{X_k\}_{k \in \IN}$. De Finetti's Theorem \ref{thm_definetti}, and as a corollary Lemma \ref{lemma_condGC}, give us a random variable $H \sim \gamma \in M_+^1(\mathfrak{H})$. Since each $X_k$ takes values only in $\{0,1\}$, necessarily almost every path of $H$ has only one value different from $\{0,1\}$, which is $H_{t}$ for $t \in [0,1)$. So we define $M:=1-H_{1/2}$ and observe that conditioned on $M$, the random variables $X_k$ are iid Bernoulli with success probability $M$. This implies the claim.
\end{proof}

In words, the canonical stochastic model for conditionally iid $\bm{X}$ with values in $\{0,1\}^d$ is a sequence of $d$ independent coin tosses with success probability $M$ which is identical for all coin tosses, but simulated once before the first coin toss. We end this section with two examples of particular interest.

\begin{example}[P\'olya's urn]
Let $r \in \IN$ and $b \in \IN$ denote the numbers of red and blue balls in an urn. Define a random vector $\bm{X} \in \{0,1\}^d$ as follows:
\begin{itemize}
\item[(i)] Set $k:=1$.
\item[(ii)] Draw a ball at random from the urn.
\item[(iii)] Set $X_k:=1$ if the ball is red, and $X_k:=0$ otherwise.
\item[(iv)] Put the ball back into the urn with $1$ additional ball of the same color. 
\item[(v)] Increment $k:=k+1$.
\item[(vi)] If $k=d+1$, stop, otherwise go to step (ii).
\end{itemize}
It is not difficult to observe that $\bm{X}$ is exchangeable, since
\begin{gather*}
\IP(\bm{X}=\bm{x}) = \frac{\prod_{k=0}^{\norm{\bm{x}}_1-1}(r+k)\,\prod_{k=0}^{d-\norm{\bm{x}}_1-1}(b+k)}{\prod_{k=0}^{d-1}(r+b+k)},\quad \bm{x} \in \{0,1\}^d,
\end{gather*}
depends on $\bm{x}$ only through $\norm{\bm{x}}_1$. Like in Theorem \ref{thm_definetti01} we denote by $p_k$ the probability $\IP(\bm{X}=\bm{x})$ if $\norm{\bm{x}}_1=k$, $k=0,\ldots,d$. Using induction over $k=d,d-1,\ldots,0$ in order to verify ($\ast$) below and knowledge about the moments of the Beta-distribution\footnote{See, e.g., \cite[p.\ 35]{gupta04}.} in ($\ast\ast$) below, we observe that
\begin{align*}
b_k &:=\sum_{i=0}^{d-k}\binom{d-k}{i}\,p_{d-i} = \sum_{i=0}^{d-k}\binom{d-k}{i}\,\frac{(r+b-1)!\,(r+d-i-1)!\,(b+i-1)!}{(r-1)!\,(b-1)!\,(r+b+d-1)!}\\
& \stackrel{(\ast)}{=} \frac{(r+k-1)!\,(r+b-1)!}{(r-1)!\,(b+r+k-1)!} = \frac{\Gamma(r+k)\,\Gamma(r+b)}{\Gamma(r)\,\Gamma(r+b+k)}\stackrel{(\ast\ast)}{=}\IE[M^k],
\end{align*}
where $M$ is a random variable with Beta-distribution whose density is given by
\begin{gather*}
f_{M}(x) = \frac{\Gamma(r+b)}{\Gamma(r)\,\Gamma(b)}\,x^{r-1}\,(1-x)^{b-1},\quad 0<x<1.
\end{gather*}
Thus, the probability law of $\bm{X}$ has a conditionally iid representation like in Theorem \ref{thm_definetti01}. This is one of the traditional examples, in which the conditionally iid structure is a priori not easy to guess from the original motivation of $\bm{X}$ - in this case a simple urn replacement model. 
\end{example}

\begin{example}[Ferromagnetic Curie-Weiss Ising model]
Motivated by several models in statistical mechanics, \cite{liggett07} study random vectors which admit a density with respect to the law of a vector with iid components which is the exponential of a quadratic form. Concretely, they consider the situation
\begin{gather}
\label{liggettmodel}
\IP(\bm{X} \in \mathrm{d}\bm{x}) = \frac{1}{c_d}\,e^{\frac{1}{2}\big(\sum_{k=1}^{d}x_k\big)^2}\,\IP(\bm{Y} \in \mathrm{d}\bm{x}),
\end{gather}
where $\bm{Y}=(Y_1,\ldots,Y_d)$ is a vector with iid components and $Y_1$ is assumed to satisfy
\begin{align}
\psi(v):=\IE\Big[ e^{v\,Y_1}\Big]<\infty\mbox{ for all }v \in \IR,\quad c_d:=\IE\Big[ e^{\frac{1}{2}\big( \sum_{k=1}^{d}Y_k\big)^2}\Big]<\infty.
\label{liggett_cond}
\end{align}
Of particular interest are cases in which $Y_1$ takes only finitely many different values. Especially if $Y_1 \in \{0,1\}$, the vector $\bm{X}$ is a binary sequence like in the present section. 
\par
A prominent model motivating the investigation of \cite{liggett07} is the so-called Curie-Weiss Ising model. In probabilistic terms, this model is a probability law on $\{-1,1\}^d$ with two parameters $J,h \in \IR$, and the components of a random vector $\bm{Z}$ with this probability law models the so-called spins at $d$ different sites. These spins can either have the value $-1$ or $1$ (so $\bm{X}:=(1_{\{Z_1>0\}},\ldots,1_{\{Z_d>0\}})$ is a transformation from $\{-1,1\}^d$ to $\{0,1\}^d$). We denote for $\bm{n} \in \{-1,1\}^d$ by $N(\bm{n})$ the number of $1$'s in $\bm{n}$, so that $d-N(\bm{n})$ equals the number of $-1$'s. For $\bm{n} \in \{-1,1\}^d$ we define
\begin{gather}
\IP(\bm{Z}=\bm{n}) = \frac{e^{h\,\big(2\,N(\bm{n})-d\big)+\frac{J}{2}\,\big(2\,N(\bm{n})-d\big)^2}}{\sum_{k=0}^{d}\binom{d}{k}e^{h\,(2\,k-d)+\frac{J}{2}\,(2\,k-d)^2}},\quad \bm{n} \in \{-1,1\}^{d},
\label{CurieIsing}
\end{gather}
which is an exchangeable probability law on $\{-1,1\}^d$. The exponent of the numerator can be re-written as 
\begin{gather*}
h\,\big(2\,N(\bm{n})-d\big)+\frac{J}{2}\,\big(2\,N(\bm{n})-d\big)^2 = h\,\sum_{k=1}^{d}n_k+\frac{J}{2}\,\sum_{k=1}^{d}\sum_{i=1}^{d}n_k\,n_i
\end{gather*}
and is called the \emph{Hamilton operator} of the model. The parameter $h$ determines the external magnetic field and the parameter $J$ denotes a coupling constant. If $J \geq 0$ the model is called ferromagnetic, and for $J<0$ it is called antiferromagnetic. The ferromagnetic case arises as special case of (\ref{liggettmodel}), if $Y_1$ takes values in $\{-\sqrt{J},\sqrt{J}\}$ with respective probabilities $\IP(Y_1=\sqrt{J})=1-\IP(Y_1=-\sqrt{J})=\exp(h)/(\exp(h)+\exp(-h))$. Then the law of $\bm{Z}/\sqrt{J}$ on $\{-1,1\}^d$ is precisely given by the Curie-Weiss Ising model in (\ref{CurieIsing}) with $J \geq 0$. Notice that for the antiferromagnetic case $J<0$ this construction is impossible.
\par
\cite[Theorem 1.2]{liggett07} shows that $\bm{X}$ as defined in (\ref{liggettmodel}) is conditionally iid. More concretely, conditioned on a random variable $M$ with density\footnote{Completing the square shows that $f_{M}$ defines a proper density function on $\IR$.}
\begin{gather*}
f_{M}(v):=\frac{\psi(x)}{c_d}\,\frac{e^{-\frac{x^2}{2}}}{\sqrt{2\,\pi}},\quad x \in \IR,
\end{gather*}
the components of $\bm{X}$ are iid with common distribution
\begin{gather*}
\IP(X_k \in \mathrm{d}x\,|\,M) = \frac{e^{M\,x}}{\psi(M)}\,\IP(Y_1 \in \mathrm{d}x),\quad k=1,\ldots,d,
\end{gather*}
as can easily be checked. In particular, this shows that the aforementioned ferromagnetic Curie-Weiss Ising model is conditionally iid, a result originally due to \cite{papangelou89}. 
\end{example}

\newpage

\section{Classical results for static factor models} \label{sec:static}
Besides the seminal de Finetti's Theorem \ref{thm_definetti}, the most popular results in the theory on conditionally iid models concern latent factor processes $H$ of a very special form to be discussed in the present section. To this end, we consider a popular one-parametric family of one-dimensional distribution functions $x \mapsto F_{m}(x)$ on the real line and put a prior distribution on the parameter $m \in \IR$. Then define $H=\{H_t\}_{t \in \IR}$ in the canonical construction (\ref{canonical_construction}) by $H_t=F_{M}(t)$, where $M$ is some random variable taking values in the set of admissible values for the parameter $m$. For some prominent families, for example the zero mean normal law or the exponential law, the resulting distribution of the random vector $\bm{X}$ belongs to a prominent multivariate family of distributions $\mathfrak{M}$, and in fact defines the subset $\mathfrak{M}_{\ast} \subset \mathfrak{M}$. Of particular interest is the case when the subset $\mathfrak{M}_{\ast}$ of $\mathfrak{M}$ admits a convenient analytical description within the framework of the analytical description of the larger family $\mathfrak{M}$. By construction, in this method of generating conditionally iid laws the dependence-inducing latent factor process $H$ is fully determined already by a single random parameter $M$, so that it appears unnatural to formulate the model in terms of a ``stochastic process'' $H$ at all. Since we investigate situations in which this appears to be more natural in later sections, we purposely do this anyway in order to present all results of the present article under one common umbrella. The ``single-parameter construction'' just described can then be classified as some kind of ``static'' process within the realm of all possible processes with laws in $M_+^1(\mathfrak{H})$.
 \par
More rigorously, let $\{H_t\}_{t \geq 0}$ be the stochastic process from the canonical stochastic representation (\ref{canonical_construction}) of some multivariate law in $\mathfrak{M}_{\ast} \subset \mathfrak{M}$. Equivalently, we view this probability law as a $d$-dimensional marginal law of some infinite exchangeable sequence of random variables $\{X_k\}_{k \in \IN}$, and define $\{H_t\}_{t \geq 0}$ according to Lemma \ref{lemma_condGC} as the uniform limit of $\big\{\sum_{k=1}^{d}1_{\{X_k \leq \,t\}}/d\big\}_{t \geq 0}$ as $d \rightarrow \infty$. We call the probability law of $\bm{X}=(X_1,\ldots,X_d)$ \emph{static}, if the natural filtration generated by $\{H_t\}_{t \geq 0}$, i.e.\ $\mathcal{H}_t:=\sigma(H_s\,|\,s \leq t)$, $t \in \IR$, is trivial, meaning that there is some $T \in [-\infty,\infty)$ such that $\mathcal{H}_t = \{\emptyset,\Omega\}$ for $t \leq T$ (``zero information before $T$'') and $\mathcal{H}_t=\mathcal{H}$ for $t > T$ (``total information after $T$''). The present section reviews well-known families of distributions $\mathfrak{M}$, for which the set $\mathfrak{M}_{\ast}$ consists only of static laws. As already mentioned, this situation typically occurs when the random distribution function $H \sim \gamma \in M_+^1(\mathfrak{H})$ is itself given by $H_t=F_{M}(t)$, for a popular family $F_m$ of one-dimensional distribution functions and a single random variable $M$ representing a random parameter pick.  
\begin{example}[The multivariate normal law revisited]
It follows from Examples \ref{ex_Normal} and \ref{ex_normalcont} that $\mathcal{N}(\bm{\mu},\Sigma)_{\ast}$, the conditionally iid normal laws, are static. The random distribution function $H$ as given by (\ref{H_normal}) obviously satisfies $\mathcal{H}=\sigma(H_t\,:\,t \in \IR)=\sigma(M)=\mathcal{H}_t$ for arbitrary $t \in \IR$.
\end{example}

\begin{example}[Binary sequences revisited]
If one (hence all) of the conditions of Theorem \ref{thm_definetti01} is satisfied, the law of the binary sequence $\bm{X} \in \{0,1\}^d$ is static. Using the notation in Theorem \ref{thm_definetti01}, the random distribution function $H$ equals $H_t:=(1-M)\,1_{\{t \geq 0\}}+M\,1_{\{t \geq 1\}}$. Obviously, $\mathcal{H}=\sigma(H_t\,:\,t \in \IR)=\sigma(M)=\mathcal{H}_t$ for arbitrary $t>0$.
\end{example}

In the remaining section we treat the mixture of iid zero mean normals in paragraph \ref{subsec_sphere} and the mixture of iid exponentials in paragraph \ref{subsec_AC}, since these are the best-studied cases of the theory with nice analytical characterizations. The interested reader is also referred to \cite{diaconis87,rachev91} who additionally study mixtures of iid geometric variables, iid Poisson variables, and iid uniform variables. Mixtures of uniform random variables are discussed in more detail also in Section \ref{subsec_gnedin} below.

\subsection{Spherical laws (aka $\ell_2$-norm symmetric laws)}\label{subsec_sphere}
A random vector $\bm{X} \in \IR^d$ is called \emph{spherical} if its probability distribution remains invariant under unitary transformations, such as rotations or reflections, i.e.\ $\bm{X} \stackrel{d}{=} \bm{X}\,O$ for an arbitrary orthogonal matrix $O \in \IR^{d \times d}$. A spherical random vector $\bm{X}$ has a canonical stochastic representation
\begin{gather}
\bm{X} \stackrel{d}{=} R\,\bm{S},
\label{representation_spherical}
\end{gather}
where $R$ is a non-negative random variable and the random vector $\bm{S}$ is independent of $R$ and uniformly distributed on the Euclidean unit sphere $\{\bm{x} \in \IR^d\,:\,\norm{\bm{x}}_2=1\}$, see \cite[Chapter 2]{fang90}. Hence, realizations of spherical laws must be thought of as being the result of a two-step simulation algorithm: first draw one completely random point on the unit $d$-sphere, and then scale this point according to some one-dimensional probability distribution on the positive half-axis. In analytical terms, spherical laws are most conveniently treated via their (multivariate) characteristic functions. In particular, it is not difficult to see that $\bm{X}$ has a spherical law if and only if there exists a real-valued function $\varphi:[0,\infty) \rightarrow \IR$ in one variable such that 
\begin{gather*}
\IE\Big[e^{\mathrm{i}\,(u_1\,X_1+\ldots+u_d\,X_d)} \Big] = \varphi(\norm{\bm{u}}_2^2),\quad \bm{u}=(u_1,\ldots,u_d) \in \IR^d,
\end{gather*}
see, e.g., \cite[Lemma 4.1, p.\ 161]{mai12}. The function $\varphi$ is called the \emph{characteristic generator}. If the components of $\bm{X}$ are conditionally iid, the function $\varphi$ is of a very special form, see Schoenberg's Theorem \ref{schoenberg_thm} below.
\par
If the components of $\bm{Y}=(Y_1,\ldots,Y_d)$ are iid standard normally distributed, and $M \in (0,\infty)$ is an independent random variable, the random vector $\bm{X}=M\,\bm{Y}$ is spherical, because $\bm{Y}\,O$ is a vector of iid standard normal components for any orthogonal matrix $O$. Furthermore, the components of $\bm{X}$ are iid conditioned on the $\sigma$-algebra generated by the mixture variable $M$. Schoenberg's Theorem states that the converse is true as well, i.e.\ all conditionally iid spherical laws are mixtures of zero-mean normals.
\begin{theorem}[Schoenberg's Theorem] \label{schoenberg_thm}
Let $\mathfrak{M}$ be the family of $d$-dimensional spherical laws, and let the law of $\bm{X}$ be in $\mathfrak{M}$, and assume $\bm{X}$ is not identically equal to a vector of zeros. The following are equivalent 
\begin{itemize}
\item[(a)] The law of $\bm{X}$ lies in $\mathfrak{M}_{\ast}$. 
\item[(b)] There are iid standard normal random variables $Y_1,\ldots,Y_d$ and an independent positive random variable $M \in (0,\infty)$ such that
\begin{gather*}
\bm{X} \stackrel{d}{=} {M}\,(Y_1,\ldots,Y_d).
\end{gather*}
In other words, this means that $\bm{X}$ has a stochastic representation as in (\ref{canonical_construction}) with $H_t:=\Phi(t/M)$, where $\Phi$ denotes the distribution function of a standard normally distributed random variable.
\item[(c)] There is a random variable $Z$ with $\chi^2$-law with $d$ degrees of freedom, a positive random variable $M \in (0,\infty)$, and $\bm{S}$ uniformly distributed on the Euclidean unit sphere, all three being mutually independent, such that
\begin{gather*}
\bm{X} \stackrel{d}{=} {M}\,\sqrt{Z}\,\bm{S}.
\end{gather*}
In other words, the random variable $R$ of the general representation (\ref{representation_spherical}) is of the special form $R \stackrel{d}{=}M\,\sqrt{Z}$.
\item[(d)] The (multivariate) characteristic function of $\bm{X}$ has the form
\begin{gather*}
\IE\Big[e^{i\,(u_1\,X_1+\ldots+u_d\,X_d)} \Big] = \varphi(\norm{\bm{u}}_2^2),\quad \bm{u}=(u_1,\ldots,u_d) \in \IR^d,
\end{gather*}
where $\varphi$ is the Laplace transform $\varphi$ of some positive random variable.
\end{itemize}
\end{theorem}
 \begin{proof}
Named after \cite{schoenberg38}, see also \cite{kingman72} or \cite[p.\ 22]{aldous85} for further references. An alternative proof is also given in \cite{diaconis87}. Statement (c) is only included in order to highlight how the random radius $R$ must be chosen in the canonical representation (\ref{representation_spherical}) such that the law of $\bm{X}$ is in $\mathfrak{M}_{\ast}$, see also Remark \ref{rmk_uniform2} below; the interested reader can find a proof for the equivalence (b) $\Leftrightarrow$ (c) in \cite[Lemma 4.2, p.\ 166]{mai12}. Similarly, the equivalence (b) $\Leftrightarrow$ (d) is obvious, and $\varphi$ in (d) equals the Laplace transform of the positive random variable $M^2/2$ with $M$ from (b). Trivially, (b) implies (a). We only verify the non-obvious implication (a) $\Rightarrow$ (b), and the proof consists of two steps, following the lines of \cite[p.\ 22]{aldous85}.
\begin{itemize}
\item[(i)] As a first step we show Maxwell's Theorem, i.e.\ if $X_1,\ldots,X_d$ are independent and $(X_1,\ldots,X_d)$ is spherically symmetric, then all components $X_k$ are actually iid sharing a normal distribution with mean zero. Since $(X_1,\ldots,X_d)$ is spherically symmetric, its characteristic function can be written as
\begin{gather*}
\IE\Big[ e^{i\,(u_1\,X_1+\ldots+u_d\,X_d)}\Big] =:\varphi(\norm{\bm{u}}_2^2),\quad \bm{u}=(u_1,\ldots,u_d) \in \IR^d,
\end{gather*}
for some function $\varphi$ in one variable, see, e.g., \cite[Lemma 4.1, p.\ 161]{mai12}. Denoting the characteristic function of $X_k$ by $f_k$, $k=1,\ldots,d$, independence of the components implies that $\varphi(\norm{\bm{u}}_2^2)=f_1(u_1)\,\dots\,f(u_d)$. Taking the derivative\footnote{Notice that characteristic functions are differentiable.} w.r.t.\ $u_k$ and dividing by $\varphi(\norm{\bm{u}}_2^2)$ on both sides of the last equation implies for arbitrary $k=1,\ldots,d$ that
\begin{gather}
\frac{f_k^{'}(u_k)}{f_k(u_k)\,2\,u_k} = \frac{\varphi^{'}(\norm{\bm{u}}_2^2)}{\varphi(\norm{\bm{u}}_2^2)}.
\label{sr1}
\end{gather}
Let $u,y \in \IR$ arbitrary. Plugging $\bm{u}=(u,\ldots,u)$ into (\ref{sr1}) shows that
\begin{gather}
\frac{f_k^{'}(u)}{f_k(u)\,2\,u}= \frac{\varphi^{'}(\norm{\bm{u}}_2^2)}{\varphi(\norm{\bm{u}}_2^2)} =  \frac{f_j^{'}(u)}{f_j(u)\,2\,u},\quad \mbox{arbitrary }1 \leq k,j \leq d.
\label{sr2}
\end{gather}
Plugging some $\bm{u}$ which has $u$ as its $k$-th and $y$ as its $j$-th component into (\ref{sr1}), we observe
\begin{gather*}
\frac{f_k^{'}(u)}{f_k(u)\,2\,u} = \frac{\varphi^{'}(\norm{\bm{u}}_2^2)}{\varphi(\norm{\bm{u}}_2^2)}= \frac{f_j^{'}(y)}{f_j(y)\,2\,y} \stackrel{(\ref{sr2})}{=} \frac{f_k^{'}(y)}{f_k(y)\,2\,y}.
\end{gather*} 
Since $u,y$ were arbitrary, the functions $x \mapsto f_k^{'}(u)/(f_k(u)\,2\,u)$ are therefore shown to equal some constant $c$ independent of $k$. Since $f_k(0)=1$, solving the resulting ordinary differential equation implies that $f_k(u)=\exp(c\,u^2)$. Left to show is now only that $c \leq 0$, because this would imply that $f_k$ equals the characteristic function of a zero-mean normal. Since $f_k$ is a characteristic function and as such must be positive semi-definite, the inequality 
\begin{gather*}
\det\left[\begin{matrix}
f_k(0-0) & f_k(0-1) \\
f_k(1-0) & f_k(1-1)
\end{matrix}\right]=f_k(0)^2-f_k(1)\,f_k(-1)=1-e^{2\,c} \geq 0
\end{gather*}
must hold. Clearly, this is only possible for $c \leq 0$. The case $c=0$ is ruled out by the assumption that $\bm{X}$ is not identical to a vector of zeros.
\item[(ii)] If the law of $\bm{X}$ lies in $\mathfrak{M}_{\ast}$ we can without loss of generality assume that $\bm{X}$ equals the first $d$ members of an infinite exchangeable sequence $\{X_k\}_{k \in \IN}$. Conditioned on the tail-$\sigma$-field $\mathcal{H}:=\cap_{n \geq 1}\sigma(X_n,X_{n+1},\ldots)$ the random variables $X_1,\ldots,X_d$ are iid according to de Finetti's Theorem \ref{thm_definetti}. We observe for an arbitrary orthogonal matrix $O \in \IR^{d \times d}$ that 
\begin{gather*}
(\bm{X}\,O,X_{d+1},X_{d+2},\ldots) \stackrel{d}{=}(\bm{X},X_{d+1},X_{d+2},\ldots), 
\end{gather*}
since $\bm{X}$ is spherical. Since $\mathcal{H}$ does not depend on $\bm{X}$ (but only on the tail of the infinite sequence), this implies that the conditional distribution of $\bm{X}$ and $\bm{X}\,O$ given $\mathcal{H}$ are identical. As $O$ was arbitrary, $\bm{X}$ conditioned on $\mathcal{H}$ is spherical. Maxwell's Theorem now implies that $\bm{X}$ conditioned on $\mathcal{H}$ is an iid sequence of zero mean normals. Thus, only the standard deviation may still be a $\mathcal{H}$-measurable random variable, which we denote by $M$. 
\end{itemize}
\end{proof}

If (P) in Problem \ref{motivatingproblemrefined} is the property of ``having a spherical law (in some dimension)'', then Schoenberg's Theorem \ref{schoenberg_thm} also implies that $\mathfrak{M}_{\ast}=\mathfrak{M}_{\ast\ast}$, which follows trivially from the equivalence of (a) and (b), since the stochastic construction in (b) clearly works for arbitrary $n>d$ as well. Furthermore, it is observed that the random distribution function $H_t=\Phi(t/M)$ in part (b) satisfies the condition in Lemma \ref{lemma_rs} with $\mu=0$, so conditionally iid spherical laws are radially symmetric. In fact, (arbitrary) spherical laws are always radially symmetric, since $(X_1,\ldots,X_d) \stackrel{d}{=}(-X_1,\ldots,-X_d)$ follows immediately from the definition.

\begin{remark}[Realization of uniform law on Euclidean unit sphere]\label{rmk_uniform2}
Denoting $\bm{Y}=(Y_1,\ldots,Y_d)$, the equivalence (b) $\Leftrightarrow$ (c) in Theorem \ref{schoenberg_thm} implies 
\begin{gather*}
\bm{S} \stackrel{d}{=} \Big( \frac{Y_1}{\norm{\bm{Y}}_2},\ldots,\frac{Y_d}{\norm{\bm{Y}}_2}\Big),
\end{gather*}
which shows how to generate realizations of the uniform law on the Euclidean unit sphere from a list of iid standard normals.  
\end{remark}

\begin{remark}[Elliptical laws]
Spherical laws are always exchangeable, which is easy to see. A popular method to enrich the family of spherical laws to obtain a larger family beyond the exchangeable paradigm is linear transformation. To wit, for $\bm{X} \in \IR^k$ spherical with characteristic generator $\varphi$, $A \in \IR^{k \times d}$ some matrix with $\Sigma:=A'\,A \in \IR^{d \times d}$ and rank of $\Sigma$ equal to $k \leq d$, and with $\bm{b} =(b_1,\ldots,b_d)$ some real-valued row vector, the random vector
\begin{gather}
\bm{Z}=(Z_1,\ldots,Z_d) = \bm{X}\,A+\bm{b}
\label{def_elliptical}
\end{gather}
is said to have an \emph{elliptical law} with parameters $(\varphi,\Sigma,\bm{b})$. This generalization from spherical laws to elliptical laws is especially well-behaved from an analytical viewpoint, since the apparatus of linear algebra gets along perfectly well with the definition of spherical laws. The most prominent elliptical law is the multivariate normal distribution, which is obtained in the special case when $\varphi(x) = \exp(-x/2)$ is the Laplace transform of the constant $1/2$. The case when $\IE[\norm{\bm{X}}_2^2]<\infty$ is of most prominent importance, since the random vector $\bm{Z}$ then has existing covariance matrix given by $\IE[\norm{\bm{X}}_2^2]\,\Sigma\,/k$. 
\par
Since the normal distribution special case occupies a commanding role when deciding whether or not a spherical law is conditionally iid according to Theorem \ref{schoenberg_thm}(b), and since we have also solved our motivating Problem \ref{motivatingproblem} for the multivariate normal law in Example \ref{ex_Normal}, it is not difficult to decide when an elliptical law is conditionally iid as well. To wit, in the most important case when $\IE[\norm{\bm{X}}_2^2]<\infty$ the random vector $\bm{Z}$ in (\ref{def_elliptical}) has a stochastic representation that is conditionally iid if and only if $b_1=\ldots=b_d$, and $\bm{Z} \stackrel{d}{=}R\,\bm{Y}+\bm{b}$ with $R$ some positive random variable with finite second moment and $\bm{Y}=(Y_1,\ldots,Y_d)$ multivariate normal with zero mean vector and covariance matrix such as in Example \ref{ex_Normal}, i.e.\ with identical diagonal elements $\sigma^2 > 0$ and identical off-diagonal elements $\rho\,\sigma^2 \geq 0$.
\end{remark}

\subsection{$\ell_1$-norm symmetric laws}\label{subsec_AC}
According to \cite{mcneil09}, a random vector $\bm{X} \in [0,\infty)^d$ is called \emph{$\ell_1$-norm symmetric} if it has a stochastic representation 
\begin{gather*}
\bm{X} \stackrel{d}{=} R\,\bm{S},
\end{gather*}
where $R$ is a non-negative random variable and the random vector $\bm{S}$ is independent of $R$ and uniformly distributed on the unit simplex $S_d=\{\bm{x} \in [0,\infty)^d\,:\,\norm{\bm{x}}_1=1\}$. Comparing this representation to (\ref{representation_spherical}), the only difference is that $\bm{S}$ is now uniformly distributed on the unit sphere with respect to the $\ell_1$-norm (restricted to the positive orthant $[0,\infty)^d$), rather than on the unit sphere with respect to the Euclidean norm. Consequently, quite similar to spherical laws, realizations of $\ell_1$-norm symmetric distributions must be thought of as being the result of the following two-step simulation algorithm:  first draw one completely random point on the $d$-dimensional unit simplex, and then scale this point according to some one-dimensional probability distribution on the positive half-axis. 
\par 
Remark \ref{rmk_uniform2} points out an important relationship between the (univariate) standard normal distribution and the uniform law on the Euclidean unit sphere (w.r.t.\ the Euclidean norm $\norm{.}_2$). It is not difficult to observe that the (univariate) standard exponential law plays the analogous role for the uniform law on the unit simplex (w.r.t.\ the $\ell_1$-norm $\norm{.}_1$). More precisely, if the components of $\bm{E}=(E_1,\ldots,E_d)$ are iid exponentially distributed with unit mean, then
\begin{gather*}
\bm{S} :=\Big( \frac{E_1}{\norm{\bm{E}}_1},\ldots,\frac{E_d}{\norm{\bm{E}}_1}\Big)  
\end{gather*}
is uniformly distributed on the unit simplex, cf.\ \cite[Lemma 2.2(2), p.\ 77]{mai12} or \cite[Theorem 5.2(2), p.\ 115]{fang90}. An arbitrary $\ell_1$-norm symmetric random vector $\bm{X}$ is represented as
\begin{gather}
\bm{X} \stackrel{d}{=} R\,\Big( \frac{E_1}{\norm{\bm{E}}_1},\ldots,\frac{E_d}{\norm{\bm{E}}_1}\Big) 
\label{canonic_repr_AC}
\end{gather}
with independent $R$ and $\bm{E}$. With the analogy to the spherical case in mind, heuristic reasoning suggests that $\bm{X}$ is extendible if and only if $R$ is chosen such that it ``cancels'' out the denominator of $\bm{S}$ in distribution. Since $\norm{\bm{E}}_1$ has a unit-scale Erlang distribution with parameter $d$, this would imply that $R$ should be chosen as $R=Z/M$ for some positive random variable $M$ and an independent random variable $Z$ with Erlang distribution and parameter $d$. This is precisely the case, as Theorem \ref{thm_AC} below shows.
\par
Generally speaking, it follows from the canonical stochastic representation (\ref{canonic_repr_AC}) that 
\begin{align*}
\IP(X_k>x) &= \IP\Big(E_k>\frac{x}{R-x}\,\sum_{i \neq k}E_i,R>x\Big)=\IE\Big[ e^{-\frac{x}{R-x}\,\sum_{i \neq k}E_i}\,1_{\{R>x\}}\Big]\\
& = \IE\Big[ \max\Big\{ 1-\frac{x}{R},0\Big\}^{d-1}\Big]=:\varphi_{d,R}(x),\quad k=1,\ldots,d,
\end{align*}
where the last equality uses knowledge about the Laplace transform of the Erlang-distributed random variable $\sum_{i \neq k}E_i$. This means that the marginal survival functions of the components $X_k$ are given by the so-called \emph{Williamson $d$-transform} $\varphi_{d,R}$ of $R$. It has been studied in \cite{williamson56}, who shows in particular that the law of $R$ is uniquely determined by $\varphi_{d,R}$. A similar computation as above shows that the joint survival function of $\bm{X}$ is given by
\begin{gather*}
\IP(\bm{X}>\bm{x}) = \varphi_{d,R}(\norm{\bm{x}}_1),\quad \bm{x}=(x_1,\ldots,x_d) \in [0,\infty)^d.
\end{gather*}

Theorem \ref{thm_AC} solves Problem \ref{motivatingproblemrefined} for the property (P) of ``having an $\ell_1$-norm symmetric law (in some dimension)''.

\begin{theorem}[Conditionally iid $\ell_1$-norm symmetric laws]  \label{thm_AC}
Let $\varphi:[0,\infty) \rightarrow [0,1]$ be a function in one variable. The following statements are equivalent: 
\begin{itemize}
\item[(a)] There is an infinite sequence of random variables $\{X_k\}_{k \in \IN}$ such that for arbitrary $d \in \IN$ we have
\begin{gather*}
\IP(\bm{X}>\bm{x}) = \varphi(\norm{\bm{x}}_1),\quad \bm{x} \in  [0,\infty)^d.
\end{gather*}
\item[(b)] The function $\varphi$ equals the Laplace transform of some positive random variable $M$, i.e.\ $\varphi(x)=\IE[\exp(-x\,M)]$.
\end{itemize}
In this case, for arbitrary $d \in \IN$ we have
\begin{gather*}
\bm{X}=(X_1,\ldots,X_d) \stackrel{d}{=} \frac{1}{M}\,Z\,\bm{S}\stackrel{d}{=} \frac{1}{M}\,\bm{E},
\end{gather*}
where $\bm{X}$ as in (a), $M$ as in (b), $\bm{S}$ uniformly distributed on the unit simplex, $\bm{E}=(E_1,\ldots,E_d)$ a vector of iid unit exponentials, and $Z$ a unit-scale Erlang distributed variate with parameter $d$, all mutually independent. In other words, $\bm{X}$ has a stochastic representation as in (\ref{canonical_construction_2}) with $Z_t:=M\,t$, in particular is conditionally iid.
\end{theorem}
\begin{proof}
The implication (b) $\Rightarrow$ (a) works precisely along the stochastic model claimed, and is readily observed. The implication (a) $\Rightarrow$ (b) is known as Kimberling's Theorem, see \cite{kimberling74}. We provide a proof sketch in the sequel. From $d=1$ we observe that $\varphi$ is the survival function of some positive random variable. Consequently, due to Bernstein's Theorem\footnote{The original reference is \cite{bernstein29}, a detailed proof can be found in \cite{berg84} or \cite{schilling10}.}, it is sufficient to prove that $\varphi$ is completely monotone, meaning that $(-1)^{d}\,\varphi^{(d)} \geq 0$ for all $d \in \IN_0$. To this end, recall that
\begin{gather*}
(-1)^d\,\varphi^{(d)}(x) = \Delta^d_h[\varphi](x)+O(h),\quad \Delta^d_h[\varphi](x):=\sum_{k=0}^{d}\binom{d}{k}(-1)^{d-k}\,\varphi(x-k\,h),
\end{gather*}
so that it is sufficient to show that $\Delta^d_h[\varphi](x) \geq 0$ for arbitrary $d \in \IN_0$ and $x,h$ such that $0 \leq x-d\,h$. To this end, we consider the infinite sequence of random variables $\{U_k\}_{k \in \IN}$ with $U_k:=\varphi(X_k)$, $k \in \IN$, and with $\alpha:=\varphi(x/d)$ and $\beta:=\varphi(x/d-h)>\alpha$ define the events
\begin{gather*}
A_I:=\Big( \cap_{j \in I}\{U_j \leq \alpha\}\Big)\,\cap\,\Big( \cap_{j \notin I}\{U_j \leq \beta\}\Big),\quad I \subset \{1,\ldots,d\}.
\end{gather*}
A lengthy but straightforward computation, with one application of the inclusion exclusion principle, shows that
\begin{align*}
\Delta^d_h[\varphi](x) = \ldots = \IP\Big( A_{\emptyset}\setminus \big( \cup_{k=1}^{d}\{U_k \leq \alpha\}\big) \Big) \geq 0,
\end{align*}
which implies the claim.
\end{proof}

\begin{remark}[On involved probability transforms]
In Theorem \ref{thm_AC}, the function $\varphi$ in part (b) equals the Laplace transform of the random variable $M$. Furthermore, the survival function of any element in $\mathfrak{M}$ has the form as claimed in (a), only the parameterizing function $\varphi$ needs not be a Laplace transform in general. Instead, $\varphi$ always equals the Williamson $d$-transform of some positive random variable (namely of $R$). The Williamson $d$-transform of some random variable is also a Williamson $(d+1)$-transform (of some other random variable), and Laplace transforms can be viewed as a proper subset of Williamson $d$-transforms given by
\begin{gather*}
\{\mbox{Laplace transforms}\}=\bigcap_{d \in \IN}\{\mbox{Williamson }d\mbox{-transforms}\}.
\end{gather*}
The most important example for a Williamson $d$-transform, which is not a Laplace transform (in fact, not even a Williamson $(d+1)$-transform), is given by $\varphi(x)=\varphi(x;d,r)=(1-x/r)^{d-1}_+$ , with a constant $r>0$. In fact, \cite{williamson56} shows that the set of Williamson $d$-transforms is a simplex with extremal boundary given by $\{\varphi(.;d,r)\}_{r>0}$, which is just another way to say that the function $\varphi_{d,R}$ determines the probability law of the positive random variable $R$ uniquely. Similarly, Laplace transforms form a simplex with extremal boundary given by the functions $x \mapsto \exp(-m\,x)$ for $m>0$, which is just another way to say that the function $\varphi(x)=\IE[\exp(-x\,M)]$ determines the law of the positive random variable $M$ uniquely. Typical parametric examples for Laplace transforms in the context of $\ell_1$-norm symmetric distributions are $\varphi(x)=(1+x)^{-\theta}$ with $\theta>0$, corresponding to a Gamma distribution of $M$, or $\varphi(x)=\exp(-x^{\theta})$ with $\theta \in (0,1)$, corresponding to a stable distribution of $M$.
\end{remark}

\begin{remark}[Archimedean copulas]\label{rmk_AC}
Considering $\bm{X}=(X_1,\ldots,X_d)$ with $\ell_1$-norm symmetric law associated with the Williamson $d$-transform $\varphi=\varphi_{d,R}$, the random vector $(U_1,\ldots,U_d):=\big(\varphi(X_1),\ldots,\varphi(X_d) \big)$ has distribution function
\begin{gather*}
C_{\varphi}(u_1,\ldots,u_d):=\IP(U_1\leq u_1,\ldots,U_d \leq u_d)=\varphi\big(\varphi^{-1}(u_1)+\ldots+ \varphi^{-1}(u_d)\big),
\end{gather*}
for $u_1,\ldots,u_d \in [0,1]$. Recall that $\varphi^{-1}$ denotes the generalized inverse of $\varphi$. The function $C_{\varphi}$ is called an \emph{Archimedean copula} and the study of $\ell_1$-norm symmetric distributions can obviously be translated into an analogous study of Archimedean copulas. In the statistical and applied literature, however, Archimedean copulas have received considerably more attention. For instance, nested and hierarchical extensions of (exchangeable) Archimedean copulas have become quite popular, see, e.g.\ \cite{cossette17,hering10,hofertscherer10,mcneil08,zhu16,mai19}.
\end{remark}

\begin{remark}[Extension to Liouville distributions]\label{rmk_Liouville}
Analyzing the analogy between spherical laws (aka $\ell_2$-norm symmetric laws) and $\ell_1$-norm symmetric laws, there is one common mathematical fact on which the analytical treatment of both families relies. To wit, for both families the uniform distribution on the $d$-dimensional unit sphere can be represented as the normalized vector of iid random variables. In the spherical case the normalized vector $\bm{Y}/\norm{\bm{Y}}_2$ of $d$ iid standard normals $\bm{Y}=(Y_1,\ldots,Y_d)$ is uniform on the $\norm{.}_2$-sphere, whereas in the $\ell_1$-norm symmetric case the normalized vector $\bm{E}/\norm{\bm{E}}_1$ of $d$ iid standard exponentials $\bm{E}=(E_1,\ldots,E_d)$ is uniform on the $\norm{.}_1$-sphere restricted to the positive orthant $[0,\infty)^d$.  Furthermore, in both cases the normalization can be ``canceled out'' in distribution, that is
\begin{gather*}
 \sqrt{Z}\,\frac{\bm{Y}}{\norm{\bm{Y}}_2} \stackrel{d}{=}\bm{Y},\quad  R\,\frac{\bm{E}}{\norm{\bm{E}}_1} \stackrel{d}{=}\bm{E},
\end{gather*}
where $\sqrt{Z} \stackrel{d}{=}\norm{\bm{Y}}_2$ is independent of $\bm{Y}$ and $Z$ has a $\chi^2$-law with $d$ degrees of freedom and $R \stackrel{d}{=}\norm{\bm{E}}_1$ is independent of $\bm{E}$ and has an Erlang distribution with parameter $d$. The so-called \emph{Lukacs Theorem}, due to \cite{lukacs55}, states that the exponential distribution of the $E_k$ in the last distributional equality can be generalized to a Gamma distribution (but no other law on $(0,\infty)$ is possible). More precisely, if $\bm{G}=(G_1,\ldots,G_d)$ are independent random variables with Gamma distributions with the same scale parameter, then $\norm{\bm{G}}_1$ is independent of $\frac{\bm{G}}{\norm{\bm{G}}_1}$, which means that
\begin{gather}
R\,\frac{\bm{G}}{\norm{\bm{G}}_1} \stackrel{d}{=}\bm{G},\mbox{ where } R\stackrel{d}{=}\norm{\bm{G}}_1 \mbox{ is independent of }\bm{G}.
\label{lukacs}
\end{gather}
The random vector $\bm{S}:=\bm{G}/\norm{\bm{G}}_1$ on the unit simplex is not uniformly distributed unless the $G_k$ happen to be iid exponential. In general, the law of $\bm{S}$ is called \emph{Dirichlet distribution}, parameterized by the $d$ values $\bm{\alpha}=(\alpha_1,\ldots,\alpha_d)$, where the Gamma density of $G_k$ is given by
\begin{gather}
f_k(x) = x^{\alpha_k-1}\,e^{-x}/\Gamma(\alpha_k),\quad x>0, \quad k=1,\ldots,d.
\label{gamma_density}
\end{gather}
Notice that the scale parameter of this Gamma distribution is without loss of generality set to one, since it has no influence on the law of $\bm{S}$. A $d$-parametric generalization of $\ell_1$-norm symmetric laws is obtained by replacing the uniform law of $\bm{S}$ on the unit simplex (which is obtained for $\alpha_1=\ldots=\alpha_d$) with a Dirichlet distribution (with arbitrary $\alpha_k>0$). One says that the random vector $\bm{X}=R\,\bm{S}$ with $R$ some positive random variable and $\bm{S}$ an independent Dirichlet-distributed random vector on the unit simplex, follows a \emph{Liouville distribution}. It is precisely the property (\ref{lukacs}) that implies that the generalization to Liouville distributions is still analytically quite convenient to work with, see \cite{mcneil10} for a detailed study. Analogous to the $\ell_1$-norm symmetric case, the components of $\bm{X}$ are conditionally iid if $\alpha_1=\ldots=\alpha_d$ and $R$ satisfies $R \stackrel{d}{=}Z/M$ with $Z \stackrel{d}{=}\norm{\bm{G}}_1$ and $M$ some independent positive random variable. 
\end{remark}

Having at hand the apparatus of Archimedean copulas, we are now in the position to provide a non-trivial example for the situation $\mathfrak{M}_{\ast\ast} \subsetneq \mathfrak{M}_{\ast}$.

\begin{example}[In general, $\mathfrak{M}_{\ast\ast} \subsetneq \mathfrak{M}_{\ast}$] \label{ex_MstarnotMstarstar}
Consider the family $\mathfrak{M} \subset M_+^1([0,1]^2)$ defined by the property (P) of ``having an Archimedean copula as distribution function and being radially symmetric''. It is well-known since \cite[Theorem 4.1]{frank79} that the set $\mathfrak{M}$ comprises precisely Frank's copula family, that is the bivariate distribution function of an element in $\mathfrak{M}$ is either given by $C_{-\infty}(u_1,u_2):=\max\{u_1+u_2-1,0\}$, by $C_0(u_1,u_2):=u_1\,u_2$, by $C_{\infty}(u_1,u_2):=\min\{u_1,u_2\}$, or by 
\begin{gather*}
C_{\theta}(u_1,u_2) := -\frac{1}{\theta}\,\log\Big\{ 1+\,\frac{\big( e^{-\theta\,u_1}-1\big)\,\big( e^{-\theta\,u_2}-1\big)}{\big( e^{-\theta}-1\big)}\Big\},\quad u_1,u_2 \in [0,1],
\end{gather*}
for some parameter $\theta \in (-\infty,0) \cup (0,\infty)$. Since Kendall's Tau of the copula $C_{\theta}$ is negative in the case $\theta<0$, Lemma \ref{lemma_kendall} implies that the subset $\mathfrak{M}_{\ast}$ can at best contain the elements corresponding to $\theta \in [0,\infty]$. Indeed, the cases $\theta \in \{0,\infty\}$ are obviously contained in $\mathfrak{M}_{\ast\ast} \subset \mathfrak{M}_{\ast}$, and for $\theta \in (0,\infty)$ membership in $\mathfrak{M}_{\ast}$ follows via the canonical construction (\ref{canonical_construction}) with the choice $H \sim \gamma \in M_+^1(\mathfrak{H}_+)$, given by
\begin{gather}
H_t = \Big( \frac{1-e^{-\theta\,t}}{1-e^{-\theta}}\Big)^M,\quad t \in [0,1],
\label{ACciidprocess}
\end{gather}
for a random variable $M$ with logarithmic distribution $\IP(M=m)=(1-\exp(-\theta))^m/(m\,\theta)$, $m \in \IN$. Furthermore, we can deduce from Theorem \ref{thm_AC} that the property of ``having an Archimedean copula as distribution function (in arbitrary dimension)'' implies that potential elements in $\mathfrak{M}_{\ast\ast}$ must necessarily be induced by a stochastic process of the form (\ref{ACciidprocess}) with some positive random variable $M$, which must necessarily be logarithmic in the radially symmetric case by the result of Frank. The only thing left to check is whether the multivariate Archimedean copula derived from the canonical construction via $H$ defined by (\ref{ACciidprocess}) with logarithmic $M$ is radially symmetric in arbitrary dimension $d \geq 2$. According to Lemma \ref{lemma_rs} this is the case if and only if
\begin{align*}
& \Big\{\Big( \frac{1-e^{-\theta\,\big(\frac{1}{2}-t\big)}}{1-e^{-\theta}}\Big)^M \Big\}_{t \in \big[-\frac{1}{2},\frac{1}{2} \big]} = \{H_{\frac{1}{2}-t}\}_{t \in \big[-\frac{1}{2},\frac{1}{2} \big]} \\
& \qquad \stackrel{d}{=} \{1-H_{t+\frac{1}{2}}\}_{t \in \big[-\frac{1}{2},\frac{1}{2} \big]}=\Big\{1-\Big( \frac{1-e^{-\theta\,\big(t+\frac{1}{2}\big)}}{1-e^{-\theta}}\Big)^M \Big\}_{t \in \big[-\frac{1}{2},\frac{1}{2} \big]}. 
\end{align*}
This statement is false, however, as will briefly be explained. Assuming it was true, then in particular for $t=0$ we observe that the law of the random variable $H_{\frac{1}{2}}$ was symmetric about $\frac{1}{2}$. In particular, this symmetry would imply
\begin{align*}
0=\IE\Big[\Big(H_{\frac{1}{2}}-\frac{1}{2}\Big)^3 \Big]=[\ldots]=\varphi_{\theta}\big( 3\,\varphi_{\theta}^{-1}(1/2)\big)-\frac{3}{2}\,\varphi_{\theta}\big( 2\,\varphi_{\theta}^{-1}(1/2)\big)+\frac{1}{4},
\end{align*}
with $\varphi_{\theta}(x)=-\log\big( e^{-x}\,(e^{-\theta}-1)+1\big)/\theta$. Numerically, it is easily verified that the last equality does not hold for any $\theta \in (0,\infty)$, since the right-hand side is strictly smaller than zero. Thus, we see that $\mathfrak{M}_{\ast\ast}$ consists only of two elements, namely those corresponding to $\{0,\infty\}$. Thus, $\mathfrak{M}_{\ast\ast} \subsetneq \mathfrak{M}_{\ast} \subsetneq \mathfrak{M}$, since $\{0,\infty\} \subsetneq [0,\infty] \subsetneq [-\infty,\infty]$.
\end{example}

\subsection{$\ell_{\infty}$-norm symmetric laws}\label{subsec_gnedin}
\cite[Theorem 2]{gnedin95} studies random vectors $\bm{X}=(X_1,\ldots,X_d)$ which are absolutely continuous with density given by
\begin{gather}
f_{\bm{X}}(\bm{x}) = g_d(x_{[d]}),\quad \bm{x} \in (0,\infty)^d,
\label{gnedin_linfsymmdens}
\end{gather}
with some measurable function $g_d:(0,\infty) \rightarrow [0,\infty)$. Recall that 
\begin{gather*}
x_{[d]}:=\max\{x_1,\ldots,x_d\}=\norm{\bm{x}}_{\infty} 
\end{gather*}
equals the $\ell_{\infty}$-norm of $\bm{x} \in (0,\infty)^d$. Since $f_{\bm{X}}$ is invariant with respect to permutations of the components of $\bm{x}$, the random vector $\bm{X}$ is exchangeable. But whether or not it is conditionally iid depends on the choice of $g_d$. First of all, since $f_{\bm{X}}$ is a probability density,
\begin{gather}
1=\iint_{(0,\infty)^d}g_d(x_{[d]})\,\mathrm{d}\bm{x} = d\,\int_0^{\infty}g_d(x)\,x^{d-1}\,\mathrm{d}x,
\label{gnedin_nec}
\end{gather}
constituting a necessary and sufficient integrability condition on $g_d$ such that $f_{\bm{X}}$ defines a proper probability density. Furthermore, lower-dimensional margins of $\bm{X}$ have a density of the same structural form, since
\begin{align}
\int_0^{\infty}f_{\bm{X}}(x_1,\ldots,x_{d-1},x)\,\mathrm{d}x &= g_{d-1}(x_{[d-1]}),\quad x_1,\ldots,x_{d-1}>0, \nonumber\\
\mbox{where }g_{d-1}(x)&:=\int_x^{\infty}g_d(u)\,\mathrm{d}u+x\,g_d(x),\label{marginingcond_gnedin}
\end{align}
and the function $g_{d-1}$ is easily checked to satisfy (\ref{gnedin_nec}) in dimension $d-1$, that is $1=(d-1)\,\int_0^{\infty}g_{d-1}(x)x^{d-2}\,\mathrm{d}x$. It is further not difficult to verify that $g_d$ is given in terms of $g_{d-1}$ as
\begin{gather}
g_d(x) = \frac{g_{d-1}(x)}{x}-\int_x^{\infty}\frac{g_{d-1}(u)}{u^2}\,\mathrm{d}u.
\label{identityfordecr_gnedin}
\end{gather}
If $\mathfrak{M}$ denotes the family of all laws with density of the form (\ref{gnedin_linfsymmdens}), i.e.\ with a function $g_d$ satisfying (\ref{gnedin_nec}), the following result provides necessary and sufficient conditions on $g_d$ to define a law in $\mathfrak{M}_{\ast}$.

\begin{theorem}[Conditionally iid $\ell_{\infty}$-norm symmetric densities]\label{thm_gnedin}
Let $\mathfrak{M}$ be the family of probability laws on $(0,\infty)^d$ with densities of the form (\ref{gnedin_linfsymmdens}) with a measurable function $g_d:(0,\infty) \rightarrow [0,\infty)$ satisfying (\ref{gnedin_nec}). For $\bm{X}$ with law in $\mathfrak{M}$, the following statements are equivalent:
\begin{itemize}
\item[(a)] The law of $\bm{X}$ lies in $\mathfrak{M}_{\ast}$.
\item[(b)] $g_d$ is non-increasing. 
\item[(c)] For a vector $\bm{U}=(U_1,\ldots,U_d)$ whose components are iid uniform on $[0,1]$ and an independent, positive random variable $M$ we have
\begin{gather*}
\bm{X} \stackrel{d}{=} M\,\bm{U}.
\end{gather*}
\end{itemize}
\end{theorem}
\begin{proof}
This is \cite[Theorem 2]{gnedin95}. Clearly, (c) $\Rightarrow$ (a) is obvious. In order to see (b) $\Rightarrow (c)$, we first conclude from (\ref{gnedin_nec}) that 
\begin{gather}
0 = \lim_{x \rightarrow \infty}g_d(x)\,x^d = \lim_{x \rightarrow \infty}g_d\Big(\frac{1}{x} \Big)\,\frac{1}{x^d}.
\label{shortrefgnedinproof}
\end{gather}
By non-increasingness, we may without loss of generality assume that $g_d$ is right-continuous (otherwise, change to its right-continuous version, which does not change the density $f_{\bm{X}}$ essentially). Applying integration by parts, (\ref{shortrefgnedinproof}) and (\ref{gnedin_nec}) imply
\begin{gather*}
\int_0^{\infty}x^d \,\mathrm{d}\big(-g_d(x)\big) = d\,\int_0^{\infty}g_d(x)\,x^{d-1}\,\mathrm{d}x = 1.
\end{gather*}
Consequently, $x \mapsto \int_0^{x}y^d\,\mathrm{d}\big( -g_d(y)\big)$ defines the distribution function of a positive random variable $M$, and we see that
\begin{gather*}
\IE[1_{\{M>x\}}\,M^{-d}] = \int_{x}^{\infty}\frac{y^d}{y^d}\,\mathrm{d}\big( -g_d(y)\big) = g_d(x).
\end{gather*}
Now let $\bm{U}$ as claimed be independent of $M$. Conditioned on $M$, the density of $M\,\bm{U}$ is
\begin{gather*}
\bm{x} \mapsto \prod_{k=1}^d\frac{1_{\{0<x_k <M\}}}{M} = 1_{\{0<x_{[d]} <M\}}\,\frac{1}{M^d}.
\end{gather*}
Integrating out $M$, the density of $M\,\bm{U}$ is found to be
\begin{gather*}
\int_0^{\infty}1_{\{x_{[d]} <m\}}\,\frac{1}{m^d}\mathrm{d}\IP(M \leq m) = \IE[1_{\{M>x_{[d]}\}}\,M^{-d}] = g_d(x_{[d]}),
\end{gather*}
which shows (c). The hardest part is (a) $\Rightarrow$ (b). Fix $\epsilon>0$ arbitrary. Due to measurability of $g_d$, Lusin's Theorem guarantees continuity of $g_d$ on a set $C_{\epsilon}$ whose complement has Lebesgue measure less than $\epsilon$. Without loss of generality we may assume that all points $t$ in $C_{\epsilon}$ are density points, i.e.\ satisfy 
\begin{gather*}
\lim_{\delta \searrow 0}\frac{\lambda(C_{\epsilon} \cap [t-\delta,t+\delta])}{2\,\delta} = 1,
\end{gather*}
where $\lambda$ denotes Lebesgue measure. Let $\{X_k\}_{k \in \IN}$ an infinite exchangeable sequence such that $d$-margins have the density $f_{\bm{X}}$. Fix $t \geq s$ arbitrary. We define the sequence of random variables $\{\xi_k\}_{k \in \IN}$ by
\begin{gather*}
\xi_k := \frac{1}{2\,\delta}\big( 1_{\{X_k \in A_s\}}-1_{\{X_k \in A_{t}\}}\big),\quad k \in \IN,
\end{gather*} 
where $A_x :=C_{\epsilon} \cap [x-\delta,x+\delta]$ for $x \in \{s,t\}$. Notice that the $\xi_k$ are square-integrable and
\begin{gather*}
0 \leq \IE[(\xi_1+\ldots+\xi_d)^2] = d\,\IE[\xi_1^2]+d\,(d-1)\,\IE[\xi_1\,\xi_2].
\end{gather*}
If we divide by $d^2$ and let $d \rightarrow \infty$, it follows that $\IE[\xi_1\,\xi_2] \geq 0$. Denoting by $g_2$ the marginal density of $(X_1,X_2)$, we observe 
\begin{align*}
0 & \leq \IE[\xi_1\,\xi_2] = \frac{1}{4\,\delta^2}\,\big\{ \IE[1_{\{X_1,X_2 \in A_s\}}]+\IE[1_{\{X_1,X_2 \in A_t\}}]-2\,\IE[1_{\{X_1 \in A_t,X_2 \in A_s\}}]\big\}\\
& = \frac{1}{4\,\delta^2}\,\Big\{\iint_{A_s \times A_s}g_2(x_{[2]})\,\mathrm{d}\bm{x} +\iint_{A_t \times A_t}g_2(x_{[2]})\,\mathrm{d}\bm{x}-2\,\iint_{A_t \times A_s}g_2(x_{[2]})\,\mathrm{d}\bm{x}\Big\}\\
& = g_2(\eta_s)+g_2(\eta_{t})-2\,g_2(\tilde{\eta}_t),
\end{align*}
for certain values $s-\delta \leq \eta_s \leq s+\delta$ and $t-\delta \leq \eta_t,\tilde{\eta}_t \leq t+\delta$ by the mean value theorem for Lebesgue integration. As $\delta \searrow 0$, we thus observe that $g_2(s) \geq g_2(t)$, i.e.\ $g_2$ is non-increasing. Making use of (\ref{identityfordecr_gnedin}) and integrating by parts, we observe that
\begin{gather*}
g_3(x) = \int_x^{\infty}\frac{1}{u}\,\mathrm{d}\big(-g_2(u)\Big),
\end{gather*}
which implies that $g_3$ is non-increasing as well. Inductively, the same argument implies that $g_4,\ldots,g_d$ are all non-increasing.
\end{proof} 
From the equivalence of (a) and (c) in Theorem \ref{thm_gnedin} we observe easily that $\mathfrak{M}_{\ast}=\mathfrak{M}_{\ast\ast}$, when considering the property (P) of ``having a density of the form (\ref{gnedin_linfsymmdens}) (in some dimension $d \in \IN$)'' in Problem \ref{motivatingproblemrefined}. Notice furthermore that the law of $M\,\bm{U}$ is static in the sense defined in the beginning of this section, and we have
\begin{gather*}
H_t:=\IP(X_k \leq t\,|\,M) = \max\Big\{ 0,\min\Big\{ 1,\frac{t}{M}\Big\}\Big\},\quad t \in \IR,
\end{gather*}
for $X_k:=M\,U_k$ as defined in part (c) of Theorem \ref{thm_gnedin}.

\begin{remark}[Common umbrella of $\ell_p$-norm symmetry results]
Theorem \ref{thm_gnedin} on $\ell_{\infty}$-norm symmetric densities is very similar in nature to Schoenberg's Theorem \ref{schoenberg_thm} on $\ell_2$-norm symmetric characteristic functions and Theorem \ref{thm_AC} on $\ell_1$-norm symmetric survival functions, which makes it a beautiful result with regards to the present survey. The reference \cite{rachev91} considers all these three cases under one common umbrella, and even manages to generalize them in some meaningful sense to the case of arbitrary $\ell_p$-norm, with $p \in [1,\infty]$ arbitrary\footnote{The authors even allow for $p \in (0,1)$, but in this case $\norm{.}_p$ is no longer a norm.}. More precisely, it is shown that an infinite exchangeable sequence $\{X_k\}_{k \in \IN}$ of the form $X_k:=M\,Y_k$, $k \in \IN$, with $M>0$ and an independent iid sequence $\{Y_k\}_{k \in \IN}$ of positive random variables is $\ell_p$-norm symmetric in some meaningful sense\footnote{See \cite{rachev91} for details.} if and only if the random variables $Y_k$ have density $f_p$ given by
\begin{gather*}
f_p(x) := \frac{p^{1-\frac{1}{p}}}{\Gamma(1/p)}\,e^{-\frac{x^p}{p}},\quad 0<p,x<\infty,\quad f_{\infty}(x):=1_{\{x \in (0,1)\}}.
\end{gather*}
Notice that $f_1$, $f_2$, and $f_{\infty}$ are the densities of the unit exponential law, the absolute value of a standard normal law, and the uniform law on $[0,1]$, respectively. This parametric family in the parameter $p$ is further investigated, and might for instance be characterized by the fact that $f_p$ for $p<\infty$ has maximal entropy among all densities on $(0,\infty)$ with $p$-th moment equal to one, and $f_{\infty}$ has maximal entropy among all densities with support $(0,1)$, which is \cite[Theorem 3.5]{rachev91}. 
\end{remark}

An analogous result to Theorem \ref{thm_gnedin} on mixtures of the form $M\,\bm{U}$, when the components of $\bm{U}$ are iid uniform on $[-1,1]$, is also presented in \cite{gnedin95}. The resulting densities depend on two arguments, $x_{[1]}$ and $x_{[d]}$. Furthermore, \cite[Corollary 4.3]{rachev91} prove that an infinite exchangeable sequence $\{X_k\}_{k \in \IN}$ satisfies
\begin{gather*}
\{X_k\}_{k \in \IN} \stackrel{d}{=} \{M\,U_k\}_{k \in \IN},\quad U_1,U_2,\ldots \mbox{ iid uniform on }[0,1],\,M>0\mbox{ independent},
\end{gather*}
if and only if for arbitrary $d \in \IN$ and almost all $s>0$ the law of $\bm{X}=(X_1,\ldots,X_d)$ conditioned on the event $\{\norm{\bm{X}}_{\infty}=s\}$ is uniformly distributed on the sphere $\{\bm{x} \in (0,\infty)^d\,:\,\norm{\bm{x}}_{\infty}=s\}$. This provides an alternative characterization of densities that are $\ell_{\infty}$-norm symmetric and conditionally iid. 

\begin{remark}[Relation to non-homogeneous pure birth processes]\label{rmk_purebirth}
\cite{shaked02} provide an interesting interpretation of $\ell_{\infty}$-norm symmetric densities, which is briefly explained. Every non-negative function $g_d$ satisfying (\ref{gnedin_nec}) is of the form
\begin{gather*}
g_d(x) = c_d\,r_d(x)\,e^{-\int_0^{x}r_d(u)\,\mathrm{d}u}
\end{gather*}
for some non-negative function $r_d$ satisfying $\int_0^{\infty}r_d(x)\,\mathrm{d}x=\infty$, and some normalizing constant $c_d>0$. To wit, a function
\begin{gather}
r_d(x):=\frac{g_d(x)}{c_d\,\int_x^{\infty}g_d(u)\,\mathrm{d}u},\quad x>0,
\label{solution_rfromg}
\end{gather}
for some normalizing constant $c_d>0$ does the job, as can readily be checked. From such a function $r_d$ we iteratively define functions $r_{d-1},\ldots,r_1$ by solving the equations 
\begin{gather}
r_k(x)\,e^{-R_{k}(x)}=\frac{e^{-R_{k+1}(x)}}{\int_0^{\infty}e^{-R_{k+1}(u)}\,\mathrm{d}u},\quad k=d-1,\ldots,1,
\label{recursion_epoch_birth}
\end{gather}
where $R_k(x):=\int_0^{x}r_k(u)\,\mathrm{d}u$ for $k=1,\ldots,d$. Notice that $r_k$ is related to the right-hand side of (\ref{recursion_epoch_birth}) exactly in the same way as $r_d$ is related to $g_d$, so the solution (\ref{solution_rfromg}) shows how the $r_k$ look like in terms of $r_{k+1}$. We define independent positive random variables $E_1,\ldots,E_d$ with survival functions $\IP(E_k>x)=\exp(-R_k(x))$, $k=1,\ldots,d$, $x \geq 0$. Independently, let $\Pi$ be a random permutation of $\{1,\ldots,d\}$ with $\IP(\Pi=\pi)=1/d!$ for each permutation $\pi$ of $\{1,\ldots,d\}$, i.e.\ $\Pi$ is uniformly distributed on the set of all $d!$ permutations. We consider the increasing sequence of random variables $T_1<T_2 <\ldots<T_d$ defined by $T_k:=E_1+\ldots+E_k$. Then the (obviously exchangeable) random vector $\bm{X}=(X_1,\ldots,X_d):=(T_{\Pi(1)},\ldots,T_{\Pi(d)})$ has density (\ref{gnedin_linfsymmdens}). If $E_1,E_2,\ldots$ is an arbitrary sequence of independent, absolutely continuous, positive random variables the counting process
\begin{gather*}
N_t:=\sum_{k \geq 1}1_{\{E_1+\ldots+E_k \leq t\}},\quad t \geq 0,
\end{gather*}
is called \emph{non-homogeneous pure birth process} with intensity rate functions $r_k(x):=-\frac{\partial}{\partial x}\log\{\IP(E_k>x)\}$, $k \geq 1$. A random permutation of the first $d$ jump times $T_k:=E_1+\ldots+E_k$, $k =1,\ldots,d$, of a pure birth process $N$ thus has an $\ell_{\infty}$-norm symmetric density if the intensities $r_1,\ldots,r_{d-1}$ can be retrieved recursively from $r_d$ via (\ref{recursion_epoch_birth}). The case of arbitrary intensities $r_1,\ldots,r_d$ hence provides a natural generalization of the family of $\ell_{\infty}$-norm symmetric densities. It appears to be an interesting open problem to determine necessary and sufficient conditions on $r_1,\ldots,r_d$ such that the respective exchangeable density is conditionally iid, see also paragraph \ref{open_iidunif} below. 
\end{remark}

\begin{example}[Pareto mixture of uniforms]\label{ex_ParetoUnif}
Let $M$ in Theorem \ref{thm_gnedin} have survival function $\IP(M>x)=\min\{1,x^{-\alpha}\}$ for some $\alpha>0$. The associated function $g_d$ generating the $\ell_{\infty}$-norm symmetric density is given by
\begin{gather*}
g_d(x) = \IE[1_{\{M>x\}}\,M^{-d}] = \alpha\,\int_{\max\{x,1\}}^{\infty}u^{-d-1-\alpha}\,\mathrm{d}u = \frac{\alpha}{d+\alpha}\max\{1,x\}^{-d-\alpha}.
\end{gather*}
The components $X_k$ of $\bm{X}$ have the following one-dimensional distribution function $G(x):=\IP(X_k \leq x)$, and respective inverse $G^{-1}$, given by
\begin{align*}
G(x) &= \begin{cases}
\frac{\alpha}{1+\alpha}\,x & \mbox{, if }x<1\\
1-\frac{1}{1+\alpha}\,x^{-\alpha} & \mbox{, if }x \geq 1\\
\end{cases},\\
G^{-1}(y) &= \begin{cases}
\frac{1+\alpha}{\alpha}\,y & \mbox{, if }0 < y<\frac{\alpha}{1+\alpha}\\
\big( (1-y)\,(1+\alpha)\big)^{-\frac{1}{\alpha}} & \mbox{, if }\frac{\alpha}{1+\alpha}\leq y<1\\
\end{cases}.
\end{align*}
This induces the one-parametric bivariate copula family defined by
\begin{align*}
&C_{\alpha}(u_1,u_2) := \IP\big( G(X_1) \leq u_1,G(X_2) \leq u_2\big)\\
&  \qquad = \begin{cases}
\frac{(1+\alpha)^2}{\alpha\,(\alpha+2)}\,u_{1}\,u_{2} & \mbox{, if }u_1,\,u_2 \leq \frac{\alpha}{1+\alpha}\\
u_{[1]}-\frac{(1+\alpha)^{1+1/\alpha}}{2+\alpha}\,u_{[1]}\,(1-u_{[2]})^{1+1/\alpha} & \mbox{, if }u_{[1]} \leq \frac{\alpha}{1+\alpha} \leq u_{[2]}\\
u_{[1]}-\frac{\alpha}{2+\alpha}\,(1-u_{[1]})^{-1/\alpha}\,(1-u_{[2]})^{1+1/\alpha} & \mbox{, else}\\
\end{cases}.
\end{align*}
Scatter plots from this copula for different values of $\alpha$ are depicted in Figure \ref{fig:Pareto}, visualizing the dependence structure behind pairs of $\bm{X}$. The dependence decreases with $\alpha$, and the limiting cases $\alpha=0$ and $\alpha=\infty$ correspond to perfect positive association and independence, respectively. One furthermore observes that the dependence is highly asymmetric, i.e.\ large values of $G(X_1),G(X_2)$ are more likely jointly close to each other than small values, which behave like independence. This effect can be quantified in terms of the so-called upper- and lower-tail dependence coefficients, given by
\begin{gather*}
\lim_{x  \rightarrow \infty}\IP( X_1>x\,|\,X_2>x\big) = \frac{2}{2+\alpha},\quad \lim_{x  \searrow 0}\IP( X_1 \leq x\,|\,X_2\leq x\big) = 0,
\end{gather*}
respectively.
\end{example}

\begin{figure}[!ht]
\caption{Left: $5000$ samples of $(G(X_1),G(X_2))$ for $\alpha=0.1$ in Example \ref{ex_ParetoUnif}. Right: $5000$ samples of $(G(X_1),G(X_2))$ for $\alpha=1$ in Example \ref{ex_ParetoUnif}.}
{\includegraphics[width=7cm]{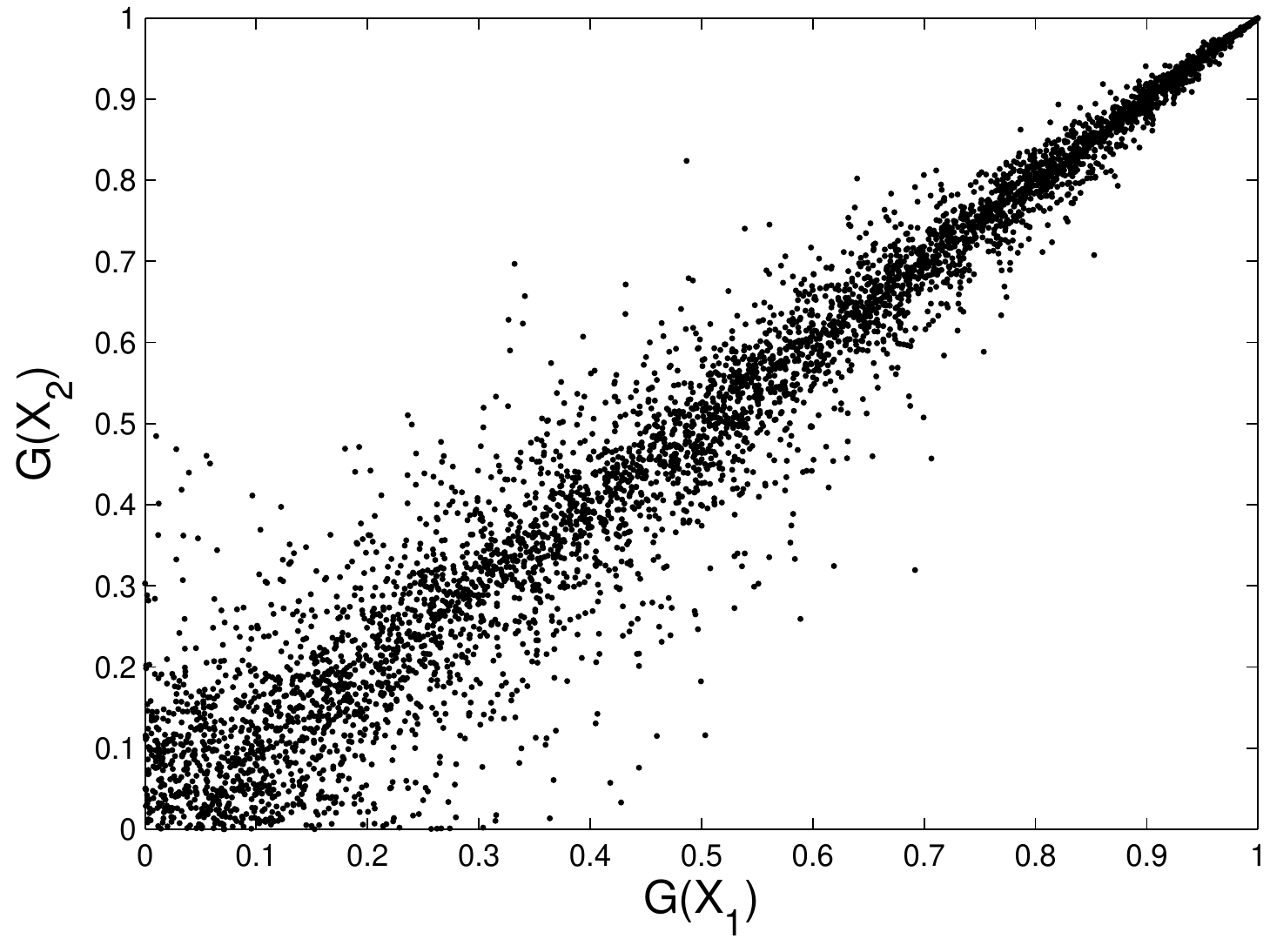} 
\includegraphics[width=7cm]{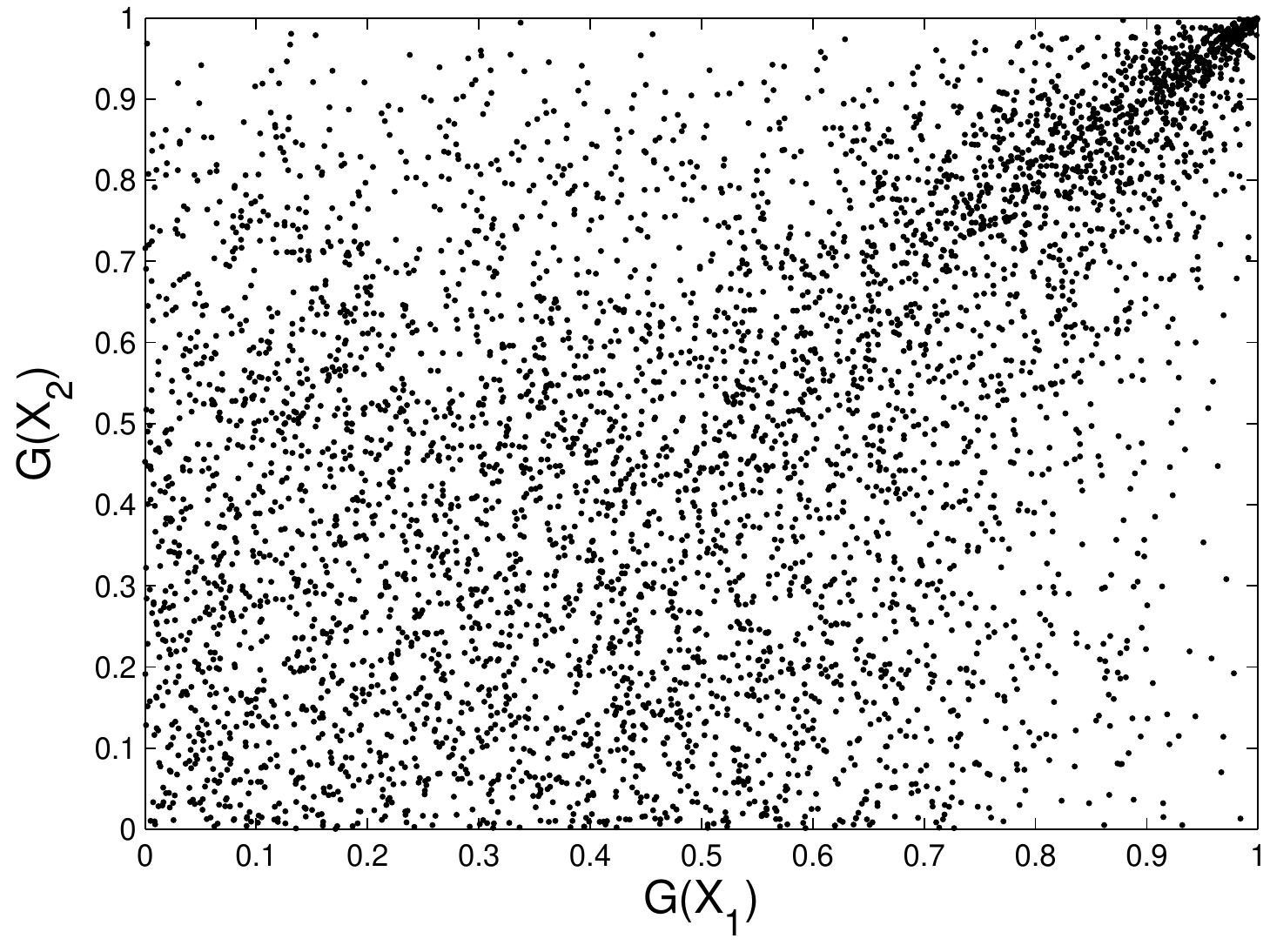}}
\label{fig:Pareto}
\end{figure}

\newpage

\section{The multivariate lack-of-memory property} \label{sec:lom}
A random vector $\bm{X}=(X_1,\ldots,X_d)$ with non-negative components is said to satisfy the \emph{(multivariate) lack-of-memory property} if for arbitrary $1 \leq i_1<\ldots<i_n \leq d$ we have that
\begin{align*}
&\IP(X_{i_1}>t_{i_1}+t,\ldots,X_{i_n}>t_{i_n}+t\,|\,X_{i_1}>t,\ldots,X_{i_n}>t)\\
& \qquad = \IP(X_{i_1}>t_{i_1},\ldots,X_{i_n}>t_{i_n}),
\end{align*}
with the $t,t_1,\ldots,t_d$ either in $(0,\infty)$ (continuous support case) or in $\IN_0$ (discrete support case). The lack-of-memory property is very intuitive when the $k$-th component $X_k$ of $\bm{X}$ is interpreted as the future time point at which the $k$-th component in a system of $d$ components fails. In words, it means that conditioned on the survival of an arbitrary sub-system $(i_1,\ldots,i_n)$ until time $t$, the residual lifetimes of the components $i_1,\ldots,i_n$ are identical in distribution to the lifetimes at inception of the system. Needless to mention that such intuitive property occupies a commanding role in reliability theory, see \cite{barlow75} for a textbook treatment, but is also important in other contexts such as financial risk management, e.g., \cite{giesecke03,lindskog03}. An alternative way to formulate the multivariate lack-of-memory property, due to \cite{brigo16}\label{brigoref}, is the following. For $1 \leq i_1<\ldots<i_n \leq d$ we denote by $Z_{i_1,\ldots,i_n}(t):=(1_{\{X_{i_1}>t\}},\ldots,1_{\{X_{i_n}>t\}})$, $t \geq 0$, the stochastic process which indicates for each of the $n$ components $i_1,\ldots,i_n$ whether it is still working or already dysfunctional. The random vector $\bm{X}$ has the lack-of-memory property if and only if $Z_{i_1,\ldots,i_n}$ is a continuous-time Markov chain for all $1 \leq i_1<\ldots<i_n \leq d$.
\par
From a theoretical point of view, studying the (multivariate) lack-of-memory property is also natural as it generalizes very popular one-dimensional probability distributions to the multivariate case. Indeed, if $d=1$ we abbreviate $X:=X_1$ and recall the following classical characterizations.
\begin{lemma}[Characterization of lack-of-memory for $d=1$]\label{lemma_lom}
$\mbox{}$
\begin{itemize}
\item[($\mathcal{E}$)] If the support of $X$ equals $[0,\infty)$, then $X$ satisfies the lack-of-memory property if and only if $X$ has an exponential distribution, that is $\IP(X>t)=\exp(-\lambda\,t)$ for some $\lambda>0$.
\item[($\mathcal{G}$)] If the support of $X$ equals $\IN$, then $X$ satisfies the lack-of-memory property if and only if $X$ has a geometric distribution, that is $\IP(X>n)=(1-p)^n$ for some $p \in (0,1)$.
\end{itemize}
\end{lemma}
\begin{proof}
In the geometric case, inductively we see that $\bar{F}(n):=\IP(X>n)$ satisfies $\bar{F}(n)=\bar{F}(1)^n$, $n \in \IN_0$, an the claim follows with $p:=1-\bar{F}(1)$. Notice that $\bar{F}(1) \in \{0,1\}$ is ruled out by the assumption of support equal to $\IN$. The exponential case follows similarly, see \cite[p.\ 190]{billingsley95}.
\end{proof}

\subsection{Marshall-Olkin and multivariate geometric distributions}
The well known characterizations of univariate lack-of-memory in Lemma \ref{lemma_lom} have been lifted to the multivariate case in \cite{marshall67} and \cite{arnold75,mai13}, respectively, which is briefly recalled. First of all, we introduce the multivariate exponential models of \cite{marshall67} and \cite{arnold75}. To this end, we denote by $\mathcal{E}(\lambda)$ the univariate exponential law with rate $\lambda>0$, and by $\mathcal{G}(p)$ the univariate geometric distribution with parameter $p\in (0,1)$, i.e.\ with survival function $\IP(X>n)=(1-p)^n$. In order to include boundary cases, we denote by $\mathcal{E}(0),\mathcal{G}(0)$ the probability law of a degenerate random variable that is identically equal to infinity, and by $\mathcal{G}(1)$ the probability law of a degenerate random variable that is identically equal to one. 
\begin{example}[Probability laws with multivariate lack-of-memory]\label{ex_lom}
$\mbox{}$
\begin{itemize}
\item[($\mathcal{E}$)] For each non-empty $I \subset \{1,\ldots,d\}$ let $\lambda_I \geq 0$ with $\sum_{I\,:\,k \in I}\lambda_I>0$ for each $k=1,\ldots,d$. With $E_I \sim \mathcal{E}(\lambda_I)$ a list of $2^d-1$ independent random variables, we define $\bm{X}$ via
\begin{gather*}
X_k:=\min\{E_I\,:\,k \in I\},\quad k=1,\ldots,d.
\end{gather*}
Then $\bm{X}$ satisfies the multivariate lack-of-memory property, which is easy to see while noticing that the survival function of $\bm{X}$ equals
\begin{gather*}
\bar{F}(\bm{x})=\IP(\bm{X}>\bm{x}) = \exp\Big\{ -\sum_{\emptyset \neq I \subset \{1,\ldots,d\}}\lambda_I\,\max_{k \in I}\{x_k\}\Big\}.
\end{gather*}
\item[($\mathcal{G}$)] For each (possibly empty) $I \subset \{1,\ldots,d\}$ let $p_I \in [0,1]$ with $\sum_{I\,:\,k \notin I}p_I<1$ for each $k=1,\ldots,d$ and $\sum_{I}p_I=1$. The probabilities $p_I$ define a probability law on the power set of $\{1,\ldots,d\}$. Let $S_1,S_2,\ldots$ be an iid sequence drawn from this law and denote by $G_I$ the smallest $n \in \IN$ such that $S_n=I$. Notice that $G_I \sim \mathcal{G}(p_I)$. We define the random vector $\bm{X}$ with values in $\IN^d$ by
\begin{gather*}
X_k:=\min\{G_I\,:\,k \in I\},\quad k=1,\ldots,d.
\end{gather*}
Then $\bm{X}$ satisfies the multivariate lack-of-memory property. Furthermore, the survival function of $\bm{X}$ equals
\begin{gather*}
\bar{F}(\bm{n})=\IP(\bm{X}>\bm{n}) = \prod_{k=1}^{d}\,\Big( \sum_{I\,:\,\{\sigma_{\bm{n}}(k),\ldots,\sigma_{\bm{n}}(d)\}\cap I = \emptyset}p_I\Big)^{n_{[k]}-n_{[k-1]}},
\end{gather*}
where $\sigma_{\bm{n}}$ denotes a permutation of $\{1,\ldots,d\}$ such that $n_{\sigma_{\bm{n}}(1)} \leq \ldots \leq n_{\sigma_{\bm{n}}(d)}$, and $n_0:=0$, for $\bm{n}=(n_1,\ldots,n_d) \in \IN_0^d$.
\end{itemize}
\end{example}
The probability distribution in part ($\mathcal{E}$) of Example \ref{ex_lom} is called \emph{Marshall-Olkin distribution}. It is named after \cite{marshall67}. The probability distribution in part ($\mathcal{G}$) of Example \ref{ex_lom} is called \emph{wide-sense geometric distribution}. The stochastic model has been introduced in \cite{arnold75}. The presented form of the survival function is computed in \cite{mai13}. The following lemma shows that the multivariate stochastic models in Example \ref{ex_lom} define precisely the multivariate analogues of the univariate exponential and geometric laws, when defined via the lack-of-memory property. Thus, it constitutes a multivariate extension of Lemma \ref{lemma_lom}.
\begin{lemma}[Characterization of lack-of-memory for $d \geq 1$]
$\mbox{}$
\begin{itemize}
\item[($\mathcal{E}$)] The $d$-variate Marshall-Olkin distribution is the only probability law with support $[0,\infty)^d$ satisfying the lack-of-memory property.
\item[($\mathcal{G}$)] The $d$-variate wide-sense geometric distribution is the only probability law with support $\IN^d$ satisfying the lack-of-memory property.
\end{itemize}
\end{lemma}
\begin{proof}
Part ($\mathcal{E}$) is due to the original reference \cite{marshall67}, while part ($\mathcal{G}$) is shown in \cite{mai13}.
\end{proof}

\begin{example}[Narrow-sense geometric law]\label{ex_narrowgeo}
If $\bm{Y}$ has a Marshall-Olkin distribution and we define $\bm{X}:=(\lceil Y_1 \rceil,\ldots,\lceil Y_d \rceil)$, then $\bm{X}$ is said to have the \emph{narrow-sense geometric distribution}. As the nomenclature suggests, the narrow-sense geometric distribution is a subset of the wide-sense geometric distribution in dimensions $d \geq 2$ (and identical for $d=1$), which is very easy to see by the characterizing lack-of-memory property of the Marshall-Olkin law. Not every wide-sense geometric law can be constructed like this, i.e.\ the narrow-sense family defines a proper subset of the wide-sense family. This indicates that for $d \geq 2$ the structure of the discrete lack-of-memory property is more delicate than the structure of its continuous counterpart. For example, while two components of a random vector with Marshall-Olkin distribution or narrow-sense geometric distribution cannot be negatively correlated, two components of a random vector with wide-sense geometric distribution can be, see \cite{mai13} for details.  
\end{example}

\subsection{Infinite divisibility and L\'evy subordinators} \label{rmk_ID}
The concept of infinite divisibility is of fundamental importance in the present section, but also in Sections \ref{sec:evd} and \ref{sec:exshock} below. Thus, we briefly recall the required background in the present paragraph. For an elaborate textbook treatment we refer to \cite{sato99}. The concept of a L\'evy subordinator plays an essential role when studying the conditionally iid subfamily of the Marshall-Olkin distribution, a result first discovered in \cite{maischerer11}. Recall that a c\`adl\`ag stochastic process $Z=\{Z_t\}_{t \geq 0}$ with $Z_0=0$ is called a \emph{L\'evy process} if it has stationary and independent increments, which means that:
\begin{itemize}
\item[(i)] The law of $Z_{t+h}-Z_t$ is independent of $t \geq 0$ for each $h \geq 0$, i.e.\ $Z_{t+h}-Z_t \stackrel{d}{=}Z_h$ .
\item[(ii)] $Z_{t_2}-Z_{t_1},\ldots,Z_{t_n}-Z_{t_{n-1}}$ are independent for $0 \leq t_1<\ldots<t_n$.
\end{itemize}
Hence, L\'evy processes are the continuous-time equivalents of discrete-time random walks. A non-decreasing L\'evy process is called a \emph{L\'evy subordinator}. However, there is one fundamental difference between a random walk and a L\'evy process: the probability law of the increments in a random walk is arbitrary on $\IR$, whereas the law of the increments in a L\'evy process need to satisfy a certain compatibility condition with respect to time, as increments of arbitrarily large time span can be considered. Concretely, it is immediate from the definition of a L\'evy process $Z=\{Z_t\}_{t \geq 0}$ that the probability law of $Z_1$ is infinitely divisible. Recall that a random variable $X$ is called \emph{infinitely divisible} if for each $n\in \IN$ there exist iid random variables $X_1^{(n)},\ldots,X_n^{(n)}$ such that $X\stackrel{d}{=}X^{(n)}_1+\ldots+X_n^{(n)}$. Furthermore, if $X$ has an infinitely divisible probability law, there exists a L\'evy process $Z=\{Z_t\}_{t \geq 0}$, which is uniquely determined in law, such that $Z_1 \stackrel{d}{=} X$. As a consequence, a L\'evy subordinator $Z$ is uniquely determined in distribution by the law of $Z_1$, or analytically by the function $\Psi(x):=-\log(\IE[\exp(-x\,Z_1)])$, $x \geq 0$. One calls $\Psi$ the \emph{Laplace exponent} of the infinitely divisible random variable $Z_1$ (or of the L\'evy subordinator $Z$). The function $\Psi$ is a so-called \emph{Bernstein function}, which means that it is infinitely often differentiable on $(0,\infty)$ and the derivative $\Psi^{(1)}$ is completely monotone, i.e.\ $(-1)^{k+1}\,\Psi^{(k)} \geq 0$ for all $k \geq 1$, see \cite{berg84,schilling10} for textbook treatments on the topic. The value $\Psi(0)$ is by definition equal to zero but we might have a jump at zero meaning that $\Psi(x)>\epsilon>0$ for all $x>0$ is possible. Intuitively, this is the case if and only if $\IP(Z_t=\infty)>0$ for $t>0$, and in this case one sometimes also speaks of a \emph{killed} L\'evy subordinator. 

\subsection{Analytical characterization of exchangeability and conditionally iid}
By Lemma \ref{lemma_exchangeable} a random vector $\bm{X}$ with either Marshall-Olkin distribution or wide-sense geometric distribution can only be conditionally iid if it is exchangeable. An elementary computation shows that the Marshall-Olkin distribution (resp.\ wide-sense geometric distribution) is exchangeable if and only if its parameters $\lambda_I$ (resp.\ $p_I$) depend on the indexing subsets $I$ only through their cardinality $|I|$. In this exchangeable case, we denote these parameters by $\lambda_1,\ldots,\lambda_d$ (resp.\ $p_0,p_1,\ldots,p_d$), with subindices denoting the possible cardinalities, i.e.\ $\lambda_k:=\lambda_{\{1,\ldots,k\}}$ and $p_k:=p_{\{1,\ldots,k\}}$, and combinatorial computations show that the survival function $\bar{F}$ of $\bm{X}$ takes the convenient algebraic form
\begin{gather}
\bar{F}(\bm{x}) = \prod_{k=1}^{d}b_k^{x_{[d-k+1]}-x_{[d-k]}},\quad \bar{F}(\bm{n}) = \prod_{k=1}^{d}b_k^{n_{[d-k+1]}-n_{[d-k]}},
\label{sf_lom}
\end{gather}
for either $x_1,\ldots,x_d \in [0,\infty)$ with $x_0:=0$ (in the Marshall-Olkin case) or $n_1,\ldots,n_d \in \IN_0$ with $n_0:=0$ (in the wide-sense geometric case), and with\footnote{The empty product is conveniently defined to be equal to one, i.e\ $\prod_{i=1}^{0}:=1$.}
\begin{align}
b_k &:= \prod_{i=1}^{k}\exp\Big\{- \sum_{j=0}^{d-i}\binom{d-i}{j}\,\lambda_{j+1} \Big\},  & \mbox{(Marshall-Olkin case)} \label{reparamexp}\\
b_k &:= \sum_{i=0}^{d-k}\binom{d-k}{i}\,p_{i}, & \mbox{(wide-sense geometric case)},\label{reparamgeo}
\end{align}
for $k=0,\ldots,d$. While the parameters $\lambda_k$ (resp.\ $p_k$) are intuitive since they allow for the probabilistic interpretations according to Example \ref{ex_lom}, the re-parameterization in terms of the new parameters $b_k$ is more convenient with regards to finding an answer to the question: when is $\bm{X}$ conditionally iid? The main result in this regard is stated in Theorem \ref{thm_ciidlom} below, which requires the notion of $d$-monotone sequences and log-$d$-monotone sequences. The concept of $d$-monotonicity as well as the notations $\mathcal{M}_d$ and $\mathcal{M}_{\infty}$ have already been introduced in paragraph \ref{subsec_hausdorff}, the related concept of log-$d$-monotonicity is introduced in the following definition.

\begin{definition}[Log- monotone sequences]
For $d \in \IN$, a finite sequence $(b_0,b_1,\ldots,b_{d}) \in (0,\infty)^{d+1}$ is said to be \emph{log-$d$-monotone} if $\nabla^{d-k}\log(b_k) \geq 0$ for $k=0,1,\ldots,d-1$. An infinite sequence $\{b_k\}_{k \in \IN_0}$ with positive members is said to be \emph{completely log-monotone} if $(b_0,\ldots,b_{d})$ is log-$d$-monotone for each $d \geq 1$.
\end{definition}
The notion of a log-$d$-monotone sequence is less intuitive than that of a $d$-monotone sequence. First notice that, in contrast to the definition of a $d$-monotone sequence in paragraph \ref{subsec_hausdorff}, $\log(b_{d}) \geq 0$ needs not hold for a log-$d$-monotone sequence, which is explained by the following useful relationship between $(d-1)$-monotonicity and log-$d$-monotonicity. It helps to transform statements involving log-$d$-monotonicity into statements involving only the simpler notion of $(d-1)$-monotonicity\footnote{This statement simply follows from the fact that $\log(1)=0$ and $\nabla^{d-k-1}b_k = \nabla^{d-k}\log(\tilde{b}_k)$ with $\tilde{b}_k:=\exp(-\sum_{i=0}^{k-1}b_i)$ for $k=0,\ldots,d$, with an empty sum being conveniently defined as zero.}:
\begin{align}
& (b_0,\ldots,b_{d-1}) \; (d-1)\mbox{-monotone }\nonumber\\
& \quad \Leftrightarrow \Big( 1,e^{-b_0},e^{-(b_0+b_1)},\ldots,e^{-\sum_{i=0}^{d-1}b_i}\Big)\mbox{ log-}d\mbox{-monotone}.
\label{corresplogdandd}
\end{align}
The set of all log-$d$-monotone sequences starting with $b_0=1$ will be denoted by $\mathcal{LM}_d$ in the following. Similarly, $\mathcal{LM}_{\infty}$ denotes the sets of completely log-monotone sequences starting with $b_0=1$. \cite[Proposition 4.4]{mai13} shows that $\{b_k\}_{k \in \IN} \in \mathcal{LM}_{\infty}$ if and only if $\{b_k^t\}_{k \in \IN} \in \mathcal{M}_{\infty}$ for arbitrary $t>0$. In particular, $\mathcal{LM}_{\infty} \subset \mathcal{M}_{\infty}$. Theorem \ref{thm_ciidlom} below provides a second result, besides Theorem \ref{thm_definetti01}, showing that whether or not a (log-) $d$-monotone sequence can be extended to a completely (log-) monotone sequence plays an important role in the context of the present survey. 
\par
In order to better understand the following theorem it is helpful to know that the Laplace exponent $\Psi$ of a L\'evy subordinator $Z$ is already completely determined by its values on $\IN$, i.e.\ by the sequence $\{\Psi(k)\}_{k \in \IN_0}$. Furthermore, the sequence $\{\exp(-\Psi(k))\}_{k \in \IN_0}$ equals the moment sequence of the random variable $\exp(-Z_1)$, so lies in $\mathcal{M}_{\infty}$ by the little moment problem, see paragraph \ref{subsec_hausdorff}. Since for arbitrary $t>0$ even the sequence $\{\exp(-t\,\Psi(k))\}_{k \in \IN_0}$ lies in $\mathcal{M}_{\infty}$ as the moment sequence of $\exp(-Z_t)$, the sequence $\{\exp(-\Psi(k))\}_{k \in \IN_0}$ even lies in the smaller set $\mathcal{LM}_{\infty}$ of completely log-monotone sequences. The subset $\mathcal{LM}_{\infty} \subsetneq \mathcal{M}_{\infty}$ corresponds to precisely the infinitely divisible laws on $[0,\infty]$, which is the discrete analogue of the well known statement that $\exp(-t\,\Psi)$ is a completely monotone function for arbitrary $t>0$ if and only if $\Psi^{'}$ is completely monotone. With this information and the information of paragraph \ref{rmk_ID} as background the following theorem is now quite intuitive. 
\par
Theorem \ref{thm_ciidlom} solves Problem \ref{motivatingproblemrefined} for the property (P) of ``satisfying the multivariate lack-of-memory property''.

\begin{theorem}[Lack-of-memory, exchangeability and conditionally iid]\label{thm_ciidlom}
$\mbox{}$
\begin{itemize}
\item[($\mathcal{E}$)] The function (\ref{sf_lom}) is a survival function (of some $\bm{X}$) with support $[0,\infty)^d$ if and only if we have $(b_0,\ldots,b_d) \in \mathcal{LM}_{d}$. Furthermore, the associated exchangeable Marshall-Olkin distribution admits a stochastic representation that is conditionally iid if there exist $b_{d+1},b_{d+2},\ldots$ such that $\{b_k\}_{k \in \IN_0} \in \mathcal{LM}_{\infty}$. To wit, in this case there exists a (possibly killed) L\'evy subordinator $Z=\{Z_t\}_{t \geq 0}$, determined in law via
\begin{gather}
b_k:=\IE\Big[ e^{-k\,Z_1}\Big],\quad k \in \IN_0,
\label{corresp_levy_logcm}
\end{gather}
such that $\bm{X}$ has the same distribution as the vector defined in (\ref{canonical_construction_2}).
\item[($\mathcal{G}$)] The function (\ref{sf_lom}) is a survival function (of some $\bm{X}$) with support $\IN^d$ if and only if we have $(b_0,b_1,\ldots,b_d) \in \mathcal{M}_{d}$. Furthermore, the associated exchangeable wide-sense geometric distribution admits a stochastic representation that is conditionally iid if there exist $b_{d+1},b_{d+2},\ldots$ such that $\{b_k\}_{k \in \IN_0} \in \mathcal{M}_{\infty}$. To wit, in this case there exists an iid sequence $Y_1,Y_2,\ldots$ of random variables taking values in $[0,\infty]$, determined in law via
\begin{gather}
b_k:=\IE\Big[ e^{-k\,Y_1}\Big],\quad k \in \IN_0,
\label{corresp_littlemoment}
\end{gather}
such that $\bm{X}$ has the same distribution as the vector defined in (\ref{canonical_construction_2}) when
\begin{gather*}
Z_t:=Y_1+Y_2+\ldots+Y_{\lfloor t \rfloor},\quad t \geq 0.
\end{gather*}
\end{itemize}
\end{theorem}
\begin{proof}
Part ($\mathcal{E}$) is due to \cite{maischerer09,maischerer11}, while part ($\mathcal{G}$) is due to \cite{mai13}. 
\par
First, we observe that once the correspondence between $\mathcal{M}_{d}$ and the wide-sense geometric law is established, the correspondence between $\mathcal{LM}_{d}$ and the narrow-sense geometric law (or, algebraically equivalent, its continuous counterpart the Marshall-Olkin law) follows from (\ref{corresplogdandd}) together with (\ref{reparamexp}) and (\ref{reparamgeo}). This is because the $\lambda_j$ in (\ref{reparamexp}) are arbitrary non-negative numbers, and the $p_i$ in (\ref{reparamgeo}) are also arbitrary non-negative up to scaling (i.e.\ with an additional scale factor $c>0$ we have that $c\,(p_0,\ldots,p_{d-1})$ and $(\lambda_1,\ldots,\lambda_d)$ both run through all of $[0,\infty)^d \setminus \{(0,\ldots,0)\}$, noticing that $p_d$ is determined by $p_0,\ldots,p_{d-1}$). Concretely, by the correspondence between $\mathcal{M}_{d}$ and the wide-sense geometric law, we obtain a correspondence between $\mathcal{M}_{d}$ and $[0,\infty)^{d}\setminus \{(0,\ldots,0)\}$ up to scaling in (\ref{reparamgeo}). In particular, the property of being $d$-monotone is not affected by $c$. Replacing the $\lambda_j$ in (\ref{reparamexp}) by $c\,p_{j-1}$ and making use of (\ref{corresplogdandd}), we then end up with the correspondence between $\mathcal{LM}_{d}$ and the Marshall-Olkin law. To establish the correspondence between $\mathcal{M}_{d}$ and the wide-sense geometric law is really only a tedious algebraic computation, see \cite{mai13} for details. Essentially, $d$-monotonicity enters the scene precisely for the same reason as in paragraph \ref{subsec_hausdorff}.
\par
Regarding the conditionally iid subfamliy, the crucial insight is that $\mathcal{M}_{\infty}$ stands in one-to-one relation with the set of probability measures on $[0,\infty]$ via (\ref{corresp_littlemoment}), which is exactly the well-known statement of the little moment problem, only formulated for the compact interval $[0,\infty]$ instead of the more usual interval $[0,1]$ via the transformation $-\log$. That the (discrete) random walk construction in part $(\mathcal{G})$ can only be ``made continuous'' in case $Y_1$ is infinitely divisible is very intuitive, and the L\'evy subordinator in part $(\mathcal{E})$ is simply the continuous analogue of the discrete random walk in that case.  
\end{proof}

Since the narrow-sense geometric law of Example \ref{ex_narrowgeo} is a special case of the wide-sense geometric law, it follows that $\mathcal{LM}_d \subsetneq \mathcal{M}_d$, which in fact is not an obvious statement. Furthermore, $\bm{X}$ in part ($\mathcal{G}$) of Theorem \ref{thm_ciidlom} happens to be narrow-sense geometric if and only if the random variable $Y_1$ is infinitely divisible. In fact, the elements of $\mathcal{LM}_{\infty}$ stand in one-to-one correspondence with the family of infinitely divisible laws on $[0,\infty]$ via (\ref{corresp_levy_logcm}), whereas the elements of the larger set $\mathcal{M}_{\infty}$ stand in one-to-one correspondence with the family of arbitrary probability laws on $[0,\infty]$ via (\ref{corresp_littlemoment}), which is just a slight re-formulation of the little moment problem. 

\begin{remark}[Analytical criterion for conditionally iid]\label{remk_truncMOM}
Given an exchangeable random vector $\bm{X}$ with lack-of-memory property and parameters $(b_0,\ldots,b_d)$, Theorem \ref{thm_ciidlom} implies that $\bm{X}$ has a stochastic representation that is conditionally iid if $(b_0,\ldots,b_d)$ can be extended to a completely (log-) monotone sequence. Using (\ref{corresplogdandd}), an element $(b_0,\ldots,b_{d}) \in \mathcal{LM}_{d}$ is extendible to an element in $\mathcal{LM}_{\infty}$ if and only if the $(d-1)$-monotone sequence $(-\log(b_1/b_0),\ldots,-\log(b_d/b_{d-1}))$ is extendible to a completely monotone sequence. Thus, we can concentrate on the completely monotone case. Deciding whether a $d$-monotone sequence can be extended to a completely monotone sequence is the truncated Hausdorff moment problem again, see Section \ref{subsec_hausdorff}. This means that an effective analytical criterion for extendibility is known.
\end{remark}

The following example demonstrates how a parameter sequence $\{b_k\}_{k \in \IN_0}$ for some wide-sense geometric law is conveniently defined via the link to the little moment problem, setting $b_k:=\IE[X^k]$, $k \in \IN_0$, where $X$ is some arbitrary random variable taking values in $[0,1]$.

\begin{example}[A two-parametric family based on the Beta distribution]
Consider a random variable $X$ with density
\begin{gather*}
f_X(x) = \frac{\Gamma(p+q)}{\Gamma(p)\,\Gamma(q)}\,x^{p-1}\,(1-x)^{q-1},\quad 0<x<1,
\end{gather*}
with parameters $p,q>0$, which is a Beta distribution. The moment sequence is known to be\footnote{See, e.g., \cite[p.\ 35]{gupta04}.}
\begin{gather*}
\IE[X^k] = \int_0^{1}f_X(x)\,x^k\,\mathrm{d}x=\frac{\Gamma(p+k)\,\Gamma(p+q)}{\Gamma(p)\,\Gamma(p+q+k)},\quad k \in \IN_0,
\end{gather*}
so that a two-parametric family of $d$-variate wide-sense geometric survival functions (for arbitrary $d \geq 1$) is given by
\begin{gather*}
\bar{F}_{p,q}(\bm{n}) = \Big(\frac{\Gamma(p+q)}{\Gamma(p)}\Big)^{n_{[d]}}\prod_{k=1}^{d}\Big(\frac{\Gamma(p+k)}{\Gamma(p+q+k)}\Big)^{n_{[d-k+1]}-n_{[d-k]}},\quad \bm{n} \in \IN_0^d.
\end{gather*}
The associated probability distribution of $Y_1$ in Theorem \ref{thm_ciidlom}($\mathcal{G}$) is given by $Y_1 \stackrel{d}{=}-\log(X)$, i.e.\ the logarithm of the reciprocal of the Beta distribution in concern. Similarly, making use of (\ref{corresplogdandd}), a two-parametric family of $d$-variate Marshall-Olkin survival functions (for arbitrary $d \geq 1$) is given by
\begin{align*}
\bar{F}_{p,q}(\bm{x})&= \exp\Big\{-\frac{\Gamma(p+q)}{\Gamma(p)}\,\sum_{k=1}^{d} (x_{[d-k+1]}-x_{[d-k]})\,\sum_{i=0}^{k-1}\frac{\Gamma(p+i)}{\Gamma(p+q+i)}\Big\}\\
& = \exp\Big\{-\frac{\Gamma(p+q)}{\Gamma(p)}\,\sum_{k=1}^{d} \frac{\Gamma(p+k-1)}{\Gamma(p+q+k-1)}\,x_{[d-k+1]}\Big\},\quad \bm{x} \in [0,\infty)^d.
\end{align*}
In the special case when $q=2$, the L\'evy subordinator in Theorem \ref{thm_ciidlom}($\mathcal{E}$) is of compound Poisson type with intensity $p+1$ and jumps that are exponentially distributed with parameter $p$.
\end{example}

\newpage

\section{Max-/ min-stable laws and extreme-value copulas}\label{sec:evd}
Throughout this paragraph, for the sake of a more compact notation we implicitly make excessive use of the abbreviations $f(0):=\lim_{x \searrow 0}f(x)$ and $f(\infty):=\lim_{x \rightarrow \infty}f(x)$ for functions $f:(0,\infty) \rightarrow (0,\infty)$, provided the respective limits exist in $[0,\infty]$. 
\subsection{Max-/ min-stability and multivariate extreme-value theory}
\begin{definition}[Max- and min-stability]
We denote by $F$ (resp.\ $\bar{F}$) the $d$-variate distribution function (resp.\ survival function) of some $d$-dimensional random vector $\bm{Y}=(Y_1,\ldots,Y_d)$ (resp.\ $\bm{X}=(X_1,\ldots,X_d)$). 
\begin{itemize}
\item[(a)] (The probability law of) $\bm{Y}$ is said to be \emph{max-stable} if for arbitrary $t>0$ there are $\alpha_i(t)>0$, $\beta_i(t) \in \IR$ such that 
\begin{align*}
F(\bm{x})^t=F\big( \alpha_1(t)\,x_1+\beta_1(t),\ldots,\alpha_d(t)\,x_d+\beta_d(t)\big).
\end{align*}
In this case, we also say that $F$ is \emph{max-stable}. In words, $F^t$ is again a distribution function and equals $F$ modulo a linear transformation of its arguments.
\item[(b)] (The probability law of) $\bm{X}$ is said to be \emph{min-stable} if for arbitrary $t>0$ there are $\alpha_i(t)>0$, $\beta_i(t) \in \IR$ such that 
\begin{align*}
\bar{F}(\bm{x})^t=\bar{F}\big( \alpha_1(t)\,x_1+\beta_1(t),\ldots,\alpha_d(t)\,x_d+\beta_d(t)\big).
\end{align*}
In this case, we also say that $\bar{F}$ is \emph{min-stable}. In words, $\bar{F}^t$ is again a survival function and equals $\bar{F}$ modulo a linear transformation of its arguments.
\end{itemize}
\end{definition}
If $\bm{Y}$ is max-stable and $\bm{Y}^{(i)}$ are independent copies of $\bm{Y}$, then for arbitrary $n \in \IN$ we observe 
\begin{align*}
\bm{Y} \stackrel{d}{=}\Big(\frac{\max_{i=1}^{n}\{Y^{(i)}_1\}-\beta_1(1/n)}{\alpha_1(1/n)},\ldots,\frac{\max_{i=1}^{n}\{Y^{(i)}_d\}-\beta_d(1/n)}{\alpha_d(1/n)}\Big).
\end{align*}
Similarly, if $\bm{X}$ is min-stable this means
\begin{align*}
\bm{X} \stackrel{d}{=}\Big(\frac{\min_{i=1}^{n}\{X^{(i)}_1\}-\beta_1(1/n)}{\alpha_1(1/n)},\ldots,\frac{\min_{i=1}^{n}\{X^{(i)}_d\}-\beta_d(1/n)}{\alpha_d(1/n)}\Big).
\end{align*}
In words, the component-wise re-scaled maxima of iid copies of $\bm{Y}$ (resp.\ minima of iid copies of $\bm{X}$) have the same distribution as $\bm{Y}$ (resp.\ $\bm{X}$). 
\par
Max and min-stability play a central role in multivariate extreme-value theory, as will briefly be explained. If $\bm{V}^{(i)}$ are independent copies of some random vector $\bm{V}=(V_1,\ldots,V_d)$, one is interested in the probability law of the vectors of component-wise maxima, that is
\begin{gather*}
\big(\max_{i=1}^{n}\{V^{(i)}_1\},\ldots,\max_{i=1}^{n}\{V^{(i)}_d\} \big),\quad n \in \IN.
\end{gather*}
If one can find sequences $\alpha_1(n),\ldots,\alpha_d(n)>0$ and $\beta_1(n),\ldots,\beta_d(n) \in \IR$ such that the re-scaled vector
\begin{align*}
\Big(\frac{\max_{i=1}^{n}\{V^{(i)}_1\}-\beta_1(n)}{\alpha_1(n)},\ldots,\frac{\max_{i=1}^{n}\{V^{(i)}_d\}-\beta_d(n)}{\alpha_d(n)}\Big)
\end{align*}
converges in distribution to some $\bm{Y}=(Y_1,\ldots,Y_d)$, then one says that $\bm{Y}$ has a \emph{multivariate extreme-value distribution}. 
A classical result in multivariate extreme-value theory states that $\bm{Y}$ has a multivariate extreme-value distribution if and only if $\bm{Y}$ is max-stable, see, e.g., \cite[pp.\ 172-174]{joe97}. 
\par
Since $\bm{Y}$ is max-stable if and only if $-\bm{Y}$ is min-stable (obviously), max- and min-stability can be studied jointly by focusing on one of the two concepts. Classical extreme-value theory textbooks typically focus on max-stability and further subdivide the study of the probability law of max-stable $\bm{Y}$ into two sub-studies: 
\begin{itemize}
\item[(i)] By the Fisher-Tippett-Gnedenko Theorem, the univariate distribution function $F_k$ of each component $Y_k$ necessarily belongs to either the Gumbel, the Fr\'echet or the Weibull family, see \cite[Chapter 2, p.\ 45 ff]{beirlant04} for background.  
\item[(ii)] Having understood the univariate marginal distribution functions $F_1,\ldots,F_d$ according to (i), the distribution function $F$ of $\bm{Y}$ necessarily takes the form
\begin{gather*}
F(\bm{x}) = C\big( F_1(x_1),\ldots,F_d(x_d)\big),
\end{gather*}
for a copula $C:[0,1]^d \rightarrow [0,1]$ with the characterizing property that $C(\bm{u})^t=C(u_1^t,\ldots,u_d^t)$ for each $t>0$, a so-called \emph{extreme-value copula}.
\end{itemize}
In order to focus on a deeper understanding of extreme-value copulas it is convenient to normalize the margins $F_1,\ldots,F_d$. In classical extreme-value theory, it is standard to normalize to standardized Fr\'echet distributions, i.e.\ $F_k(x)=\exp(-\lambda_k/x)\,1_{\{x>0\}}$ for some $\lambda_k>0$. Furthermore, we observe that $\bm{X}:=(1/Y_1,\ldots,1/Y_d)$ is well-defined, $X_k$ is exponential with rate $\lambda_k$, and $\bm{X}$ is min-stable (since $x \mapsto 1/x$ is strictly decreasing, so max-stability of $\bm{Y}$ is flipped to min-stability of $\bm{X}$). The vector $\bm{X}$ is thus called \emph{min-stable multivariate exponential} and has survival function
\begin{gather*}
\bar{F}(\bm{x}) =\IP(\bm{X}>\bm{x})=\IP(\bm{Y}<1/\bm{x})= C\big( e^{-\lambda_1\,x_1},\ldots,e^{-\lambda_d\,x_d}\big),
\end{gather*}
with extreme-value copula $C$. The survival function $\bar{F}$ is min-stable, satisfying 
\begin{gather}
\bar{F}(\bm{x})^t=\bar{F}(t \,\bm{x}),\quad t>0.
\label{analytical_minstable}
\end{gather}
The analytical property (\ref{analytical_minstable}) characterizes the concept of min-stable multivariate exponentiality on the level of survival functions, and serves as a convenient starting point to study the conditionally iid subfamily (of extreme-value copulas, resp.\ min-stable multivariate exponential distributions). For a given extreme-value copula $C$ it further turns out convenient to consider its so-called \emph{stable tail dependence function}
\begin{gather*}
\ell(\bm{x}) := -\log\Big( C\big( e^{-x_1},\ldots,e^{-x_d}\big)\Big),\quad \bm{x} \in [0,\infty)^d,
\end{gather*}
which satisfies $\ell(t\,\bm{x})=t\,\ell(\bm{x})$. Clearly, $\ell$ determines $C$ and $C$ determines $\ell$, so that investigating $\ell$ instead of $C$ is just a matter of convenience. Wrapping up, a min-stable multivariate exponential distribution is fully determined by the rates $(\lambda_1,\ldots,\lambda_d)$ specifying the one-dimensional exponential margins, and by a stable tail dependence function $\ell$ which stands in a one-to-one relationship with the associated extreme-value copula $C$.

\subsection{Analytical characterization of conditionally iid}
In the sequel, we are interested in the question: when is a min-stable multivariate exponential vector $\bm{X}$, i.e.\ one whose survival function satisfies (\ref{analytical_minstable}), conditionally iid? We start with two important examples.
\begin{example}[Independent exponentials]\label{ex_iid}
If the components $X_1,\ldots,X_d$ of $\bm{X}$ are iid, then we only need to consider the law of $X_1$. By definition, $X_1$ must have an exponential law, so there is some $\lambda>0$ such that for arbitrary $t>0$ we have
\begin{gather*}
\bar{F}(\bm{x})^t = \Big(\prod_{k=1}^{d}e^{-\lambda\,x_k}\Big)^t= e^{-t\,\lambda\,\sum_{k=1}^{d}x_k}=\prod_{k=1}^{d}e^{-\lambda\,t\,x_k}=\bar{F}(t\,\bm{x}).
\end{gather*}
Consequently, $\bm{X}$ is min-stable multivariate exponential. The associated stable tail dependence function is $\ell(\bm{x})=x_1+\ldots+x_d=\norm{x}_1$.
\end{example}

For arbitrary $c \geq 0$ we introduce the notation $\mathfrak{H}_{+,c} \subset \mathfrak{H}_{+}$ for distribution functions of non-negative random variables with mean equal to $c$. For $G \in \mathfrak{H}_+$ we further denote by $M_G:=\int_0^{\infty}1-G(x)\,\mathrm{d}x \in [0,\infty]$ its mean.

\begin{example}[An important semi-parametric family]\label{ex_lF}
Let $G \in\mathfrak{H}_{+}$ with $0<M_G<\infty$. With an iid sequence of unit exponentials $\eta_1,\eta_2,\ldots$ we consider the stochastic process
\begin{gather*}
Z_t:=\sum_{n \geq 1}-\log\Big\{ G\Big( \frac{\eta_1+\ldots+\eta_n}{t}-\Big)\Big\},\quad t \geq 0,
\end{gather*}
taking values in $[0,\infty]$. It is not difficult to see that $H:=1-\exp(-Z)$ takes values in $\mathfrak{H}_+$. Consequently, we may define a conditionally iid random vector $\bm{X}$ via the canonical stochastic model (\ref{canonical_construction_2}) from this process $H$. Conditioned on $H$, the components of $\bm{X}$ are iid with distribution function $H$. It turns out that $\bm{X}$ is min-stable multivariate exponential. To see this, we recall that the increasing sequence $\{\eta_1+\ldots+\eta_n\}_{n \geq 1}$ equals the enumeration of the points of a Poisson random measure on $[0,\infty)$ with intensity measure equal to the Lebesgue measure. This implies with the help of \cite[Proposition 3.6]{resnick87} in $(\ast)$ below that the survival function $\bar{F}$ of $\bm{X}$ is given by
\begin{align*}
\bar{F}(\bm{x}) &= \IP(Z_{x_1} \leq \epsilon_1,\ldots,Z_{x_d} \leq \epsilon_d) = \IE\Big[ e^{-\sum_{k=1}^{d}Z_{x_k}}\Big]\\
& = \IE\Big[ \exp\Big\{ -\sum_{n \geq 1}-\log\Big\{ \prod_{k=1}^{d}\,G\Big( \frac{\eta_1+\ldots+\eta_n}{x_k}-\Big)\Big\}\Big\}\Big]\\
& \stackrel{(\ast)}{=} \exp\Big( -\int_0^{\infty} 1-\prod_{k=1}^{d}\,G\Big( \frac{u}{x_k}\Big)\,\mathrm{d}u \Big).
\end{align*}
We introduce the notation 
\begin{gather*}
\ell_G(\bm{x}):=-\log(\bar{F}(\bm{x}/M_G)) = \frac{1}{M_G}\,\int_0^{\infty} 1-\prod_{k=1}^{d}\,G\Big( \frac{u}{x_k}\Big)\,\mathrm{d}u, 
\end{gather*}
and we observe by substitution that $t\,\ell_G(\bm{x})=\ell_G(t\,\bm{x})$ for arbitrary $t>0$. This implies $\bar{F}(\bm{x})^t=\bar{F}(t\,\bm{x})$, so $\bm{X}$ is min-stable multivariate exponential. The function $\ell_G$ is the stable tail dependence function of $\bm{X}$. The constant $M_G$ equals the exponential rate of the exponential random variables $X_1,\ldots,X_d$.
\end{example}
The main theorem in this section states that Examples \ref{ex_iid} and \ref{ex_lF} are general enough to understand the structure of the set of all infinite exchangeable sequences $\{X_k\}_{k \in \IN}$ whose finite-dimensional margins are both min-stable multivariate exponential and conditionally iid. Concretely, Theorem \ref{thm_minstable} solves Problem \ref{motivatingproblemrefined} for the property (P) of ``having a min-stable multivariate exponential distribution (in some dimension)''. In analytical terms, it states that the stable tail dependence function associated with the extreme-value copula of a conditionally iid min-stable multivariate exponential random vector is a convex mixture of stable tail dependence functions having the structural form as presented in Examples \ref{ex_iid} and \ref{ex_lF}. 

\begin{theorem}[Which min-stable laws are conditionally iid?]\label{thm_minstable}
Let $\{X_k\}_{k \in \IN}$ be an infinite exchangeable sequence of positive random variables such that $\bm{X}=(X_1,\ldots,X_d)$ is min-stable multivariate exponential for all $d \in \IN$. Assume that $\{X_k\}_{k \in \IN}$ is not iid, i.e.\ not given as in Example \ref{ex_iid}. Then there exists a unique triplet $(b,c,\gamma)$ of two constants $b \geq 0$, $c>0$ and a probability measure $\gamma$ on $\mathfrak{H}_{+,1}$, such that $X_k$ is exponential with rate $b+c$ for each $k \in \IN$ and the stable tail dependence function of $\bm{X}$ equals
\begin{gather*}
\ell(\bm{x}):=-\log\Big\{\bar{F}\Big(\frac{\bm{x}}{b+c}\Big)\Big\}=\frac{b}{b+c}\,\norm{x}_1+\frac{c}{b+c}\,\int_{\mathfrak{H}_{+,1}} \ell_G(\bm{x})\,\gamma(\mathrm{d}G).
\end{gather*}
In probabilistic terms, the random distribution function $H$, defined as the limit of empirical distribution functions of the $\{X_k\}_{k \in \IN}$ as in Lemma \ref{lemma_condGC}, necessarily satisfies $H \stackrel{d}{=}1-\exp(-Z)$ with
\begin{gather}
Z_t=b\,t+c\,\sum_{n \geq 1}-\log\Big\{ G^{(n)}_{\frac{\eta_1+\ldots+\eta_n}{t}-}\Big\},\quad t \geq 0,
\label{canonicalstrongIDTsub}
\end{gather}
where $G^{(k)}$ is an iid sequence drawn from the probability measure $\gamma$, independent of the iid unit exponentials $\eta_1,\eta_2,\ldots$.
\end{theorem}
\begin{proof}
A proof consists of three steps, which have been accomplished in the three references \cite{maischerer13,kopp18,mai18b}, respectively, and which are sketched in the sequel.
\begin{itemize}
\item[(i)] For $Z =-\log(1-H)$ with $H$ as defined in Lemma \ref{lemma_condGC} from the sequence $\{X_k\}_{k \in \IN}$, \cite[Theorem 5.3]{maischerer13} shows that
\begin{gather}
Z \stackrel{d}{=}\Big\{ \sum_{i=1}^{n}Z^{(i)}_{\frac{t}{n}}\Big\}_{t \geq 0},\quad n \in \IN,
\label{strongIDT}
\end{gather}
where $Z^{(i)}$ are independent copies of $Z$. Conversely, it is shown that if $Z$ is non-decreasing and satisfies (\ref{strongIDT}), then $1-\exp(-Z)$ is an element of $\Theta_d^{-1}(\mathfrak{M}_{\ast\ast})$, when $\mathfrak{M}_{\ast\ast}$ is as in Problem \ref{motivatingproblemrefined} and (P) is the property of ``having a min-stable multivariate exponential distribution''.
\item[(ii)] \cite{kopp18} show that a non-negative stochastic process $Z$ satisfying (\ref{strongIDT}) admits a series representation of the form
\begin{gather*}
\{Z_t\}_{t \geq 0} \stackrel{d}{=} \Big\{b\,t+\sum_{n \geq 1}f^{(n)}_{\frac{t}{\eta_1+\ldots+\eta_n}} \Big\}_{t \geq 0},
\end{gather*}
where $f^{(n)}$ are iid copies of some c\`adl\`ag stochastic process $f$ with $f_0=0$ satisfying some integrability condition, and $b \in \IR$.
\item[(iii)] \cite{mai18b} proves that in the series representation in (ii) necessarily $b \geq 0$ and $f$ is almost surely non-decreasing. Furthermore, the integrability condition on $f$ can be re-phrased to say that $t \mapsto \tilde{G}_t:=\exp(-\lim_{s \downarrow t }f_{1/s})$ defines almost surely the distribution function of some random variable with finite mean $M_{\tilde{G}}=\int_0^{\infty}1-\tilde{G}_t\,\mathrm{d}t>0$. Finally, the distribution function $t \mapsto G_t:=\tilde{G}_{M_{\tilde{G}}\,t}$ has unit mean, and the claimed representation for $\ell$ is obtained when $c:=\IE[M_{\tilde{G}}]$ and $\gamma$ is defined as the probability law of $G$ after an appropriate measure change. That $(b,c,\gamma)$ is unique follows from the normalization to unit mean of $G$ (for each single realization).
\end{itemize} 
\end{proof}
Stochastic processes with property (\ref{strongIDT}) are said to be \emph{strongly infinitely divisible with respect to time (strong IDT)}. Particular examples of strong IDT processes have been studied in \cite{mansuy05,ouknine08,ouknine12}, with an emphasis on the associated multivariate min-stable laws also in \cite{maischerer13,bernhart15,mai18,maischerer18}. 
\par
Every L\'evy process is strong IDT, but the converse needs not hold. For instance, if $Z=\{Z_t\}_{t \geq 0}$ is a non-trivial L\'evy subordinator and $a>b>0$, then the stochastic process $\{Z_{a\,t}+Z_{b\,t}\}_{t \geq 0}$ is strong IDT, but not a L\'evy subordinator. The probability law $\gamma$ in Theorem \ref{thm_minstable} in case of a L\'evy subordinator is specified as the probability law of 
\begin{gather}
G_t = e^{-M}+\Big( 1-e^{-M}\Big)\,1_{\{1-e^{-M}\geq 1/t\}},\quad t \geq 0,
\label{ex_LFCasSTDF}
\end{gather}
with an arbitrary random variable $M$ taking values in $(0,\infty]$. The L\'evy measure of $Z$ and the probability law of $M$ stand in one-to-one relation. We know from the preceding section that if $Z$ is a L\'evy subordinator, the associated element in $\mathfrak{M}_{\ast\ast}$ is a $d$-variate Marshall-Olkin distribution. Indeed, the Marshall-Olkin distribution is one of the most important examples of min-stable multivariate exponential distributions. Two further examples are presented in the sequel.
\begin{example}[The (negative) logistic model]
If we reconsider Example \ref{ex_lF} with the Fr\'echet distribution function $G(x)=\exp(-\{\Gamma(1-\theta)\,x\}^{-1/\theta})$ for $\theta \in (0,1)$, then we observe
\begin{gather*}
\ell_G(\bm{x}) = \Big(\sum_{k=1}^{d}x_k^{\frac{1}{\theta}}\Big)^{\theta}=\norm{\bm{x}}_{\frac{1}{\theta}}.
\end{gather*}
This is the so-called \emph{logistic model}. It is particularly convenient to be looked at from the perspective of conditionally iid models, since the associated strong IDT process $Z$ takes a very simple form, to wit 
\begin{gather*}
Z \stackrel{d}{=} \big\{ S\,t^{\frac{1}{\theta}}\big\}_{t \geq 0},\quad S \mbox{ a }\theta\mbox{-stable random variable, i.e.\ }\IE\Big[ e^{-x\,S}\Big]=e^{-x^{\theta}}.
\end{gather*}
In particular, the resulting extreme-value copula, named \emph{Gumbel copula} after \cite{gumbel60,gumbel61}, is also an Archimedean copula, see Remark \ref{rmk_AC}. In fact, it is the only copula that is both Archimedean and of extreme-value kind, a result first discovered in \cite{genest89}.
\par
A related example is obtained, if we choose the Weibull distribution function $G(x)=1-\exp(-\{\Gamma(\theta+1)\,x\}^{1/\theta})$, which implies
\begin{gather*}
\ell_G(\bm{x}) = \sum_{j=1}^{d}(-1)^{j+1} \sum_{1 \leq i_1<\ldots<i_j \leq d}\Big(\sum_{k=1}^{j}x_{i_k}^{-\theta}\Big)^{-\frac{1}{\theta}}.
\end{gather*}
This is the so-called \emph{negative logistic model}. The associated extreme-value copula is named \emph{Galambos copula} after \cite{galambos75}. There exist many analogies between logistic and negative logistic models, the interested reader is referred to \cite{genest17} for background. In particular, the Galambos copula is the most popular representative of the family of so-called \emph{reciprocal Archimedean copulas} as introduced in \cite{genest18}, see also paragraph \ref{open_iidunif} below.
\end{example}

\begin{example}[A rich parametric family]
For $G \in \mathfrak{H}_{+,1}$ the function $\Psi_G(z):=\int_0^{\infty}1-G(t)^z\,\mathrm{d}t$ defines a Bernstein function with $\Psi_G(1)=1$, see \cite[Lemma 3]{mai18}. This implies for $z \in (0,\infty)$ that $G_z \in \mathfrak{H}_{+,1}$, where $G_z(x):= G(x\,\Psi_G(z))^z$. Consequently, if $M$ is a positive random variable, we may define $\gamma \in M_+^1(\mathfrak{H}_{+,1})$ as the law of $G_M$. The associated stable tail dependence function equals $\ell(\bm{x}) := \IE[\ell_{G_M}(\bm{x})]$. Many parametric models from the literature are comprised by this construction. In particular, Example \ref{ex_lF} corresponds to the case $M \equiv 1$, and if $G(x)=\exp(-1)+(1-\exp(-1))\,1_{\{1-\exp(-1) \geq 1/x\}}$ we observe that $G_M$ equals the random distribution function (\ref{ex_LFCasSTDF}) corresponding to the Marshall-Olkin subfamily. See \cite{maischerer18} for a detailed investigation and applications of this parametric family.
\end{example}

\begin{remark}[Extension to laws with exponential minima]
We have seen that the Marshall-Olkin distribution is a subfamily of min-stable multivariate exponential laws. The seminal reference \cite{esary74} treats both families as multivariate extensions of the univariate exponential law and in the process introduces the even larger family of laws with exponential minima. A random vector $\bm{X}$ is said to have \emph{exponential minima} if $\min\{X_{i_1},\ldots,X_{i_k}\}$ has a univariate exponential law for arbitrary $1 \leq i_1<\ldots i_k \leq d$. Obviously, a min-stable multivariate exponential law has exponential minima, but the converse needs not hold in general. It is shown in \cite{maischerer13} that if $Z=\{Z_t\}_{t \geq 0}$ is a right-continuous, non-decreasing process such that $\IE[\exp(-x\,Z_t)]=\exp(-t\,\Psi(x))$ for some Bernstein function $\Psi$, then $\bm{X}$ as defined in (\ref{canonical_construction_2}) has exponential minima. The process $Z$ is said to be \emph{weakly infinitely divisible with respect to time (weak IDT)}, and - as the nomenclature suggests - every strong IDT process is also weak IDT. However, there exist weak IDT processes which are not strong IDT. Notice in particular that a L\'evy subordinator is uniquely determined in law by the law of $Z_1$ (or equivalently the Bernstein function $\Psi$), but neither strong nor weak IDT processes are determined in law by the law of $Z_1$. If one takes two independent, but different, strong IDT processes $Z^{(1)},Z^{(2)}$ subject to $Z^{(1)}_1 \stackrel{d}{=}Z^{(2)}_1$, then the stochastic process
\begin{gather*}
Z_t:=\begin{cases}
Z^{(1)}_t & \mbox{ if }B=1,\\
Z^{(2)}_t & \mbox{ if }B=0,\\
\end{cases},\quad B \mbox{ independent Bernoulli}\Big(\frac{1}{2}\Big)\mbox{-variate},\quad t \geq 0,
\end{gather*}
is weak IDT, but not strong IDT. On the level of $\bm{X}$ this means that the mixture of two min-stable multivariate exponential random vectors always has exponential minima, but needs not be min-stable anymore.
\end{remark}

\begin{remark}[Archimax copulas]
The study of min-stable multivariate exponentials is analogous to the study of extreme-value copulas. From this perspective, Theorem \ref{thm_minstable} gives us a canonical stochastic model for all conditionally iid extreme-value copulas. Another family of copulas for which we understand the conditionally iid subfamily pretty well is Archimedean copulas, related to $\ell_1$-norm symmetric distributions and mentioned in Remark \ref{rmk_AC}. The family of so-called \emph{Archimax copulas} is a superclass of both extreme-value and Archimedean copulas. It has been studied in \cite{caperaa00,charpentier14} with the intention to create a rich copula family that comprises well-known subfamilies. An extreme-value copula $C$ is conveniently described in terms of its stable tail dependence function. Recall that Theorem \ref{thm_minstable} is formulated in terms of the stable tail dependence function and gives an analytical criterion for $C$ to be conditionally iid. An Archimax copula $C$ is a multivariate distribution function of the functional form
\begin{gather*}
C_{\ell,\varphi}(u_1,\ldots,u_d) = \varphi\Big( \ell\big( \varphi^{-1}(u_1),\ldots, \varphi^{-1}(u_d)\big) \Big).
\end{gather*}
It is recognized that if $\ell(x_1,\ldots,x_d)=\norm{\bm{x}}_1$ then $C_{\ell,\varphi}$ is an Archimedean copula, and if $\varphi(x)=\exp(-x)$, then $C_{\ell,\varphi}$ is an extreme-value copula. By combining our knowledge from Theorems \ref{thm_AC} and \ref{thm_minstable} about Archimedean and extreme-value copulas, it is immediate to show that
\begin{align}
&\Big\{ \gamma \in M_+^1(\mathfrak{H}_{+}) \,:\,\big\{ 1-e^{-Z_{M\,t}}\big\}_{t \geq 0}\sim \gamma ,\,M>0 \mbox{ a positive random variable, and } \nonumber\\
& \qquad \qquad \{Z_t\}_{t \geq 0} \mbox{ non-decreasing strong IDT} \Big\} \subset \Theta_d^{-1}\big(\mathfrak{M}_{\ast\ast}\big),
\label{ext_ARCHIMAX}
\end{align}
when $\mathfrak{M}$ denotes the family of all probability laws with the property (P) of ``having a survival function of the functional form $\varphi \circ \ell$ with $\ell$ some stable tail dependence function''. In this case, the function $\varphi$ equals the Laplace transform of $M$ and $\ell$ is given in terms of a triplet $(b,c,\gamma)$ such as in Theorem \ref{thm_minstable}, associated with the strong IDT process $Z$, and $b+c=1$. Notice that each stable tail dependence function $\ell$ equals the restriction of an orthant-monotonic norm to $[0,\infty)^d$, see \cite{molchanov08}, so that survival functions of the form $\varphi \circ \ell$ are precisely the survival functions that are symmetric with respect to the norm $\ell$.
\end{remark}

\newpage

\section{Exogenous shock models} \label{sec:exshock}
The present section studies a family $\mathfrak{M}$ of multivariate distribution functions that have a stochastic representation according to the following exogenous shock model: We consider some system consisting of $d$ components and interpret the $k$-th component of our random vector $\bm{X}=(X_1,\ldots,X_d)$ with law in $\mathfrak{M}$ as the lifetime of the $k$-th component in our system. A component lives until it is affected by an exogenous shock, and the arrival times of these exogenous shocks are modeled stochastically. For each non-empty subset $I \subset \{1,\ldots,d\}$ of components, we denote by $E_I$ a non-negative random variable. We assume that all $E_I$ are independent and interpret $E_I$ as the arrival time of an exogenous shock affecting all components of our random vector which are indexed by $I$. This means that we define
\begin{gather}
X_k:=\min\{E_I\,:\,k \in I\},\quad k=1,\ldots,d.
\label{exshock_canonic}
\end{gather}
Such exogenous shock models are popular in reliability theory, insurance risk, and portfolio credit risk. Recall from Example \ref{ex_lom}($\mathcal{E}$) that this model is a generalization of the Marshall-Olkin distribution, which arises as special case if all the $E_I$ are exponentially distributed, see also Example \ref{ex_MOasexshock} below. For the sake of clarity, we formally introduce the following definition.
\begin{definition}[Exogenous shock model]
A probability measure $\mu \in M_+^1(\IR^d)$ is said to \emph{define an exogenous shock model} if on some probability space there exists a random vector $\bm{X}$ with stochastic representation (\ref{exshock_canonic}) such that $\bm{X} \sim \mu$.
\end{definition}

\subsection{Exchangeability and the extendibility problem}
We are interested in a solution of Problem \ref{motivatingproblemrefined} for the property (P) of ``defining an exogenous shock model''. By Lemma \ref{lemma_exchangeable} exchangeability is a necessary requirement on $\bm{X}$, and we observe immediately from (\ref{exshock_canonic}) that this implies that the distribution function of $E_I$ is allowed to depend on the subset $I$ only through its cardinality $|I|$. Some simple algebraic manipulations, see the proof of Theorem \ref{thm_exShock} below, reveal that the survival function of $\bm{X}$ necessarily must be given as the product of its arguments after being ordered and idiosyncratically distorted. Already the characterization of the exchangeable subfamily in analytical terms is an interesting problem, the interested reader is referred to \cite{mai15} for its solution. 
\par
The conditionally iid subfamily $\mathfrak{M}_{\ast\ast}$ is also investigated in \cite{mai15}. One major finding is that when the increments of the factor process $Z$ in the canonical construction (\ref{canonical_construction_2}) are independent, then one ends up with an exogenous shock model. Recall that a c\`adl\`ag stochastic process $Z=\{Z_t\}_{t \geq 0}$ with independent increments is called \emph{additive}, see \cite{sato99} for a textbook treatment. For our purpose, it is sufficient to be aware that the probability law of a non-decreasing additive process $Z=\{Z_t\}_{t \geq 0}$ with $Z_0=0$ can be described uniquely in terms of a family $\{\Psi_t\}_{t \geq 0}$ of Bernstein functions defined by $\Psi_t(x):=-\log(\IE[\exp(-x\,Z_t)])$, $x \geq 0$, i.e.\ $\Psi_t$ equals the Laplace exponent of the infinitely divisible random variable $Z_t$. The independent increment property implies for $0 \leq s \leq t$ that $\Psi_t-\Psi_s$ is also a Bernstein function and equals the Laplace exponent of the infinitely divisible random variable $Z_t-Z_s$. The easiest example for a non-decreasing additive process is a L\'evy subordinator, in which case $\Psi_t=t\,\Psi_1$, i.e.\ the probability law is described completely in terms of just one Bernstein function $\Psi_1$ (due to the defining property that the increments are not only independent but also identically distributed). Two further compelling examples of (non-L\'evy) additive processes are presented in subsequent paragraphs.

\begin{theorem}[Additive subordinators and exogenous shock models]\label{thm_exShock}
Let $\mathfrak{M}$ denote the family of probability laws with the property (P) of ``defining an exogenous shock model''. A random vector $\bm{X}$ has law in $\mathfrak{M}$ and is exchangeable if and only if it admits a survival copula of the functional form
\begin{gather}
\hat{C}(u_1,\ldots,u_d) = u_{[1]}\,\prod_{k=2}^{d}g_k(u_{[k]}),\quad u_1,\ldots,u_d \in [0,1],
\label{exshock_copula}
\end{gather}
with certain functions $g_k:[0,1] \rightarrow [0,1]$. Furthermore, 
\begin{gather*}
\Big\{\gamma \in M_+^1(\mathfrak{H}_{+}) \,:\,\big\{1-e^{-Z_t} \big\}_{t \geq 0} \sim \gamma,\,\{Z_t\}_{t \geq 0}\mbox{ additive process}\Big\} = \Theta_d^{-1}\big(\mathfrak{M}_{\ast\ast}\big).
\end{gather*}
\end{theorem}
\begin{proof}
A proof for the inclusion ``$\supset$'' has been accomplished only recently and can be found in \cite{sloot20}. A proof sketch for the inclusion ``$\subset$'' works as follows, see \cite{mai15} for details. The survival function of the random vector $\bm{X}$ defined by (\ref{exshock_canonic}) can be written in terms of the one-dimensional survival functions of the $E_I$ as
\begin{gather*}
\IP(\bm{X}>\bm{x}) = \prod_{\emptyset \neq I}\IP(E_I>\max\{x_k\,:\,k \in I\}).
\end{gather*}
Exchangeability of $\bm{X}$ implies that the probability law of $E_I$ depends on $I$ only via its cardinality $|I| \in \{1,\ldots,d\}$. If we denote the survival function of $E_I$ with $|I|=m$ by $\bar{H}_m$, we observe that
\begin{align}
\IP(\bm{X}>\bm{x}) &= \prod_{m=1}^{d}\prod_{I\,:\,|I|=m} \bar{H}_m(\max\{x_k\,:\,k \in I\}) =  \prod_{m=1}^{d}\prod_{k=1}^{d-m+1}\big(\bar{H}_m(x_{[d-k+1]})\big)^{\binom{d-k}{m-1}} \nonumber\\
& = \prod_{k=1}^{d}\prod_{m=1}^{d-k+1}\big(\bar{H}_m(x_{[d-k+1]})\big)^{\binom{d-k}{m-1}}.\label{shortref_exshockproof}
\end{align}
Noting for $\bm{x}=(x,0,\ldots,0)$ that $x_{[d]}=x$ and $x_{[1]}=\ldots=x_{[d-1]}=0$, we observe that the one-dimensional margins are
\begin{gather*}
\IP(X_k>x) = \prod_{m=1}^{d}\big(\bar{H}_m(x)\big)^{\binom{d-1}{m-1}}=:\bar{F}_1(x),\quad k=1,\ldots,d.
\end{gather*}
That (\ref{shortref_exshockproof}) can be written as $\hat{C}\big( \IP(X_1>x_1),\ldots,\IP(X_d>x_d)\big)$ with $\hat{C}$ as in (\ref{exshock_copula}) follows by a tedious yet straightforward computation with the $g_k$ defined as
\begin{gather*}
g_k:=\prod_{m=1}^{d-k+1}\big(\bar{H}_m \circ \bar{F}_1^{-1}\big)^{\binom{d-k}{m-1}},\quad k=2,\ldots,d,
\end{gather*}
where $\bar{F}_1^{-1}$ denotes the generalized inverse of the non-increasing function $\bar{F}_1$, which is defined analogous to the generalized inverse of a distribution function as
\begin{gather*}
\bar{F}_1^{-1}(x):=\inf\{t>0\,:\,\bar{F}(t) \leq x\}.
\end{gather*}
Now assume $Z$ is an additive process with associated family of Bernstein functions $\{\Psi_t\}_{t \geq 0}$. The survival copula of the random vector $\bm{X}$ of Equation (\ref{canonical_construction_2}) can be computed in closed form using the independent increment property of $Z$. It is easily shown to be of the structural form (\ref{exshock_copula}), when
\begin{gather*}
g_k(u) := \exp\big(-\Psi_{\bar{F}_1^{-1}(u)}(k)+\Psi_{\bar{F}_1^{-1}(u)}(k-1)\big),\quad k=2,\ldots,d,
\end{gather*}
with $\bar{F}_1(x) := \exp(-\Psi_x(1))$, $x \geq 0$.
\end{proof}

\begin{remark}[Related literature]
Interestingly, there exists quite some literature that proposes the use of the random distribution function $H=1-\exp(-Z)$ with $Z$ additive as a prior distribution when estimating the probability law of observed samples $X_1,\ldots,X_d$ - without noticing the relation to exogenous shock models. $H$ is sometimes called a \emph{neutral-to-the-right prior} in these references. The use of additive $Z$ in this particular application is mainly explained by analytical convenience, because it implies that prior and posterior distributions have a similar algebraic structure. To provide some examples, the interested reader is referred to the papers \cite{kalbfleisch78,hjort90,lijoi08} and the references mentioned therein. Furthermore, \cite{kingman67} calls the random probability measures defined via the random distribution function $H$ \emph{completely random measures}, and characterizes them by the property that the values they take on disjoint subsets are independent.
\end{remark}

There are some interesting subfamilies of exogenous shock models that are worth mentioning with respect to their conditionally iid substructure. The first of them is a well known friend from previous sections, re-visited once again in the following example.
\begin{example}[The Marshall--Olkin law revisited]\label{ex_MOasexshock}
If all random variables $E_I$ in the exogenous shock model construction (\ref{exshock_canonic}) are exponentially distributed, we are in the special situation of Example \ref{ex_lom}($\mathcal{E}$). Indeed, it has already been shown in the original reference \cite{marshall67} that every Marshall--Olkin distribution can be constructed like this. Hence, we already know from Theorem \ref{thm_ciidlom}($\mathcal{E}$) that an exogenous shock model with exponential arrival times is obtained via the canonical conditionally iid model (\ref{canonical_construction_2}) if the associated stochastic process $H=\{H_t\}_{t \geq 0}\sim \gamma \in M_+^1(\mathfrak{H}_+)$ is such that $Z_t:=-\log(1-H_t)$, $t \geq 0$, defines a L\'evy subordinator, which is a special additive subordinator.
\end{example}

\begin{example}[A simple global shock model]
A special case of copulas of the form (\ref{exshock_copula}) is considered in \cite{durante07}, namely $g_2=\ldots=g_d$, which we briefly put in context with the additive process construction. To this end, let $g_2$ be a strictly increasing and continuous distribution function of some random variable taking values in $[0,1]$, assuming $x \mapsto g_2(x)/x$ is non-increasing on $(0,1]$. The function $F_M(x):= x/g_2(x)$ then is a distribution function on $[0,1]$, and we let $M$ be a random variable with this distribution function. Independently, let $W_1,W_2,\ldots$ be an iid sequence drawn from $g_2$. We consider the infinite exchangeable sequence $\{X_k\}_{k \in \IN}$ with $X_k:=\min\{E_{\{k\}},E_{\{1,2,\ldots\}}\}$, $k \in \IN$, where
\begin{gather*}
E_{\{k\}}:=-\log\big( F_M(W_k)\big),\, k\in \IN,\quad E_{\{1,2,\ldots\}}:=-\log\big( F_M(M)\big).
\end{gather*}
By definition, each finite $d$-margin has an exogenous shock model representation, and the survival copula $\hat{C}$ is easily seen to be of the form (\ref{exshock_copula}) with $g_2=\ldots=g_d$, for arbitrary $d \geq 2$. The conditional distribution function is static in the sense of Section \ref{sec:static}, and given by $H=1-\exp(-Z)$ with
\begin{gather*}
Z_t:= \begin{cases}
-\log\Big( g_2\Big\{ F_M^{-1}\big(e^{-t}\big)\Big\}\Big) & \mbox{, if }t < E_{\{1,2,\ldots\}}\\
\infty & \mbox{, if }t \geq E_{\{1,2,\ldots\}}\\
\end{cases}, \quad t \geq 0.
\end{gather*}
The random variable $E_{\{1,2,\ldots\}}$ is unit exponential, and $Z$ is additive with associated family of Bernstein functions $\Psi_t(x)=-\log(\IE[\exp(-x\,Z_t)])$ given by
\begin{gather*}
\Psi_t(x) = t\,1_{\{x>0\}}+x\,\Big( -\log\Big\{ g_2\big\{ F_M^{-1}\big(e^{-t} \big)\big\}\Big\}\Big),\quad x,t \geq 0.
\end{gather*}
Notice that for each fixed $t>0$ this corresponds to an infinitely divisible distribution of $Z_t$ that is concentrated on the set 
\begin{gather*}
\Big\{-\log\Big( g_2\big\{ F_M^{-1}\big(e^{-t} \big)\big\}\Big) ,\,\infty\Big\}.
\end{gather*}
The case $g_2(x)=x^{\alpha}$ with $\alpha \in [0,1]$ implies that $Z$ is a killed L\'evy subordinator that grows linearly before it jumps to infinity, $\bm{X}$ has a Marshall-Olkin law, and the $E_I$ are exponential. In the general case, $Z$ needs not grow linearly before it gets killed.
\end{example}

Two further examples are studied in greater detail in the following two paragraphs, since they give rise to nice characterization results.

\subsection{The Dirichlet prior and radial symmetry}
In the two landmark papers \cite{ferguson73,ferguson74}, T.S. Ferguson introduces the so-called Dirichlet prior and shows that it can be constructed by means of an additive process. More clearly, let $c>0$ be a model parameter and let $G \in \mathfrak{H}_+$, continuous and strictly increasing. Consider a non-decreasing additive process $Z=\{Z_t\}_{t \in [G^{-1}(0),G^{-1}(1)]}$ whose probability law is determined by a family of Bernstein functions $\{\Psi_t\}_{t \in (G^{-1}(0),G^{-1}(1))}$, which are given by
\begin{gather*}
\Psi_t(x) = \int_{0}^{\infty}\big( 1-e^{-x\,u}\big)\,\frac{e^{-u\,c\,(1-G(t))}-e^{-u\,c}}{u\,(1-e^{-u})}\,\mathrm{d}u,\quad x\geq 0,\,G^{-1}(0)<t<G^{-1}(1).
\end{gather*} 
The random distribution function $H=\{H_t\}_{t \in [G^{-1}(0),G^{-1}(1)]}$ defined by $H_t:=1-\exp(-Z_t)$ satisfies the following property: For arbitrary $G^{-1}(0)<t_1<\ldots<t_d<G^{-1}(1)$ the random vector
\begin{gather*}
\big(H_{t_1},H_{t_2}-H_{t_1},\ldots,H_{t_d}-H_{t_{d-1}},1-H_{t_d} \big)
\end{gather*}
has a Dirichlet distribution\footnote{Recall from Remark \ref{rmk_Liouville} that $\bm{S}=(S_1,\ldots,S_d)$ has a Dirichlet distribution with parameters $\bm{\alpha}=(\alpha_1,\ldots,\alpha_d)$ if $\bm{S}\stackrel{d}{=}\bm{G}/\norm{\bm{G}}_1$ for a vector $\bm{G}$ of independent unit-scale Gamma-distributed random variables.} with parameters
\begin{gather*}
c\,\big(G({t_1}),G({t_2})-G({t_1}),\ldots,G({t_d})-G({t_{d-1}}),1-G({t_d}) \big),
\end{gather*}
and $H$ is called \emph{Dirichlet prior} with parameters $(c,G)$, denoted $DP(c,G)$ in the sequel. The probability distribution of $(X_1,\ldots,X_d)$ in (\ref{canonical_construction_2}), when $H=DP(c,G)$ for some $G$ with support $[G^{-1}(0),G^{-1}(1)]:=[0,\infty]$, is given by
\begin{align}
\IP(X_1>x_1,\ldots,X_d>x_d) &= \hat{C}_c\big( 1-G(x_1),\ldots,1-G(x_d)\big), \mbox{ with}\nonumber\\
\hat{C}_c(u_1,\ldots,u_d)&=u_{[1]}\,\prod_{k=2}^{d}\frac{c\,u_{[k]}+k-1}{c+k-1}.\label{Dirichlet_copula}
\end{align}
It is insightful to remark that for $c \searrow 0$ the copula $\hat{C}_c$ converges to the so-called \emph{upper-Fr\'echet Hoeffding copula} $\hat{C}_0(\bm{u})=u_{[1]}$, and for $c \nearrow \infty$ to the copula $\hat{C}_{\infty}(\bm{u})=\prod_{k =1}^{d}u_k$ associated with independence. The intuition of the Dirichlet prior model is that all components of $\bm{X}$ have distribution function $G$, but one is uncertain whether $G$ is really the correct distribution function. So the parameter $c$ models an uncertainty about $G$ in the sense that the process $H$ must be viewed as a ``distortion'' of $G$. For $c \nearrow  \infty$ we obtain $H=G$, while for $c \searrow 0$ the process $H$ is maximally chaotic (in some sense) and does not  resemble $G$ at all.
\par
Interestingly, if the probability law $\mathrm{d}G$ is symmetric about its median $\mu:=G^{-1}(0.5)$, then the random vector $(X_1,\ldots,X_d)$ is radially symmetric, which can be verified using Lemma \ref{lemma_rs}. One can furthermore show that there exists no other conditionally iid exogenous shock model satisfying this property, see the following lemma. To this end, recall that a copula $C$ is called \emph{radially symmetric} if $C=\hat{C}$, i.e.\ it equals its own survival copula, which means that $\bm{U}=(U_1,\ldots,U_d) \stackrel{d}{=}(1-U_1,\ldots,1-U_d)$ for $\bm{U} \sim C$.

\begin{lemma}[Radial symmetry in exchangeable exogenous shock models]\label{lemma_DPrs}
A copula of the structural form (\ref{exshock_copula}) is radially symmetric if and only if the functions $g_k$ are linear, $k=2,\ldots,d$, which is the case if and only if there is a $c \in [0,\infty]$ such that the copula takes the form (\ref{Dirichlet_copula}).
\end{lemma}
\begin{proof}
This is \cite[Theorem 3.5]{mai16}. In order to prove necessity, the principle of inclusion and exclusion can be used to express the survival copula of $\hat{C}$ as an alternating sum of lower-dimensional margins of $\hat{C}$. By radial symmetry, this expression equals $\hat{C}$, and on both sides of the equation one may now take the derivatives with respect to all $d$ arguments. A lengthy but tedious computation then shows that the $g_k$ must all be linear, which implies the claim. Sufficiency is proved using the Dirichlet prior construction. The defining properties of the Dirichlet prior imply that the assumptions of Lemma \ref{lemma_rs} are satisfied, which implies the claim.
\end{proof}

\subsection{The Sato-frailty model and self-decomposability}
A real-valued random variable $X$ is called \emph{self-decomposable} if for arbitrary $c \in (0,1)$ there exists an independent random variable $Y$ such that $X \stackrel{d}{=}c\,X+Y$. It can be shown that a self-decomposable $X$ is infinitely divisible, so self-decomposable laws are special cases of infinitely divisible laws. In particular, if $X$ takes values in $(0,\infty)$ and is infinitely divisible with Laplace exponent given by the Bernstein function $\Psi$, then $X$ is self-decomposable if and only if the function $x \mapsto x\,\Psi^{(1)}(x)$ is again a Bernstein function, see \cite[Theorem 2.6, p.\ 227]{steutel03}. Now let $\Psi$ be the Bernstein function associated with a self-decomposable law on $(0,\infty)$, and consider a family of Bernstein functions defined by $\Psi_t(x):=\Psi(x\,t)$, $t \geq 0$. One can show that there exists an additive subordinator $Z=\{Z_t\}_{t \geq 0}$ which is uniquely determined in law by $\{\Psi_t\}_{t \geq 0}$ via $\Psi_t(x)=-\log(\IE[\exp(-x\,Z_t)])$, $x,t \geq 0$, called \emph{Sato subordinator}. If we use this process in (\ref{canonical_construction_2}), the conditionally iid random vector $\bm{X}$ obtained by this construction has survival function given by
\begin{gather}
\IP(\bm{X}>\bm{x}) = \exp\Big\{{-\sum_{k=1}^{d}\Psi\big((d-k+1)\,x_{[k]}\big)-\Psi\big((d-k)\,x_{[k]}\big)}\Big\},\quad \bm{x} \in [0,\infty)^d.
\label{SD_exshock}
\end{gather}

The following lemma characterizes self-decomposability analytically in terms of multivariate probability laws given by (\ref{SD_exshock}). 

\begin{lemma}[Characterization of self-decomposable Bernstein functions]\label{lemma_BFchar}
Let $\Psi:[0,\infty) \rightarrow [0,\infty)$ be some function.
The $d$-variate function (\ref{SD_exshock}) defines a proper survival function on $[0,\infty)^d$ for all $d \geq 2$ if and only if $\Psi$ equals the Bernstein function of a self-decomposable probability law on $(0,\infty)$. 
\end{lemma}
\begin{proof}
Sufficiency is an instance of the general Theorem \ref{thm_exShock}, as demonstrated above. Necessity, i.e.\ that self-decomposability can actually be characterized in terms of the multivariate survival functions (\ref{SD_exshock}), is shown in \cite{mai17} and relies on some purely analytical, technical computations.
\end{proof}

\begin{example}[A one-parametric, multivariate Pareto distribution]
Let $\Psi(x)=\alpha\,\log(1+x)$ be the Bernstein function associated with a Gamma distribution\footnote{This is precisely the Gamma distribution with density (\ref{gamma_density}) for $\alpha=\alpha_k$.} with parameter $\alpha>0$. The Gamma distribution is self-decomposable and the survival function (\ref{SD_exshock}) takes the explicit, one-parametric form
\begin{gather*}
\IP(\bm{X}>\bm{x}) = \Big(\prod_{k=1}^{d} \frac{(d-k)\,x_{[k]}+1}{(d-k+1)\,x_{[k]}+1}\Big)^{\alpha}.
\end{gather*}
The one-dimensional marginal survival functions are given by $\bar{F}_1(x)=(1+x)^{-\alpha}$. Notice that this equals the survival function of $Y-1$, when $Y$ has a Pareto distribution with scale parameter (aka left-end point of support) equal to one and tail index $\alpha$. Thus, the random vector $\bm{X}+\bm{1}:=(X_1+1,\ldots,X_d+1)$ might be viewed as a multivariate extension of the Pareto distribution with scale parameter equal to one and tail index $\alpha$.
\end{example}

\newpage

\section{Related open problems}\label{sec:open}
\subsection{Extendibility-problem for further families}\label{open_iidunif}
The present article surveys solutions to Problems \ref{motivatingproblem} and \ref{motivatingproblemrefined} for several families $\mathfrak{M}$ of interest. One goal of the survey is to encourage others to solve the problem also for other families. We provide examples that we find compelling:
\begin{itemize}
\item[(i)] The family of min-stable laws in Section \ref{sec:evd} can be generalized to min-infinitely divisible laws. Generalizing (\ref{analytical_minstable}), a multivariate survival function $\bar{F}$ is called min-infinitely divisible if for each $t>0$ there is a survival function $\bar{F}_t$ such that $\bar{F}(\bm{x})^t=\bar{F}_t(\bm{x})$. Like min-stability is analogous to max-stability, the concept of min-infinite divisibility is equivalent to the concept of max-infinite divisibility, on which \cite{resnick87} provides a textbook treatment. It is pretty obvious that non-decreasing infinitely divisible processes occupy a commanding role with regards to the conditionally iid subfamily, but to work out a convenient analytical treatment of these in relation with the associated min-infinitely divisible laws appears to be a promising direction for further research. Notice that the family of reciprocal Archimedean copulas, introduced in \cite{genest18}, is one particular special case of max-infinitely divisible distribution functions, and in this special case the conditionally iid subfamily is determined similarly as in the case of Archimedean copulas, see \cite[Section 7]{genest18}. This might serve as a good motivating example for the aforementioned generalization.
\item[(ii)] Theorem \ref{thm_gnedin} studies $d$-variate densities of the form $g_d(x_{[d]})$, and \cite{gnedin95} also considers a generalization to densities of the form $g(x_{[1]},x_{[d]})$, depending on $x_{[1]}$ and $x_{[d]}$. From a purely algebraic viewpoint it is tempting to investigate whether exchangeable densities of the structural form $\prod_{k=1}^{d}g_k(x_{[k]})$ allow for a nice theory as well. When are these conditionally iid? This generalization of the $\ell_{\infty}$-norm symmetric case is motivated by a relation to non-homogeneous pure birth processes, as already explained in Remark \ref{rmk_purebirth}. Such processes are of interest in reliability theory, as explained in \cite{shaked02}.
\item[(iii)] On page \pageref{brigoref} it was mentioned that the Marshall-Olkin distribution is characterized by the property that for all subsets of components the respective ``survival indicator process'' is a continuous-time Markov chain. This property may naturally be weakened to the situation when only the survival indicator process $Z_t:=(1_{\{X_1>t\}},\ldots,1_{\{X_d>t\}})$ of all components is a continuous-time Markov chain. On the level of multivariate distributions, one generalizes the Marshall-Olkin distribution to a more general family of multivariate laws that has been shown to be interesting in mathematical finance in \cite{herbertsson08}. Furthermore, it is a subfamily of the even larger family of so-called multivariate phase-type distributions, see \cite{assaf84}. Which members of theses families of distributions are conditionally iid? Presumably, this research direction requires to generalize the L\'evy subordinator in the Marshall--Olkin case to more general non-decreasing Markov processes.
\end{itemize}
\subsection{Testing for conditional independence}
If a specific $d$-variate law in some family $\mathfrak{M}$ is given, do we have a practically useful, analytical criterion to decide whether or not this law is in $\mathfrak{M}_{\ast}$, resp.\ $\mathfrak{M}_{\ast\ast}$, or not? According to Theorem \ref{thm_probgen}, in general this requires to check whether a supremum over bounded measurable functions is bounded from above, which in practice is rather inconvenient - at least on first glimpse. For certain families, however, there is hope to find more useful criteria. For instance, for Marshall-Olkin distributions the link to the truncated moment problem in Remark \ref{remk_truncMOM} is helpful in this regard, like it is for binary sequences. For Archimedean copulas (resp.\ $\ell_1$-norm symmetric survival functions) this boils down to checking whether a $d$-monotone function is actually completely monotone, i.e.\ a Laplace transform. However, it is an open problem for the family of extreme-value copulas. Of course, Theorem \ref{thm_minstable} tells us which stable tail dependence functions correspond to conditionally iid laws. But given some specific stable tail dependence function, how can we tell effectively whether or not this given function has the desired form? Even for dimension $d=2$, in which case the problem is presumably easier due to the fact that the $2$-dimensional unit simplex is one-dimensional, this problem is non-trivial and open. Given we find such effective analytical criterion for some family $\mathfrak{M}$, is it even possible to build a useful statistical test based on it, i.e.\ can we test the hypothesis that the law is conditionally iid? 
\subsection{Combination of one-factor models to multi-factor models}\label{open_multifactor}
This is probably the most obvious application of the presented concepts. The idea works as follows. According to our notation, the dependence-inducing latent factor in a conditionally iid model is $H$. Depending on the stochastic properties of $H \sim \gamma \in M_+^1(\mathfrak{H})$, it may be possible to construct $H$ from a pair $(H^{(1)},H^{(2)}) \sim \gamma_1 \otimes \gamma_2 \in M_+^1(\mathfrak{H}) \times M_+^1(\mathfrak{H})$ of two independent processes of the same structural form, say $H=f(H^{(1)},H^{(2)})$. For example, if $H^{(1)}$ and $H^{(2)}$ are two strong IDT processes, see Section \ref{sec:evd}, then so is their sum $H=H^{(1)}+H^{(2)}$. In this situation, we may define \emph{dependent} processes $H^{(1,1)},\ldots,H^{(1,J)}$ from $J+1$ \emph{independent} processes $H^{(0)},\ldots,H^{(J)}$ as $H^{(1,j)}=f(H^{(0)},H^{(j)})$. The conditionally iid vectors $\bm{X}^{(1)},\ldots,\bm{X}^{(J)}$ defined via (\ref{canonical_construction}) from $H^{(1,1)},\ldots,H^{(1,J)}$ are then dependent, so that the combined vector $\bm{X}=(\bm{X}^{(1)},\ldots,\bm{X}^{(J)})$ has a hierarchical dependence structure. Such structures break out of the - sometimes undesired and limited - exchangeable cosmos and have the appealing property that the lowest-level groups are conditionally iid, so the whole structure can be sized up, i.e.\ is dimension-free to some degree. Of particular interest is the situation when the random vector $(X^{(1)}_1,\ldots,X^{(J)}_1)$ composed of one component from each of the $J$ different groups is conditionally iid and its latent factor process equals $H^{(0)}$ in distribution. In this particular situation, an understanding of the whole dependence structure of the hierarchical model $\bm{X}$ is retrieved from an understanding of the conditionally iid sub-models based on the $H^{(j)}$. In other words, the conditionally iid model can be \emph{nested} to construct highly tractable, non-exchangeable, multi-factor dependence models from simple building blocks. For instance, hierarchical elliptical laws, Archimedean copulas\footnote{See also the many references in Remark \ref{rmk_AC}.}, and min-stable laws can be constructed based on the presented one-factor building blocks, see \cite{maischerer12} for an overview. For these and other families, the design, estimation, and efficient simulation of such hierarchical structures is an active area of research or even an unsolved problem.
\subsection{Parameter estimation with uncertainty} \label{open_Bayes}
The classical statistical parameter estimation problem is to estimate the (true) parameter $m$ of a one-parametric distribution function $F_m$ from iid observations $X_1,\ldots,X_d \sim F_{m}$. A parameter estimate is then a  function $\hat{m}=\hat{m}(X_1,\ldots,X_d)$ of the observations into the set of admissible parameters. This classical problem relies on the hypothesis that there is a ``true'' parameter $m$, from which the observations are drawn. But what if we are uncertain whether or not the observations are actually drawn from some $F_{m}$? The Dirichlet prior has been introduced in \cite{ferguson73,ferguson74} with the motivation to model uncertainty about the hypothesis that observations are drawn from some $F_{m}$. Instead, it is assumed that they are drawn from $DP(c,F_{m})$ with an uncertainty parameter $c>0$. On a high level, this amounts to observing one sample $\bm{X}=(X_1,\ldots,X_d)$, with large $d$, from a parametric conditionally iid model. Optimal estimates for $m$ based on the observations can then be derived due to the convenient Dirichlet prior setting, see \cite{ferguson73,ferguson74} for details. But this question can clearly also be posed for other conditionally iid models. Let us provide a second motivation that appears to be natural: let $X_1,\ldots,X_d$ be observed time points of company bankruptcy filings within the last $10$ years. An iid assumption for $X_1,\ldots,X_d$ is well known to be inappropriate. Instead, a popular model for such time points is a Marshall-Olkin distribution, see \cite{giesecke03}. If we assume in addition - now for mathematical convenience - that $\bm{X}=(X_1,\ldots,X_d)$ is conditionally iid, we know from Theorem \ref{thm_ciidlom}($\mathcal{E}$) and Lemma \ref{lemma_condGC} that the empirical distribution function of $X_1,\ldots,X_d$ is approximately equal to $1-\exp(-Z)$ for a L\'evy subordinator $Z$. Depending on a specific parametric model for $Z$, it is well possible that we can estimate the parameters based on the observed empirical distribution function. For example, if $Z$ is a compound Poisson process with constant jump size $m$, then huge (small) jumps in the empirical distribution function apparently indicate a large (small) value of $m$. Such parameter estimation problems based on one (large) sample $\bm{X}=(X_1,\ldots,X_d)$ from a conditionally iid model appear to be very model-specific and thus possibly interesting, and the two motivating examples above indicate that one might find natural motivations for these.
\subsection{Quantification of diversity of possible extensions}\label{subsec_effectext}
All of the presented theorems solve Problem \ref{motivatingproblemrefined}, but only in some cases\footnote{To wit, Example \ref{ex_Normal}, Schoenberg's Theorem \ref{schoenberg_thm} and Theorem \ref{thm_gnedin}.} the solution set $\mathfrak{M}_{\ast\ast}$ is shown to coincide with the in general larger solution set $\mathfrak{M}_{\ast}$ in Problem \ref{motivatingproblem}. Can one show that $\mathfrak{M}_{\ast}=\mathfrak{M}_{\ast\ast}$ in the other presented solutions of Problem \ref{motivatingproblemrefined}? To provide one concrete example, from Theorem \ref{thm_ciidlom}($\mathcal{G}$) we know that $(b_0,b_1,b_2) \in \mathcal{M}_2$ determines a three-dimensional, exchangeable wide-sense geometric law. However, this exchangeable probability distribution is only in $\mathfrak{M}_{\ast\ast}$ if there exist $b_3,b_4,\ldots$ such that $\{b_k\}_{k \in \IN_0} \in \mathcal{M}_{\infty}$. Could it be that the last extension property does fail, but the three-dimensional, exchangeable wide-sense geometric law associated with $(b_0,b_1,b_2)$ is still conditionally iid? If so, then necessarily there is some $n>3$ and an infinite exchangeable sequence $\{X_k\}_{k \in \IN}$ such that $(X_1,X_2,X_3)$ has the given wide-sense geometric law but $(X_1,\ldots,X_n)$ is not wide-sense geometric.
\par
A related question concerns only elements in $\mathfrak{M}_{\ast\ast}$. There might be two infinite exchangeable sequences $\{X^{(1)}_k\}_{k \in \IN}$ and $\{X^{(2)}_k\}_{k \in \IN}$ with  $\{X^{(1)}_k\}_{k \in \IN} \stackrel{d}{\neq}\{X^{(2)}_k\}_{k \in \IN}$ but $(X^{(1)}_1,\ldots,X^{(1)}_d) \stackrel{d}{=}(X^{(2)}_1,\ldots,X^{(2)}_d)$ for some $d \in \IN$. To provide an example, related to Theorems \ref{thm_definetti01} and \ref{thm_ciidlom}, the vector $(1,b_1)$ with $b_1 \in [0,1]$ can always be extended to a sequence $\{b_k\}_{k \in \IN}$ that is completely monotone, for example set $b_k=b_1^k$. In case of Theorem \ref{thm_ciidlom}($\mathcal{G}$), all the different possible extensions imply different exchangeable sequences $\{X_k\}_{k \in \IN}$ such that $2$-margins follow the associated wide-sense geometric law with parameters $(1,b_1)$. But all these extensions in Theorem \ref{thm_ciidlom}($\mathcal{G}$) have in common that arbitrary $d$-margins are always wide-sense geometric. Can one quantify \emph{how} different such extensions are allowed to be? A similar question is: Is the ``$\subset$'' in (\ref{ext_ARCHIMAX}) actually a ``$=$''?  Notice that the proof ideas in \cite{diaconis87,rachev91}, who study such issues in the case of some static laws, might help to approach such questions. 
\subsection{Characterization of stochastic objects via multivariate probability laws}
As a general rule, for an infinite exchangeable sequence $\{X_k\}_{k \in \IN}$ defined via (\ref{canonical_construction}) the probability law of the random distribution function $H$ is uniquely determined by its mixed moments
\begin{gather*}
\IE[H_{t_1}\,\cdots\,H_{t_d}]=\IP(X_1 \leq t_1,\ldots,X_d \leq t_d),\quad d \in \IN,\,t_1,\ldots,t_d \in \IR.
\end{gather*} 
This often implies interesting analytical characterizations of the stochastic object $H$ in terms of the multivariate distribution functions $\bm{t} \mapsto \IP(X_1 \leq t_1,\ldots,X_d \leq t_d)$. In particular, if $H$ is of the form $H=1-\exp(-Z)$ like in (\ref{canonical_construction_2}), then the mixed moments above become
\begin{gather*}
\IE\Big[e^{-\sum_{k=1}^{d}Z_{t_k}}\Big]=\IP(X_1 > t_1,\ldots,X_d > t_d),\quad d \in \IN,\,t_1,\ldots,t_d \geq 0,
\end{gather*} 
that is the survival functions $\bm{t} \mapsto \IP(X_1 > t_1,\ldots,X_d > t_d)$ stand in one-to-one relation with the Laplace transforms of finite-dimensional margins of the non-decreasing process $Z$. This general relationship explains the close connection between conditionally iid probability laws and moment problems/ Laplace transforms encountered several times in this survey. For instance, Theorem \ref{thm_AC} shows that $\varphi$ is a Laplace transform if and only if $\bm{x} \mapsto \varphi(\norm{\bm{x}}_1)$ is a survival function for all $d \geq 1$, or Theorem \ref{thm_ciidlom} characterizes L\'evy subordinators in terms of multivariate survival functions, or Lemma \ref{lemma_BFchar} characterizes self-decomposable Bernstein functions via multivariate survival functions. Can further characterizations be found? Is there a compelling application for such characterizations in terms of multivariate survival functions?


\newpage

\begin{thebibliography}{}
\bibitem{aldous85} \textsc{Aldous, D.J.} (1985). \textit{Exchangeability and related topics}. 
Springer, \'Ecole d'\'Et\'e de Probabilit\'es de Saint-Flour XIII-1983, Lecture Notes in Mathematics {1117}, 1--198.
\bibitem{aldous10} \textsc{Aldous, D.J.} (1985). \textit{More uses of exchangeability: representations of complex random structures}. 
{in Probability and Methematical Genetics - papers in honour of Sir John Kingman, Cambridge University Press} 35--63.
\bibitem{alfsen71} \textsc{Alfsen, E.M.} (1971). \textit{Compact convex sets and boundary integrals}. 
Springer, Berlin.
\bibitem{arnold75} \textsc{Arnold, B.C.} (1975). 
A characterization of the exponential distribution by multivariate geometric compounding. \textit{Sankhy$\bar{a}$: The Indian Journal of Statistics} \textbf{37:1} 164--173.
\bibitem{assaf84} \textsc{Assaf, D.} and \textsc{Langberg, N.A.} and \textsc{Savits, T.H.} and \textsc{Shaked, M.} (1984). 
Multivariate phase-type distributions. \textit{Operations Research} \textbf{32:3} 688--702.
\bibitem{barlow75} \textsc{Barlow, R.E.} and \textsc{Proschan, F.} (1975). \textit{Statistical theory of reliability and life testing}. 
Rinehart and Winston, New York.
\bibitem{beirlant04} \textsc{Beirlant, J.} and \textsc{Goegebeur, Y.} and \textsc{Teugels, J.} and \textsc{Segers, J.} (2004). \textit{Statistics of Extremes: Theory and Applications}. 
John Wiley \& Sons, Chichester.
\bibitem{berg84} \textsc{Berg, C.} and \textsc{Christensen, J.P.R.} and \textsc{Ressel, P.} (1984). \textit{Harmonic analysis on semigroups}. 
Springer, Berlin.
\bibitem{bernhart15} \textsc{Bernhart, G.} and \textsc{Mai, J.-F.} and \textsc{Scherer, M.} (2015). On the construction of low-parametric families of min-stable multivariate exponential distributions in large dimensions. \textit{Dependence Modeling} \textbf{3} 29--46.
\bibitem{bernstein29} \textsc{Bernstein, S.} (1929). Sur les fonctions absolument monotones. \textit{Acta Mathematica} \textbf{52} 1--66.
\bibitem{billingsley95} \textsc{Billingsley, P.} (1995). \textit{Probability and Measure}. 
Wiley Series in Probability and Statistics, Wiley, New York.
\bibitem{brigo16} \textsc{Brigo, D.} and \textsc{Mai, J.-F.} and \textsc{Scherer, M.} (2016). Markov multi-variate survival indicators for default simulation as a new characterization of the Marshall--Olkin law \textit{Statistics and Probability Letters} \textbf{114} 60--66.
\bibitem{caperaa00} \textsc{Cap\'era\`a, P.} and \textsc{Foug\`eres, A.-L.} and \textsc{Genest, C.} (2000). Bivariate distributions with given extreme value attractor, \textit{Journal of Multivariate Analysis} \textbf{72} 30--49.
\bibitem{charpentier14} \textsc{Charpentier, A.} and \textsc{Foug\`eres, A.-L.} and \textsc{Genest, C.} and \textsc{Ne\v{s}lehov\'{a}, J.G.} (2014). Multivariate Archimax copulas, \textit{Journal of Multivariate Analysis} \textbf{126} 118--136.
\bibitem{cossette17} \textsc{Cossette, H.} and \textsc{Gadoury, S.-P.} and \textsc{Marceauand, E.} and \textsc{Mtalai, I.} (2017). Hierarchical Archimedean copulas through multivariate compound distributions. \textit{Insurance: Mathematics and Economics} \textbf{76} 1--13.
\bibitem{daboni82} \textsc{Daboni, L.} (1982). \textit{Exchangeability and completely monotone functions}. 
{in Exchangeability in Probability and Statistics, edited by G.\ Koch and F.\ Spizzichino, North-Holland Publishing Company} 39--45.
\bibitem{definetti31} \textsc{de Finetti, B.} (1931). Funzione caratteristica di un fenomeno aleatorio. \textit{Atti della R.\ Academia Nazionale dei Lincei, Serie 6. Memorie, Classe di Scienze Fisiche, Mathematica e Naturale} \textbf{4} 251--299.
\bibitem{definetti37} \textsc{de Finetti, B.} (1937). La pr\'evision: ses lois logiques, ses sources subjectives. \textit{Annales de l'Institut Henri Poincar\'e} \textbf{7} 1--68.
\bibitem{diaconis87} \textsc{Diaconis, P.} and \textsc{Freedman, D.} (1987). A dozen de Finetti-style results in search of a theory. \textit{Annales de l'Institute Henri Poincar\'e} \textbf{23} 397--423.
\bibitem{dickinson14} \textsc{Dickinson, P.J.C.} and \textsc{Gijben, L.} (2014). On the computational complexity of membership problems for the completely positive cone and its dual. \textit{Computational Optimization and Applications} \textbf{57:2} 403--415.
\bibitem{durante07} \textsc{Durante, F.} and \textsc{Quesada-Molina, J.J.} and \textsc{\'Ubeda-Flores, M.} (2007). \textit{A method for constructing multivariate copulas}. 
{in New dimensions in fuzzy logic and related technologies - Proceedings of the 5th EUSFLAT Conference, volume 1, edited by M.\ \v{S}t\v{e}pni\v{c}ka et al.} 191--195.
\bibitem{durrett10} \textsc{Durrett, R.} (2010). \textit{Probability: theory and examples, 4th edition}. 
{Cambridge University Press}, Cambridge.
\bibitem{dykstra73} \textsc{Dykstra, R.L.} and \textsc{Hewett, J.E.} and \textsc{Thompson, Jr., W.A.} (1973). Events which are almost independent. \textit{Annals of Statistics} \textbf{1:4} 674--681.
\bibitem{embrechts13} \textsc{Embrechts, P.} and \textsc{Hofert, M.} (2013). A note on generalized inverses. \textit{Mathematical methods of Operations Research} \textbf{77} 423--432.
\bibitem{esary74} \textsc{Esary, J.D.} and \textsc{Marshall, A.W.} (1974). Multivariate distributions with exponential minimums. \textit{Annals of Statistics} \textbf{2} 84--98.
\bibitem{ouknine08} \textsc{Es-Sebaiy, K.} and \textsc{Ouknine, Y.} (2008). How rich is the class of processes which are infinitely divisible with respect to time. \textit{Statistics and Probability Letters} \textbf{78} 537--547.
\bibitem{giesecke03} \textsc{Giesecke, K.} (2003). A simple exponential model for dependent defaults, \textit{Journal of Fixed Income} \textbf{13:3} 74--83.
\bibitem{gupta04} \textsc{Gupta, A.K.} and \textsc{Nadarajah, S.} (2004). \textit{Handbook of beta distributions and its applications}. 
{Marcel Dekker}, New York.
\bibitem{hewitt55} \textsc{Hewitt, E.} and \textsc{Savage, l.J.} (1955). Symmetric measures on Cartesian products. \textit{Transactions of the American Mathematical Society} \textbf{80}  470--501.
\bibitem{fang90} \textsc{Fang, K.-T.} and \textsc{Kotz, S.} and \textsc{Ng, K.-W.} (1990). \textit{Symmetric multivariate and related distributions}. 
{Chapman and Hall, London}.
\bibitem{feller66} \textsc{Feller, W.} (1966). \textit{An introduction to probability theory and its applications, volume II, 2nd edition}. 
{John Wiley and Sons, Inc., Hoboken.
\bibitem{ferguson73} \textsc{Ferguson, T.S.} (1973). A Bayesian analysis of some nonparametric problems. \textit{Annals of Statistics} \textbf1} 209--230.
\bibitem{ferguson74} \textsc{Ferguson, T.S.} (1974). Prior distributions on spaces of probability measures. \textit{Annals of Statistics} \textbf{2} 615--629.
\bibitem{frank79} \textsc{Frank, M.J.} (1979). On the simultaneous associativity of $F(x,y)$ and $x+y-F(x,y)$. \textit{Aequationes Mathematicae} \textbf{19} 194--226.
\bibitem{galambos75} \textsc{Galambos, J.} (1975). Order statistics of samples from multivariate distributions. \textit{Journal of the American Statistical Association} \textbf{70} 674--680.
\bibitem{genest17} \textsc{Genest, C.} and \textsc{Ne\v{s}lehov\'{a}, J.G.} (2017). \textit{When Gumbel met Galambos}. 
In Copulas and Dependence Models With Applications: Contributions in Honor of Roger B. Nelsen (M.\ \'Ubeda Flores, E. de Amo Artero, F. Durante, J. Fern\'andez S\'anchez, Eds.), Springer, 83--93.
\bibitem{genest18} \textsc{Genest, C.} and \textsc{Ne\v{s}lehov\'{a}, J.G.} and \textsc{Rivest, L.-P.} (2018). The class of multivariate max-id copulas with $\ell_1$-norm symmetric exponent measure. \textit{Bernoulli} \textbf{24} 3751--3790.
\bibitem{genest89} \textsc{Genest, C.} and \textsc{Rivest, L.-P.} (1989). Characterization of Gumbel's family of extreme value distributions. \textit{Statistics and Probability Letters} \textbf{8} 207--211.
\bibitem{gnedin95} \textsc{Gnedin, A.V.} (1995). On a class of exchangeable sequences. \textit{Statistics and Probability Letters} \textbf{25} 351--355.
\bibitem{gumbel60} \textsc{Gumbel, E.J.} (1960). Bivariate exponential distributions. \textit{Journal of the American Statistical Association} \textbf{55} 698--707.
\bibitem{gumbel61} \textsc{Gumbel, E.J.} (1961). Bivariate logistic distributions. \textit{Journal of the American Statistical Association} \textbf{56} 335--349.
\bibitem{ouknine12} \textsc{Hakassou, A.} and \textsc{Ouknine, Y.} (2013). IDT processes and associated L\'evy processes with explicit constructions. \textit{Stochastics} \textbf{85:6} 1073--1111.
\bibitem{hausdorff21} \textsc{Hausdorff, F.} (1921). Summationsmethoden und Momentfolgen I. \textit{Mathematische Zeitschrift} \textbf{9:3-4} 74--109.
\bibitem{hausdorff23} \textsc{Hausdorff, F.} (1923). Momentenproblem f\"ur ein endliches Intervall. \textit{Mathematische Zeitschrift} \textbf{16} 220--248.
\bibitem{herbertsson08} \textsc{Herbertsson, A.} and \textsc{Rootz\'en, H.} (2008). Pricing $k$th-to-default swaps under default contagion: the matrix-analytic approach. \textit{Journal of Computational Finance} \textbf{12} 49--72. 
\bibitem{hering10} \textsc{Hering, C.} and \textsc{Hofert, M.} and \textsc{Mai, J.-F.} and \textsc{Scherer, M.} (2010). Constructing hierarchical Archimedean copulas with L\'evy subordinators. \textit{Journal of Multivariate Analysis} \textbf{101} 1428--1433.
\bibitem{hjort90} \textsc{Hjort, N.L.} (1990). Nonparametric Bayes estimators based on beta processes in models for life history data. \textit{Annals of Statistics} \textbf{18:3} 1259--1294.
\bibitem{hofertscherer10} \textsc{Hofert, M.} and \textsc{Scherer, M.} (2011). CDO pricing with nested Archimedean copulas. \textit{Quantitative Finance} \textbf{11} 775--787.
\bibitem{joe97} H.\ Joe (1997). \textit{Multivariate models and dependence concepts}. 
{Chapman \& Hall/CRC}, Boca Raton.
\bibitem{kalbfleisch78} \textsc{Kalbfleisch, J.D.} (1978). Non-parametric Bayesian analysis of survival time data. \textit{Journal of the Royal Statistical Society Series B} \textbf{40:2} 214--221.
\bibitem{kallenberg82} \textsc{Kallenberg, O.} (1982). \textit{A dynamical approach to exchangeability}. 
in Exchangeability in Probability and Statistics, edited by G.\ Koch and F.\ Spizzichino, North-Holland Publishing Company, 87--96.
\bibitem{karlin53} \textsc{Karlin, S.} and \textsc{Shapley, L.S.} (1953). Geometry of moment spaces, \textit{Memoirs of the American Mathematical Society} \textbf{12:93}.
\bibitem{kimberling74} \textsc{Kimberling, C.H.} (1974). A probabilistic interpretation of complete monotonicity. \textit{Aequationes Mathematicae} \textbf{10} 152--164.
\bibitem{kingman67} \textsc{Kingman, J.F.C.} (1967). Completely random measures. \textit{Pacific Journal of Mathematics} \textbf{21:1} 59--78.
\bibitem{kingman72} \textsc{Kingman, J.F.C.} (1972). On random sequences with spherical symmetry. \textit{Biometrika} \textbf{59} 492--494.
\bibitem{kingman78} \textsc{Kingman, J.F.C.} (1978). Uses of exchangeability. \textit{Annals of Probability} \textbf{6:2} 183--197.
\bibitem{konstantopoulos19} \textsc{Konstantopoulos, T.} and \textsc{Yuan, L.} (2019). On the extendibility of finitely exchangeable probability measures, \textit{Transactions of the American Mathematical Society} \textbf{371} 7067--7092.
\bibitem{kopp18} \textsc{Kopp, C.} and \textsc{Molchanov, I.} (2018). Series representations of time-stable stochastic processes, \textit{Probability and Mathematical Statistics} \textbf{38:2} 299--315.
\bibitem{liggett07} \textsc{Liggett, T.M.} and \textsc{Steiff, J.E.} and \textsc{T\'oth, B.} (2007). Statistical mechanical systems on complete graphs, infinite exchangeability, finite extensions and a discrete finite moment problem, \textit{Annals of Probability} \textbf{35:3} 867--914.
\bibitem{lijoi08} \textsc{Lijoi, A.} and \textsc{Pr\"unster, I.} and \textsc{Walker, S.G.} (2008). Posterior analysis for some classes of nonparametric models. \textit{Journal of Nonparametric Statistics} \textbf{20:5} 447--457.
\bibitem{lindskog03} \textsc{Lindskog, F.} and \textsc{McNeil, A.J.} (2003). Common Poisson shock models: applications to insurance and credit risk modelling. \textit{ASTIN Bulletin} \textbf{33:2} 209--238.
\bibitem{lukacs55} \textsc{Lukacs, E.} (1955). A characterization of the gamma distribution. \textit{Annals of Mathematical Statistics} \textbf{26} 319--324.
\bibitem{mai18} \textsc{Mai, J.-F.} (2018). Extreme-value copulas associated with the expected scaled maximum of independent random variables, \textit{Journal of Multivariate Analysis} \textbf{166} 50--61.
\bibitem{mai19} \textsc{Mai, J.-F.} (2019). Simulation of hierarchical Archimedean copulas beyond the completely monotone case, \textit{Dependence Modeling} \textbf{7} 202--214.
\bibitem{mai18b} \textsc{Mai, J.-F.} (2020). Canonical spectral representation for exchangeable max-stable sequences, \textit{Extremes} \textbf{23} 151--169.
\bibitem{mai15} \textsc{Mai, J.-F.} and \textsc{Schenk, S.} and \textsc{Scherer, M.} (2016). Exchangeable exogenous shock models. \textit{ Bernoulli} \textbf{22} 1278--1299.
\bibitem{mai16} \textsc{Mai, J.-F.} and \textsc{Schenk, S.} and \textsc{Scherer, M.} (2016). Analyzing model robustness via a distortion of the stochastic root: a Dirichlet prior approach, \textit{Statistics and Risk Modeling} \textbf{32} 177--195.
\bibitem{mai17} \textsc{Mai, J.-F.} and \textsc{Schenk, S.} and \textsc{Scherer, M.} (2017). Two novel characterizations of self-decomposability on the positive half-axis. \textit{Journal of Theoretical Probability} \textbf{30} 365--383.
\bibitem{maischerer09} \textsc{Mai, J.-F.} and \textsc{Scherer, M.} (2009). L\'evy-frailty copulas. \textit{Journal of Multivariate Analysis} \textbf{100} 1567--1585.
\bibitem{maischerer11} \textsc{Mai, J.-F.} and \textsc{Scherer, M.} (2011). Reparameterizing Marshall--Olkin copulas with applications to sampling. \textit{Journal of Statistical Computation and Simulation} \textbf{81} 59--78.
\bibitem{maischerer12} \textsc{Mai, J.-F.} and \textsc{Scherer, M.} (2012). H-extendible copulas. \textit{Journal of Multivariate Analysis} \textbf{110} 151--160.
\bibitem{maischerer13} \textsc{Mai, J.-F.} and \textsc{Scherer, M.} (2014). Characterization of extendible distributions with exponential minima via processes that are infinitely divisible with respect to time. \textit{Extremes} \textbf{17} 77--95.
\bibitem{mai12} \textsc{Mai, J.-F.} and \textsc{Scherer, M.} (2017). \textit{Simulating copulas, 2nd edition}. 
{World Scientific Publishing}, Singapore.
\bibitem{maischerer18} \textsc{Mai, J.-F.} and \textsc{Scherer, M.} (2019). Subordinators which are infinitely divisible w.r.t.\ time: construction, properties, and simulation of max-stable sequences and infinitely divisible laws. \textit{ALEA: Latin American Journal of Probability and Mathematical Statistics} \textbf{16:2} 977--1005.
\bibitem{mai13} \textsc{Mai, J.-F.} and \textsc{Scherer, M.} and \textsc{Shenkman, N.} (2013). Multivariate geometric laws, (logarithmically) monotone sequences, and infinitely divisible laws, \textit{Journal of Multivariate Analysis} \textbf{115} 457--480.
\bibitem{mansuy05} \textsc{Mansuy, R.} (2005). On processes which are infinitely divisible with respect to time. \textit{Working paper, https://arxiv.org/abs/math/0504408}.
\bibitem{marshall67} \textsc{Marshall, A.W.} and \textsc{Olkin, I.} (1967).  A multivariate exponential distribution, \textit{Journal of the American Statistical Association} \textbf{62} 30--44.
\bibitem{marshall79} \textsc{Marshall, A.W.} and \textsc{Olkin, I.} (1979). \textit{Inequalities: theory of majorization and its applications}. 
{Academic Press, New York}.
\bibitem{mcneil05} \textsc{McNeil, A.J.} and \textsc{Frey, R.} and \textsc{Embrechts, P.} (2005). \textit{Quantitative risk management}. {Princeton University Press}, Princeton.
\bibitem{mcneil08} \textsc{McNeil, A.J.} (2008). Sampling nested Archimedean copulas. \textit{Journal of Statistical Computation and Simulation} \textbf{78} 567--581.
\bibitem{mcneil09} \textsc{McNeil, A.J.} and \textsc{Ne\v{s}lehov\'{a}, J.} (2009). Multivariate Archimedean copulas, $d$-monotone functions and $l_{1}$-norm symmetric distributions. \textit{Annals of Statistics} \textbf{37:5B} 3059--3097.
\bibitem{mcneil10} \textsc{McNeil, A.J.} and \textsc{Ne\v{s}lehov\'{a}, J.} (2010). From Archimedean to Liouville copulas. \textit{Journal of Multivariate Analysis} \textbf{101} 1772--1790.
\bibitem{molchanov08} \textsc{Molchanov, I.} (2008). Convex geometry of max-stable distributions. \textit{Extremes} \textbf{11:3} 235--259.
\bibitem{mueller02} \textsc{M\"uller, A.} and \textsc{Stoyan, D.} (2002). \textit{Comparison methods for stochastic models and risks}. 
John Wiley and Sons, Chichester (2002).
\bibitem{papangelou89} \textsc{Papangelou, F.} (1989). On the Gaussian fluctuations of the critical Curie-Weiss model in statistical mechanics. \textit{Probability Theory and Related Fields} \textbf{83} 265--278.
\bibitem{pestman09} \textsc{Pestman, W.R.} (2009). \textit{Mathematical Statistics, 2nd edition}. De Gruyter, Berlin.
\bibitem{puccetti14} \textsc{Puccetti, G.} and \textsc{Wang, R.} (2015). Extremal dependence concepts. \textit{Statistical Science} \textbf{30:4} 485--517.
\bibitem{rachev91} \textsc{Rachev, S.T.} and \textsc{R\"uschendorf, L.} (1991). Approximate independence of distributions on spheres and their stability properties. \textit{Annals of Probability} \textbf{19} 1311--1337.
\bibitem{resnick87} \textsc{Resnick, S.I.} (1987). \textit{Extreme values, regular variation and point processes}. 
{Springer-Verlag}, Berlin.
\bibitem{ressel85} \textsc{Ressel, P.} (1985). de Finetti type theorems: an analytical approach. \textit{Annals of Probability} \textbf{13} 898--922.
\bibitem{ryll57} \textsc{Ryll-Nardzewski, C.} (1957). On stationary sequences of random variables and the de Finetti equivalence. \textit{Colloquium Mathematicum} \textbf{4} 149--156.
\bibitem{sato99} \textsc{Sato, K.-I.} (1999). \textit{L\'evy processes and infinitely divisible distributions}. 
{Cambridge University Press, Cambridge}.
\bibitem{scarsini85} \textsc{Scarsini, M.} (1985). Lower bounds for the distribution function of a $k$-dimensional $n$-extendible exchangeable process. \textit{Statistics and Probability Letters} \textbf{3} 57--62.
\bibitem{schilling10} \textsc{Schilling, R.} and \textsc{Song, R.} and \textsc{Vondracek, Z.} (2010). \textit{Bernstein functions}. 
{De Gruyter}, Berlin.
\bibitem{schoenberg38} \textsc{Schoenberg, I.J.} (1938). Metric spaces and positive definite functions. \textit{Transactions of the American Mathematical Society} \textbf{44} 522--536.
\bibitem{shaked77} \textsc{Shaked, M.} (1977). A concept of positive dependence for exchangeable random variables. \textit{Annals of Statistics} \textbf{5} 505--515.
\bibitem{shaked02} \textsc{Shaked, M.} and \textsc{Spizzichino, F.} and \textsc{Suter, F.} (2002). Nonhomogeneous birth processes and $\ell_{\infty}$-spherical densities, with applications in reliability theory. \textit{Probability in the Engineering and Informational Sciences} \textbf{16} 271--288.
\bibitem{sibley71} \textsc{Sibley, D.A.} (1971). A metric for weak convergence of distribution functions. \textit{Rocky Mountain Journal of Mathematics} \textbf{1:3} 427--430.
\bibitem{sklar59} \textsc{Sklar, A.} (1959). Fonctions de r\'epartition \`a $n$ dimensions et leurs marges. \textit{Publ. Inst. Statist. Univ. Paris} \textbf{8} 229--231.
\bibitem{sloot20} \textsc{Sloot, H.} (2020). The deFinetti representation of generalised Marshall--Olkin sequences. \textit{Dependence Modeling} \textbf{8:1} 107--118.
\bibitem{spizzichino82} \textsc{Spizzichino, F.} (1982). \textit{Extendibility of symmetric probability distributions and related bounds}. 
{in Exchangeability in Probability and Statistics, edited by G.\ Koch and F.\ Spizzichino, North-Holland Publishing Company}, 313--320.
\bibitem{steutel03} \textsc{Steutel, F.W.} and \textsc{van Harn, K.} (2003). \textit{Infinite divisibility of probability distributions on the real line}. 
{CRC Press, Boca Raton}.
\bibitem{taleb20} \textsc{Taleb, N.N.} (2020). \textit{Statistical Consequences of Fat Tails}. STEM Academic Press.
\bibitem{williamson56} \textsc{Williamson, R.E.} (1956). Multiply monotone functions and their Laplace transforms. \textit{Duke Mathematical Journal} \textbf{23} 189--207.
\bibitem{zhu16} \textsc{Zhu, W.} and \textsc{Wang, C.-W.} and \textsc{Tan, K.S.} (2016). Structure and estimation of L\'evy subordinated hierarchical Archimedean copulas (LSHAC): theory and empirical tests. \textit{Journal of Banking and Finance} \textbf{69} 20--36.
\end{thebibliography}
\end{document}